\documentclass{article}
\usepackage{graphicx} 
\usepackage{bbm}
\usepackage{mathrsfs}
\usepackage{amsfonts}
\usepackage{tipa}
\usepackage{amssymb}
\usepackage{amsmath}
\usepackage{amsthm}
\usepackage{amstext}
\usepackage{amsopn}
\usepackage{texdraw}
\usepackage{yhmath}
\usepackage{indentfirst}
\usepackage{mathtools}
\usepackage{enumitem}
\usepackage{pgfplots}
\usepackage{chngpage}
\usepackage{tikz}
\usepackage{epsfig}
\usepackage{verbatim}
\usepackage{cases}
\usepackage{appendix}
\usepackage{hyperref}
\usepackage{booktabs}
\usepackage{makecell}
\usepackage{subcaption}
\usepackage{xcolor}
\hypersetup{hypertex=true,
	colorlinks=true,
	linkcolor=blue,
	anchorcolor=blue,
	citecolor=blue,
	urlcolor=black}
\usetikzlibrary{arrows}
\usetikzlibrary{patterns}
\newtheorem{theorem}{Theorem}[section]
\newtheorem{lemma}[theorem]{Lemma}

\newtheorem{proposition}[theorem]{Proposition}
\newtheorem{corollary}[theorem]{Corollary}
\theoremstyle{definition}
\newtheorem{definition}[theorem]{Definition}

\theoremstyle{remark}
\newtheorem{remark}[theorem]{Remark}
\numberwithin{equation}{section}

\newcommand{\ep}{\epsilon}
\newcommand{\la}{\lambda}

\newcommand{\R}{{\mathbf R}}
\newcommand{\C}{{\mathbf C}}

\newcommand{\Po}{\partial \Omega}

\newcommand{\pa}{\partial}
\newcommand{\al}{\alpha}

\newcommand{\ds}{\displaystyle}

\newcommand\wh{\widehat }
\newcommand\wt{\widetilde }

\newcommand\px{\frac{\pa X}{\pa\omega^2}}
\newcommand\py{\frac{\pa Y}{\pa\omega^2}}
\newcommand\ol{\overline }
\newcommand\ul{\underline }

\begin{document}
	\title{Stability and bifurcations of the synchronized layered steady states in membrane-coupled delayed RD systems: The SLEP approach}
	\author{Qidong Wu, Fengqi Yi\footnote{Corresponding Author (F. Yi). Email: yi@dlut.edu.cn}\footnote{The authors were partially supported by National Key R\;$\&$\;D Program of China (2023YFA1009200), the National Natural Science Foundation of China (11971088, 12371160, 12271144) and the Fundamental Research Funds for the Central Universities (DUT25LAB303, DUT23LAB304, DUT18RC(3)070). } \\
		\\
		{\small School of Mathematical Sciences, Dalian University of Technology, Dalian, Liaoning, 116024, China.\hfill{\ }}
	}
	
	\date{}
	\maketitle	
	
	\begin{abstract}
		Inter-layer exchange in biological and chemical systems may involve transport delays and memory effects. In this paper, we mainly study the stability and bifurcations of the synchronized layered steady states in singularly perturbed membrane-coupled reaction–diffusion (RD) systems, where discrete or exponentially distributed delays are introduced into the activator coupling, while the inhibitor coupling remains instantaneous, leading to six representative configurations. By extending the singular limit eigenvalue problem (SLEP) method originally developed by [Nishiura and Fujii, SIMA, 18(6), 1987; Nishiura and Mimura, SIAP, 49(2), 1989] for the non-delayed case, we develop a unified spectral framework that incorporates the symmetric and asymmetric sub-eigenvalue problems arising from all six delayed configurations, demonstrating how the SLEP approach can be adapted to capture the limiting behavior of critical eigenvalues when time delays are present. We then establish sufficient conditions to determine the delay-independent stability, the persistence of Turing instability, as well as delay-induced Hopf bifurcations of the synchronized layered steady states. In particular, we focus on analyzing how time delays could induce in-phase and anti-phase breathers via Hopf bifurcations, leading to spatiotemporal oscillatory instabilities that were absent in the delay-free setting.  As an application, we also supply a two-layer membrane-coupled delayed FitzHugh–Nagumo system to illustrate how our results can be applied.
		
	\end{abstract}

	\section{Introduction}
	Pattern formation in reaction-diffusion (RD) systems, first elucidated by Turing's 1952 paper on morphogenesis \cite{Tu}, provides the fundamental mathematical framework for understanding spatiotemporal pattern emergence in biological, chemical, and physical settings (\cite{CDBD,LE1,NT,SY0,SY1,SY2,SY3,WY1}). Among the diverse array of RD models proposed in the literature, the following two-component singularly perturbed system has garnered considerable attention in the research community (see for instance \cite{MTH, NF, NM, WY2}):
	\begin{equation}\label{eq1}
		\begin{cases}
			\begin{aligned}
				\alpha\dfrac{\partial u}{\partial t}&=\ep\Delta u+\dfrac{1}{\ep}f(u,v),&&\text{in}\;\Omega\times(0,\infty),\\
				\dfrac{\partial v}{\partial t}&=d\Delta v+g(u,v),&&\text{in}\;\Omega\times(0,\infty),\\
				\partial_\nu u&=\partial_\nu v=0,&&\text{on}\;\partial\Omega\times(0,\infty),
			\end{aligned}
		\end{cases}
	\end{equation}
	where $u$ and $v$ denote the concentrations of the activator and inhibitor at time $t$ and position $x\in\Omega$, respectively; $\Omega$ is a bounded open domain in $\mathbb{R}^N$ with smooth boundary $\partial\Omega$; $\nu$ represents the unit outward normal vector on $\partial\Omega$, with $\partial_\nu u:=\nabla u(x)\cdot\nu$ for $x\in\partial\Omega$ (and analogously for $\partial_\nu v$); $0<\ep\ll 1$ is a sufficiently small parameter (relative to $d$) that quantifies the strength of lateral inhibition, and $\alpha$ is the relaxation coefficient governing the relative reaction rates of $u$ and $v$. The reaction terms $f(u,v)$ and $g(u,v)$ are assumed to satisfy the structural hypotheses $(H_1)$-$(H_6)$, which are given in Section 2. System \eqref{eq1} typically governs the dynamics in the biologically relevant regime where the activator $u$ possesses a sharp internal transition layer due to slow diffusion $(\ep \ll d)$ yet undergoes rapid kinetics $(\alpha \ll 1/\ep)$, whereas the inhibitor $v$ spreads diffusively with sluggish reaction dynamics. The foundational work of Mimura, Tabata, and Hosono \cite{MTH} established the existence of layered steady states of system \eqref{eq1} with sharp internal interfaces. Subsequently, Nishiura and Fujii \cite{NF}  developed the singular limit eigenvalue problem (SLEP) method, reducing infinite-dimensional eigenvalue problems to algebraically tractable scalar equations in the singular limit $\ep\to0^+$. This powerful framework enabled Nishiura and Mimura \cite{NM} to reveal a striking stability reversal phenomenon: as the relaxation parameter $\alpha$ decreases below a critical threshold, the steady state loses stability through a Hopf bifurcation, giving rise to time-periodic oscillatory solutions.
	
	While isolated RD systems have been extensively studied, many biological and chemical processes inherently involve coupled multi-layer architectures, such as stratified epithelial tissues, multi-compartment chemical reactors, and layered neural networks \cite{Bar,KS,LE2,LWO,YDZE,YE}. Such coupling can generate symmetry-breaking rhythms \cite{AEKV}, square and oscillatory Turing patterns \cite{LWO,YE}, and spatiotemporal antiresonance \cite{PPR}. For the singularly perturbed coupled systems, the pioneering theoretical study was conducted by Takaishi, Mimura, and Nishiura \cite{TMN}. The governing 4-component system takes the form:
	\begin{equation}\label{eq2}
		\begin{cases}
			\begin{split}
				\alpha\dfrac{\partial u_i}{\partial t}&=\ep\Delta u_i+\dfrac{1}{\ep}f(u_i,v_i)+\delta k_1\left(u_j-u_i\right),&&\text{in}\;\Omega\times(0,\infty),\\
				\dfrac{\partial v_i}{\partial t}&=d\Delta v_i+g(u_i,v_i)+\delta k_2(v_j-v_i),&&\text{in}\;\Omega\times(0,\infty),\\
				\partial_\nu u_i&=\partial_\nu v_i=0,&&\text{on}\;\Po\times(0,\infty),
			\end{split}
		\end{cases}
	\end{equation}
	where $u_i$ and $v_i$ denote the concentrations of the activator and inhibitor in reactor $i\in\{1,2\}$, respectively; $j\in\{1,2\}\backslash\{i\}$; $\epsilon$ and $d$ are the intra-layer diffusion coefficients; $k_1$ and $k_2$ are the inter-layer diffusion coefficients; and $\delta$ is the inter-layer coupling strength \cite{LE2}. In the special case when $k_1=0$ and $k_2>0$, Takaishi, Mimura, and Nishiura \cite{TMN} studied Turing instabilities of synchronized layered steady states (analytically) and those of in-phase breathers (numerically), but left the open problem of analyzing Turing instability of these solutions in the extended case when $k_1>0$ and $k_2>0$. This problem was recently studied in \cite{WY2} through the SLEP method (\cite{NF, NM}), Floquet theory, and regular perturbation arguments (\cite{Ma, Yi}).
	
	However, these studies assume that inter-layer couplings are instantaneous. In realistic biological, chemical, and physical settings, finite signal propagation speeds, transport lags, and memory effects in regulatory feedback introduce non-negligible time delays into the inter-layer exchange, thereby qualitatively modifying the stability and spatiotemporal dynamics of the coupled system. In particular, time delay can be incorporated into the inter-layer exchange terms through \emph{transmission} parts, acting on signals received from neighboring units, or \emph{self-feedback} parts, acting on the local-state contribution to the coupling term \cite{CZL,ST}.
	
	Regarding time delays involved in the inter-layer coupling mechanism, one usually examines how these delays are introduced from the following two aspects: 
	\begin{enumerate}
		\item[(i)] Temporal nature of delays:
		\begin{enumerate}
			\item Discrete delay \cite{RF,ST,ZSKK}. It represents the temporal lag by one or more prescribed delay times, so that the current dynamics depend on the state at specific past instants, such as $t-\tau$ with $\tau>0$.
			
			\item Distributed delay \cite{JL,LYGYW,ZLLX}. It models the lag through a weighted integral over the past history. Rather than selecting a single past instant, the distributed-delay term incorporates a continuum of past states according to a prescribed delay kernel. A commonly used family of delay kernels is given by
			\begin{equation}\label{gn}
				g_n(t)=\frac{1}{\Gamma(n)}\left(\frac{n}{\tau}\right)^{n} t^{n-1} e^{-\frac{nt}{\tau}},\; t\ge 0,
			\end{equation}
			where $\tau>0$ is the delay parameter; $n\in\{1,2,\cdots\}$ is the shape parameter \cite{S}.
		\end{enumerate}
		\item[(ii)] Placement of delays within the inter-layer coupling:
		\begin{enumerate}
			\item Transmission-and-self-feedback delayed coupling \cite{CZL}: 
			both the cross-layer transmission term and the local self-feedback term are delayed;
			
			\item Transmission-delayed coupling \cite{JL}: 
			the cross-layer transmission term is delayed, while the local self-feedback term remains instantaneous;
			
			\item Self-feedback-delayed coupling \cite{ZSKK}: 
			the local self-feedback term is delayed, whereas the cross-layer transmission term remains instantaneous.
		\end{enumerate}
		For brevity, we shall refer to these three configurations as the {\bf fully delayed}, {\bf cross-delayed}, and {\bf self-delayed} couplings, respectively.
	\end{enumerate}
	
	Time-delayed coupling has received considerable attention in coupled oscillators, RD systems, neural networks, and complex dynamical networks \cite{CZL,JL,RF,RSJ,ZLLX}. The introduction of delays can profoundly influence collective dynamics, leading to phenomena such as amplitude death \cite{Atay,RSJ}, the revival of suppressed oscillations \cite{ZSKK}, phase locking and synchronization \cite{ST}, as well as delay-induced oscillatory instabilities \cite{RF}. In particular, as an important inspiration of present paper, Ji and Li \cite{JL} investigated Turing pattern formation in a two-layer Lengyel--Epstein model with distributed delay in the cross-layer coupling. The authors showed numerically that weak coupling accelerates Turing pattern formation and enhances its spatial amplitude, whereas strong coupling prolongs the process or even inhibits the pattern. That is, the cross-delayed coupling can influence the formation, amplitude, and even persistence of Turing patterns. These studies demonstrate that both the temporal nature of the delay and its position within the coupling mechanism can play an essential role in determining the collective dynamics. 
	
	Nevertheless, to the best of our knowledge, a systematic study that treats the temporal type and placement of delayed coupling within a unified spectral framework remains lacking for singularly perturbed membrane-coupled systems. A major difficulty lies in the spectral analysis. The SLEP method, originally developed by Nishiura and Fujii \cite{NF}, Nishiura and Mimura \cite{NM}, provides a powerful framework for the spectral stability analysis of singularly perturbed layered solutions. However, the introduction of delay operators makes the SLEP systems markedly more complicated than their delay-free counterparts considered in \cite{NM,TMN,WY2}. This naturally raises several questions: How can the SLEP method be extended to delay-coupled systems? What roles do the inter-layer coupling strengths and time delays play in determining stability/instability and pattern formation? What general stability and mode-selection principles can be identified across different delayed coupling configurations?
	
	Motivated by these questions and building on the above taxonomy, we shall consider the following general four-component two-layer RD system:
	\begin{equation}\label{eq-oo}
		\begin{cases}
			\begin{aligned}
				\alpha\dfrac{\partial u_i}{\partial t}&=\ep\Delta u_i+\dfrac{1}{\ep}f(u_i,v_i)+\delta k_1\left(L_1[u_j]-L_2[u_i]\right),&&\text{in}\;\Omega\times(0,\infty),\\
				\dfrac{\partial v_i}{\partial t}&=d\Delta v_i+g(u_i,v_i)+\delta k_2(v_j-v_i),&&\text{in}\;\Omega\times(0,\infty),\\
				\partial_\nu u_i&=\partial_\nu v_i=0,&&\text{on}\;\pa\Omega\times(0,\infty),
			\end{aligned}
		\end{cases}
	\end{equation}
	where $\Omega:=(0,1)$, $i\in\{1,2\}$, $j\in\{1,2\}\backslash\{i\}$; $L_1$ and $L_2$ denote the cross-layer transmission operator and the local self-feedback operator in activator coupling, respectively, which are specified as follows.
	\begin{enumerate}
		\item[(i)]  For discrete delay, we shall consider the following 3 cases:
		\begin{enumerate}[leftmargin=4.5em]
			\item [Case 1] (fully discrete-delayed coupling): $L_1[u]=L_2[u]:=u(x,t-\tau)$.
			\item [Case 2] (cross-discrete-delayed coupling): $L_1[u]:=u(x,t-\tau),\;L_2[u]:=u$. 
			\item [Case 3] (self-discrete-delayed coupling): $L_1[u]:=u,\;L_2[u]:=u(x,t-\tau)$.
		\end{enumerate}
		\item[(ii)] For distributed delay, we shall consider the following 3 cases:
		\begin{enumerate}[leftmargin=4.9em]
			\item[Case 4] (fully distributed-delayed coupling): $L_1[u]=L_2[u]:=\int_{-\infty}^{t}\frac{1}{\tau} e^{-\frac{t-s}{\tau} }u(x,s)ds$.
			\item [Case 5] (cross-distributed-delayed coupling): $L_1[u]:=\int_{-\infty}^{t}\frac{1}{\tau} e^{-\frac{t-s}{\tau} }u(x,s)ds,\;L_2[u]:=u$. 
			\item [Case 6] (self-distributed-delayed coupling): $L_1[u]:=u, \;L_2[u]:=\int_{-\infty}^{t}\frac{1}{\tau} e^{-\frac{t-s}{\tau} }u(x,s)ds$. 
		\end{enumerate}
	\end{enumerate}
	
	In particular, for Cases 1--3, the couplings will be directly reduced to the delay-free case if we send $\tau=0$. For Cases 4--6, the distribution kernel $g_1(s)=\frac{1}{\tau}e^{-s/\tau}$ in $s\in[0,\infty)$ is well known as the weak kernel. It is exponentially decreasing in $s$ to simulate the fading-memory effect with $\int_{0}^{\infty}\frac{1}{\tau}e^{-s/\tau}ds=1$, where $\tau>0$ denotes the mean delay. Thus, the delay-free case is understood in the limiting sense $\tau\to0^+$ (\cite{S}), when system \eqref{eq-oo} reduces to system \eqref{eq2}.
	
	\begin{definition}
		\begin{enumerate}
			\item	If $U(x,\ep):=(u(x,\ep), v(x,\ep))$ is the layered steady state solution of system \eqref{eq1}, then $(U(x,\ep), U(x,\ep)):=(u(x,\ep), v(x,\ep),u(x,\ep), v(x,\ep))$ is called a {\bf synchronized layered steady state} of the delayed membrane-coupled RD system \eqref{eq-oo}. 
			\item If $(U_{1,p}(x,t,\ep),U_{2,p}(x,t,\ep)):=(u_{1,p}(x,t,\ep),v_{1,p}(x,t,\ep),u_{2,p}(x,t,\ep),v_{2,p}(x,t,\ep))$
			is the periodic solution of the delayed membrane-coupled RD system \eqref{eq-oo} bifurcating from $(U(x,\ep), U(x,\ep))$ with minimal positive period $T$, then it is called an {\bf in-phase breather} if $U_{1,p}(x,t,\ep)=U_{2,p}(x,t,\ep)$, while an {\bf anti-phase breather} if $U_{1,p}(x,t,\ep)=U_{2,p}(x,t+\frac{T}{2},\ep)$, for any $x\in\Omega$ and $t>0$.
		\end{enumerate}
	\end{definition} 
	
	Our primary aim of the present paper is to understand how time delays in inter-layer coupling affect the stability of the synchronized layered steady state $(U(x,\ep),U(x,\ep))$ and how they can induce in-phase or anti-phase breathers through Hopf bifurcations. A key feature of our analysis is that, despite the different forms of discrete and distributed delays, the eigenvalue problems for all six cases can be treated within a unified SLEP framework and decomposed into symmetric and asymmetric sub-problems. These reduced eigenvalue problems exhibit closely related algebraic structures, while their detailed analysis requires different arguments depending on the type and placement of delays. It turns out that the presence of delay operators considerably complicates the characteristic equations by introducing the transcendental factor $e^{-\lambda\tau}$ for discrete delays or the rational factor $(1+\lambda\tau)^{-1}$ for exponentially distributed delays, making the existence of purely imaginary eigenvalues and the verification of transversality conditions substantially more challenging than in the delay-free setting. Nevertheless, by exploiting the specific properties of different delay mechanisms, we establish explicit conditions for Hopf bifurcation and identify several general structural conclusions shared by the six coupling configurations.
	
	To describe the main results, we first note that $\delta>0$ can be absorbed into the coupling coefficients $k_1$ and $k_2$, thus we set $\delta=1$ throughout the paper without loss of generality. We shall also fix $\alpha>\alpha_0$ so that $U(x,\ep)$ is the stable layered steady state solution of the one-reactor based system \eqref{eq1}. Under this condition, in the absence of time delays, the stability of the synchronized layered steady state $(U(x,\ep),U(x,\ep))$ with respect to the delay-free coupled system \eqref{eq2} is well known. More precisely, by Theorem \ref{th.2-1}, for sufficiently small $\ep>0$, $(U(x,\ep),U(x,\ep))$ is unstable with respect to the delay-free coupled system \eqref{eq2} when $(k_1,k_2)\in\Gamma_1$, while stable when $(k_1,k_2)\in\Gamma_2$. In particular, the critical curve $\{(\xi(k_2),k_2):k_2\in(k_2^*,\infty)\}$  separates the Turing instability region $\Gamma_1$ from the stability region $\Gamma_2$. As we will see later, this critical curve also plays an important role in the dynamics of the delayed coupled RD system \eqref{eq-oo}.
	
	When time delays are introduced into the membrane-coupled system, our main results are summarized in the following way  (see Theorems \ref{th.mode} and \ref{th.assembly}):	
	\begin{enumerate}
		\item {\bf Delay-independent stability for $k_2\in(0,k_2^*)$ and small $k_1$.} In any of Cases 1--6, when $k_2\in(0,k_2^*)$ and $k_1>0$ is sufficiently small ($(k_1,k_2)\in\Gamma_2$), the synchronized layered steady state $(U(x,\ep),U(x,\ep))$ remains stable with respect to \eqref{eq-oo} for every fixed $\tau\in(0,\infty)$ and sufficiently small $\ep>0$. Thus, sufficiently weak inter-layer coupling cannot generate delay-induced instability, regardless of the type of delay (discrete or distributed) and the placement within the coupling operator (full, cross, or self). Certainly, the specific thresholds depend on the individual case. This common stability property is established by analyzing the corresponding reduced characteristic equations, which possess closely related algebraic structures after the frequency and delay parameters are separated. Surprisingly, the asymmetric sub-problems of Cases 2, 3, and 4 share the same delay-independent stability threshold in the singular limit $\ep\to0^+$, despite the different types and placements of delays (see Lemma \ref{le.e3} and Remark \ref{re.K234}). It reveals a notable connection among seemingly different delayed coupling mechanisms. 
		
		\item {\bf Delay-independent Turing instability for $k_2\in(k_2^*,\infty)$ and small $k_1$.} In any of Cases 1--6, when $(k_1,k_2)\in\Gamma_1$, or say $k_2\in(k_2^*,\infty)$ and $k_1\in(0,\xi(k_2))$, the synchronized layered steady state is already Turing unstable in the delay-free system \cite{WY2}. We show that the introduction of time delays cannot reverse this diffusion-driven instability: for every $\tau>0$, the corresponding eigenvalue problem maintains an eigenvalue in the open right half-plane. Consequently, $(U(x,\ep),U(x,\ep))$ remains Turing unstable with respect to \eqref{eq-oo} for any fixed $\tau\in(0,\infty)$ and sufficiently small $\ep>0$. Hence, the Turing instability generated by strong inhibitor coupling is robust with respect to the introduction of time delays. In particular, the eigenvalue problem may admit purely imaginary eigenvalues for suitable parameter values in Cases 3 and 6, leading to the delicate interaction of Turing and Hopf instabilities. It is numerically explored in Section \ref{NS}.  
		
		\item {\bf Delay-induced Hopf bifurcation for $k_2\in(k_2^*,\infty)$ and intermediate $k_1$.} In any of Cases 1--6, for $k_2\in(k_2^*,\infty)$, once $k_1$ surpasses the delay-free Turing threshold $\xi(k_2)$ ($(k_1,k_2)\in\Gamma_2$), $(U(x,\ep),U(x,\ep))$ is stable with respect to \eqref{eq-oo} in the limiting sense as $\tau\to0^+$ for sufficiently small $\ep>0$. However, this stability can be destroyed by time delays via a Hopf mechanism. More precisely, for Case $j$ ($j\in\{1,\cdots,6\}$), there exists a threshold $H_j(k_2)>\xi(k_2)$, such that for $k_2\in(k_2^*,\infty)$,  $k_1\in(\xi(k_2),H_j(k_2))$, and sufficiently small $\ep>0$, the asymmetric sub-problem admits a simple conjugate pair of eigenvalues crossing the imaginary axis transversely from left to right at a positive critical delay. Consequently, the synchronized layered steady state becomes unstable as delay increases. In Cases 1, 4, and 5, the symmetric sub-problem remains stable, so the first asymmetric crossing yields a primary Hopf bifurcation and a family of anti-phase breathers. In Cases 2, 3, and 6, both the symmetric and asymmetric sub-problems may undergo Hopf instabilities, where the first destabilizing mode can be determined under some additional conditions (see Remark \ref{re.pri}). Moreover, if the first destabilizing mode is symmetric, then the primary Hopf bifurcation triggers a family of in-phase breathers.
		
		\item {\bf Pattern selection and stabilization for arbitrary $k_2\in(0,\infty)$} and large $k_1$. A sharper distinction among the six delayed cases emerges when $k_1$ is sufficiently large. Recall that strong activator coupling has a stabilizing effect in the delay-free system \cite{WY2}. In contrast, delayed coupling can either induce instability via Hopf bifurcation or preserve stability of the synchronized layered steady state $(U(x,\ep),U(x,\ep))$. More precisely, 
		\begin{enumerate}
			\item For the fully delayed configurations (Cases 1 and 4), Hopf bifurcations arise exclusively from the asymmetric sub-problem, leading to anti-phase breathers (Theorems \ref{th-1} and \ref{th-4}).
			\item For the self-delayed configurations (Cases 3 and 6), the first Hopf bifurcation is triggered by the symmetric sub-problem, generating in-phase breathers. In particular, although both symmetric and asymmetric sub-problems may become critical in Case 3, the first critical delay is determined by the symmetric sub-problem (Theorems \ref{th-3} and \ref{th-6}).
			\item For the cross-discrete-delayed configuration (Case 2), in contrast to Case 3, the first critical delay comes from the asymmetric  sub-problem, inducing anti-phase breathers. Remarkably, the cross-distributed-delayed configuration (Case 5) behaves like the delay-free case: once $k_1$ exceeds a suitable threshold, $(U(x,\ep),U(x,\ep))$ remains stable for every fixed $\tau>0$ and sufficiently small $\ep>0$, leading to strong-coupling-induced delay-independent stability (Theorems \ref{th-2} and \ref{th-5}).
		\end{enumerate}
	\end{enumerate}
	
	Our results reveal that the type and placement of delays jointly determine the dominant oscillatory modes and the stability of the synchronized layered steady state. In particular, fully delayed and self-delayed coupling exhibit similar mode-selection mechanisms across the discrete and distributed settings, while cross-delayed coupling exhibits a significant difference between the two delay types. As an application, the two-layer membrane-coupled FitzHugh-Nagumo delayed model is employed to illustrate and complement our analytical conclusions. Finally, we remark that the present work introduces time delays only in the inter-layer activator coupling, while the inhibitor coupling remains instantaneous. This provides a benchmark toward a more general theory of delayed membrane coupling. Introducing delays simultaneously into both coupling channels would result in more complicated SLEP characteristic equations and interacting delay terms. It requires the new analytical tools or additional structural assumptions and is thus left for future study. 
	
	The remainder of this paper is organized as follows. Section 2 reviews preliminary results on the uncoupled subsystem and summarizes the stability benchmark of delay-free coupled system. Section 3 constitutes the core analysis of the paper: we first establish a unified framework for the symmetric and asymmetric sub-problems of Cases 1--6, and then state the stability criteria and Hopf bifurcation results for each case. In Section 4, numerical simulations are performed to illustrate and complement the theoretical results, including the interaction of Turing instability and delay-induced oscillation. Finally, Section 5 compares the six delay mechanisms and discusses their biological implications and several open problems. Technical lemmas are collected in the Appendix.
	
	\section{Preliminaries: the SLEP method in non-delayed systems}
	\subsection{The non-delayed uncoupled subsystem}
	We first consider the following non-delayed uncoupled subsystem
	\begin{equation}\label{eq-s}
		\begin{cases}
			\begin{split}
				\alpha\frac{\partial u}{\partial t}&=\ep\Delta u+\frac{1}{\ep}f(u,v), ~&\text{in}\;\Omega\times(0,\infty),\\
				\dfrac{\partial v}{\partial t}&=d\Delta v+g(u,v),~&\text{in}\;\Omega\times(0,\infty),\\
				\partial_\nu u&=\partial_\nu v=0,~&\text{on}\;\Po\times(0,\infty),
			\end{split}
		\end{cases}
	\end{equation}
	where $\Omega=(0,1)$ throughout this paper. Moreover, $f(u,v)$ and $g(u,v)$ are always assumed to satisfy the conditions $(H_1)$-$(H_6)$ (see \cite{MTH, NF, TMN, WY2}):
	\begin{enumerate}[leftmargin=4.5em]
		\item[($H_1$).] $f$ and $g$ are smooth functions of $u$ and $v$ defined on some open set in $\R^2$.
		\item[($H_2$).] The nullcline $f(u,v)=0$ is $S$-shaped, comprising three continuous branches $u = h_{-}(v)$, $h_0(v)$, and $h_{+}(v)$ defined on intervals $I_{-}$, $I_0$, and $I_{+}$, respectively. Setting $\min I_{-} = \underline{v}$ and $\max I_{+} = \bar{v}$, the strict ordering $h_{-}(v) < h_0(v) < h_{+}(v)$ holds for all $v \in I^* = (\underline{v}, \bar{v})$. Moreover, $h_{+}(v)$ meets $h_0(v)$ only at $v = \bar{v}$, while $h_{-}(v)$ meets $h_0(v)$ only at $v = \underline{v}$.
		
		\item[($H_3$).] The integral $J(v) = \int_{h_{-}(v)}^{h_{+}(v)} f(s,v)\,ds$ possesses an isolated root $v^* \in I^*$ with $J'(v^*) < 0$.
		
		\item[($H_4$).] The nullclines of $f(u,v)=0$ and  $g(u,v)=0$ intersect transversally at either one or three points. Any equilibrium on $u = h_{-}(v)$, $h_{+}(v)$, or $h_0(v)$ is denoted by $(u_{-},v_{-}) = (h_{-}(v_{-}),v_{-})$, $(u_{+},v_{+}) = (h_{+}(v_{+}),v_{+})$, or $R = (u_0,v_0) = (h_0(v_0),v_0)$, respectively. We assume $v_{-} < v^* < v_{+}$, where $v_-$ (resp., $v_+$) is replaced by $\ul v$ (resp., $\ol v$) if no equilibrium lies on the branch $h_-(v)$ (resp., $h_+(v)$).
		
		\item[($H_5$).] The partial derivative satisfies $f_u < 0$ on $\mathcal{R}_{-} \cup \mathcal{R}_{+}$, where
		\begin{equation*}
			\mathcal{R}_{-} := \{(u,v) : u = h_{-}(v),\, v_{-} < v \leq v^*\}, \;
			\mathcal{R}_{+} := \{(u,v) : u = h_{+}(v),\, v^* \leq v < v_{+}\}.
		\end{equation*}
		
		\item[($H_6$).]  $g|_{\mathcal{R}_{-}} < 0 < g|_{\mathcal{R}_{+}}$, $\left. (f_u g_v - f_v g_u)/f_u \right|_{\mathcal{R}_{\pm}}<0$, $\left. g_v \right|_{\mathcal{R}_{\pm}} \leq 0$.
	\end{enumerate}
	
	\begin{theorem}[\cite{MTH, NF, NM}]\label{th.1}  Suppose that $(H_1)$-$(H_6)$ hold. Then, the following conclusions hold true:
		\begin{enumerate}			
			\item There exist positive constants $d_0>0$ and $\ep_0>0$, such that for any $0<\ep<\ep_0$ and $d>d_0$, system \eqref{eq-s} has an $\ep$-family of mono-layer non-constant steady state solutions $U(x,\ep):=(u(x,\ep),v(x,\ep))\in C_{\ep}^{2}([0,1])\times C^{2}([0,1])$, where 
			\begin{equation*}
				C_{\ep}^{2}([0,1]):=\{u\in C^2\left( [0,1]\right): \sum_{k = 0}^2 \max\limits_{x\in[0,1]}\big|\big(\ep d/dx\big)^ku(x)\big|<+\infty\}.
			\end{equation*}
			In particular, for any sufficiently small $\kappa>0$,
			\begin{equation*}
				\begin{split}
					&\lim\limits_{\ep\to 0^+} u(x,\ep)=u^\star(x)\;\text{uniformly on $(0,1)\setminus (x^*-\kappa,x^*+\kappa)$},\\
					&\lim\limits_{\ep\to 0^+} v(x,\ep)=v^\star(x)\;\text{uniformly on $(0,1)$},
				\end{split}
			\end{equation*}
			where $x^*\in(0,1)$ is the layer position of $(u^\star(x), v^\star(x))$, which is the weak solution of the reduced system $($$\ep=0$$)$ of \eqref{eq-s} given by 
			\begin{equation}\label{eq.reduced}
				\begin{cases}
					\begin{split}
						&f(u,v)=0,\;dv_{xx}+g(u,v)=0,&&x\in(0,1),\\
						&u'(0)=u'(1)=v'(0)=v'(1)=0.&&
					\end{split}
				\end{cases}
			\end{equation}
			\item There exists a unique critical value $\alpha_c^\ep > 0$, with $\ep\in(0,\ep_0)$, such that the steady state $(u(x,\ep), v(x,\ep))$ is asymptotically stable when $\alpha > \alpha_c^\ep$, while unstable when $\alpha < \alpha_c^\ep$. Furthermore, system \eqref{eq-s} undergoes a Hopf bifurcation bifurcating from $(u(x,\ep), v(x,\ep))$ at $\alpha = \alpha_c^\ep$.
		\end{enumerate}
	\end{theorem}
	\begin{remark}
		By \cite{MTH, NF, NM}, it is well-known that the eigenvalue problem associated with the linearized system of \eqref{eq-s} evaluated at $U(x,\ep)$ is given by
		\begin{equation}\label{P-L-2-0}
			\begin{pmatrix}
				\ep^2&0\\
				0&d
			\end{pmatrix}
			\begin{pmatrix}
				\phi''(x)\\
				\psi''(x)
			\end{pmatrix}+
			\begin{pmatrix}
				f_u^\ep&f_v^\ep\\
				g_u^\ep&g_v^\ep\\
			\end{pmatrix}\begin{pmatrix}
				\phi(x)\\
				\psi(x)
			\end{pmatrix}=\begin{pmatrix}
				\ep\al\la&0\\
				0&\la
			\end{pmatrix}
			\begin{pmatrix}
				\phi(x)\\
				\psi(x)
			\end{pmatrix},
		\end{equation}
		where $f_u^\ep:=f_u(u(x,\ep),v(x,\ep))$, etc. Throughout this paper, all the systems are subject to the homogeneous Neumann boundary conditions, unless otherwise specified.
	\end{remark} 
	
	\subsection{The non-delayed membrane-coupled RD system}
	Consider the following membrane-coupled RD system of \eqref{eq2} with $\delta=1$, or equivalently
	\begin{equation}\label{eq.d-f}
		\begin{cases}
			\begin{split}
				\alpha\frac{\partial u_i}{\partial t}&=\ep\Delta u_i+\frac{1}{\ep}f(u_i,v_i)+ k_1\left( u_j-u_i\right),~&x\in(0,1),\;t>0,\\
				\dfrac{\partial v_i}{\partial t}&=d\Delta v_i+g(u_i,v_i)+ k_2(v_j-v_i),~&x\in(0,1),\;t>0,
			\end{split}
		\end{cases}
	\end{equation}
	where $i,j\in\{1,2\}$ and $i\neq j$.
	
	We first state a technical conclusion.
	
	\begin{proposition}[\cite{NM,WY2}]\label{prop.AB}
		Let $x^*, \rho_0^*,c_1^*,c_2^*,\gamma_n,\psi_n(x)$ be given by Lemma \ref{le.AA1} in the Appendix, and define
		\begin{equation}\label{AB}
			\begin{split}
				X(\la_R^*,(\la_I^*)^2, K)&:=\sum\limits_{n \ge 0} \dfrac{(\gamma _n + \lambda _R^* + 2     K)c_1^*c_2^*\psi_n ^2(x^*)}{(\gamma _n + \lambda _R^* + 2     K)^2 + (\lambda _I^*)^2},\\
				Y(\la_R^*,(\la_I^*)^2, K)&:=\sum\limits_{n \ge 0} \dfrac{c_1^*c_2^*\psi_n ^2(x^*)}{(\gamma _n + \lambda _R^* +2     K)^2 + (\lambda _I^*)^2},
			\end{split}    
		\end{equation}
		where $\la_R^*>-\mu^*$ for $0<\mu^*<\inf\limits_{x\in(0,1)}\frac{ f_u^\star g_v^\star-f_v^\star g_u^\star}{- f_u^\star}$, $\la_I^*\in\R$, and $K\in[0,\infty)$. Then, it holds true that
		\begin{equation}\label{eq.XY1}
			\begin{split}
				&X(0,0,0)=\sum\limits_{n \ge 0} \dfrac{c_1^*c_2^*}{\gamma _n }\psi_n ^2(x^*)>\rho_0^*,\;\;Y(0,0,0)=\sum\limits_{n \ge 0} \dfrac{c_1^*c_2^*}{\gamma _n^2 }\psi_n ^2(x^*)>0,\\
				&\lim\limits_{  K\to+\infty}X(\la_R^*,(\la_I^*)^2,   K)=\lim\limits_{  K\to+\infty}Y(\la_R^*,(\la_I^*)^2, K)=0,\\
				&\lim\limits_{(\la_I^*)^2\to+\infty}X(\la_R^*,(\la_I^*)^2,   K)=\lim\limits_{(\la_I^*)^2\to+\infty}Y(\la_R^*,(\la_I^*)^2,   K)=0,\\
				&\dfrac{\pa X}{\pa (\la_I^*)^2}(0,(\la_I^*)^2,K)<0,\;\dfrac{\pa Y}{\pa (\la_I^*)^2}(0,(\la_I^*)^2,K)<0,\;
				\dfrac{\pa X}{\pa K}(0,0,K)<0,\;\dfrac{\pa Y}{\pa K}(0,(\la_I^*)^2,K)<0.
			\end{split}
		\end{equation}
		Moreover, for $K\in[0,\infty)$ and $(\la_I^*)^2\in[0,\infty)$, we have
		\begin{equation}\label{eq.XY2}
			0\le \left\vert (\la_I^*)^2\frac{\pa X(0,(\la_I^*)^2,K)}{\pa (\la_I^*)^2} \right\vert \le X(0,(\la_I^*)^2,K),\;0\le \left\vert (\la_I^*)^2\frac{\pa Y(0,(\la_I^*)^2,K)}{\pa (\la_I^*)^2} \right\vert \le Y(0,(\la_I^*)^2,K).
		\end{equation}
	\end{proposition} 
	\begin{proof}
		See the Appendix of \cite{NM} for a detailed proof of \eqref{eq.XY1}. By direct calculations, we have
		\begin{equation*}
			\begin{split}
				0\le -(\la_I^*)^2\frac{\pa X(0,(\la_I^*)^2,K)}{\pa (\la_I^*)^2}=\sum_{n\ge0}\frac{c_1^*c_2^*\psi_n^2(x^*)(\la_I^*)^2(\gamma_n+2K)}{((\gamma_n+2K)^2+(\la_I^*)^2)^2}\le\sum_{n\ge0}\frac{c_1^*c_2^*\psi_n^2(x^*)(\gamma_n+2K)}{(\gamma_n+2K)^2+(\la_I^*)^2}= X(0,(\la_I^*)^2,K),
			\end{split}
		\end{equation*}
		leading to the first argument in \eqref{eq.XY2}. The second estimate follows similarly.
	\end{proof}
	
	The stability of synchronized layered steady state $(U(x,\ep),U(x,\ep))$ with respect to \eqref{eq.d-f} for small $\ep>0$ was investigated in \cite{WY2}.
	
	\begin{figure}[htbp]
		\centering
		\begin{tikzpicture}
			\fill [fill=green!20]  (0,4) --  (0,1) parabola (2.4,4)--cycle;
			\fill [fill=magenta!15]  (0,0) --  (0,1) parabola (2.4,4)--(2.5,4)--(2.5,0)--cycle;
			\fill [fill=magenta!15]  (2.5,0) rectangle (7.5,4);
			\draw[->] (0,0) --(7.55,0) node[right] {$k_1$};
			\draw[->] (0,0) --(0,4.05) node[above] {$k_2$};
			\draw  (0,0) node[left] {$O$};
			\draw[densely dashed] (2.5,4)--(2.5,0) node [below]{$\frac{\rho_0^*}{2}$};
			\draw (0,1) parabola (2.4,4);
			\draw (0,1) node [left] {$k_2^*$};
			\draw [densely dashed] (1.2,1.75)--(1.2,0) node [below] {$\xi(k_2)$};
			\draw [densely dashed] (1.2,1.75)--(0,1.75) node [left]{$k_2$};
			\draw node at (4,2){$\Gamma_2$};
			\draw node at (1.1,2.9){$\Gamma_1$};
		\end{tikzpicture}
		\caption{Stability diagram of the synchronized layered steady state $(U(x,\ep),U(x,\ep))$ with respect to \eqref{eq.d-f}.  $\Gamma_1:=\{(k_1,k_2): k_2\in(k_2^*,+\infty),\; k_1\in(0,\xi(k_2))\}$, $\Gamma_{2}:=\{(k_1,k_2): k_2\in(0,k_2^*),\; k_1\in(0,+\infty),\;\text{or}\; k_2\in[k_2^*,+\infty),\; k_1\in(\xi(k_2),+\infty)\}$. In $\Gamma_1$, $(U(x,\ep),U(x,\ep))$ is unstable with respect to \eqref{eq.d-f} for sufficiently small $\ep>0$. In $\Gamma_2$, $(U(x,\ep),U(x,\ep))$ is stable with respect to \eqref{eq.d-f} for sufficiently small $\ep>0$.}
		\label{Fig.1}
	\end{figure}
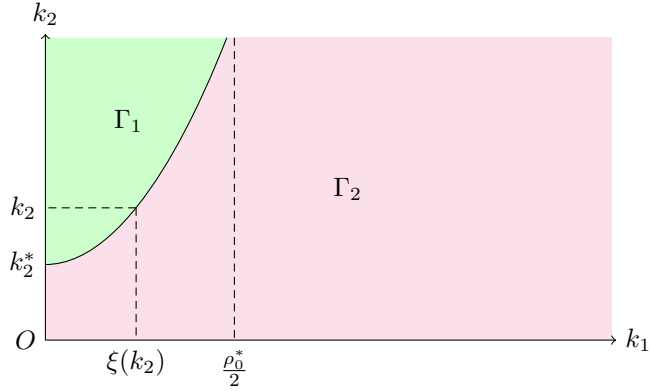
	
	\begin{theorem}[\cite{WY2}]\label{th.2-1}	
		Suppose that $\alpha\in(\alpha_0,+\infty)$, where $\alpha_0:=\sum\limits_{n \ge 0} \frac{c_1^*c_2^*}{\gamma _n^2 }\psi_n ^2(x^*)=Y(0,0,0)$. Denote
		\begin{equation*}
			\Gamma:=\{(k_1,k_2): k_2\in(k_2^*,+\infty), k_1=\xi(k_2)\}=\{(k_1,k_2): k_1\in(0,\rho_0^*/2), k_2=\zeta(k_1)\},
		\end{equation*}
		where $k_2=k_2^*$ satisfies $X(0,0,k_2^*)=\rho_0^*$, $\xi(k_2):=\frac{\rho_0^*-X(0,0,k_2)}{2}$ is increasing in $k_2\in(k_2^*,\infty)$, $\zeta(k_1)$ is the inverse function of $\xi(k_2)$ with $\xi(k_2^*)=0,\;\lim\limits_{k_2\to\infty}\xi(k_2)=\frac{\rho_0^*}{2},\;\zeta(0)=k_2^*,\;\lim\limits_{k_1\to(\rho_0^*/2)^-}\zeta\left( k_1\right) =\infty$. Define
		\begin{equation*}
			\begin{split}
				\Gamma_1&:=\{(k_1,k_2): k_2\in(k_2^*,+\infty), k_1\in(0,\xi(    k_2))\},\\
				\Gamma_2&:=\{(k_1,k_2): k_2\in(0,k_2^*),\;  k_1\in(0,\infty),\;\text{or}\; k_2\in[k_2^*,+\infty),\; k_1\in(\xi(k_2),\infty)\}.
			\end{split}
		\end{equation*}
		Then, the following conclusions hold true $($see Figure \ref{Fig.1}$)$:
		\begin{enumerate}[label=$\arabic*$.]
			\item If $(k_1,k_2)\in\Gamma_1$, then the synchronized layered steady state $(U(x,\ep), U(x,\ep))$ is unstable with respect to system \eqref{eq.d-f} for sufficiently small $\ep>0$.
			\item  If $(k_1,k_2)\in\Gamma_2$, the synchronized layered steady state $(U(x,\ep), U(x,\ep))$ is stable with respect to system \eqref{eq.d-f} for sufficiently small $\ep>0$.
		\end{enumerate}	
	\end{theorem}
	\begin{remark}
		Let $\alpha_c^\ep$ be given by Theorem \ref{th.1}. By \cite{TMN,WY2}, it holds true that $\alpha_0>\lim\limits_{\ep\to0^+}\alpha_c^\ep$. Then, for $\alpha\in(\alpha_0,\infty)$ and small $\ep>0$, the layered steady state $U(x,\ep):=(u(x,\ep),v(x,\ep))$ is stable with respect to \eqref{eq-s} and all the eigenvalues of \eqref{P-L-2-0} lie in the open left half-plane.
	\end{remark}
	
	\section{The membrane-coupled RD system with delayed couplings}
	Consider the following membrane-coupled RD system with delayed coupling:
	\begin{equation}\label{eq-o}
		\begin{cases}
			\begin{split}
				\alpha\frac{\partial u_i}{\partial t}&=\ep\dfrac{\partial ^2u_i}{\partial x^2}+\frac{1}{\ep}f(u_i,v_i)+k_1 \left( L_1[u_j]-L_2[u_i]\right),~&x\in(0,1),\;t>0,\\
				\dfrac{\partial v_i}{\partial t}&=d\dfrac{\partial ^2v_i}{\partial x^2}+g(u_i,v_i)+ k_2(v_j-v_i),~&x\in(0,1),\;t>0,\\
			\end{split}
		\end{cases}
	\end{equation}
	where $i,j\in\{1,2\}$ and $i\neq j$; $x\in(0,1)$ and $t>0$. $L_i[u]$ can be either $u$ (non-delayed case), or $\int_{-\infty}^{t} \frac{1}{\tau}e^{-\frac{t-s}{\tau}} u(x,s)ds$ (the case of the distributed delay), or $u(x,t-\tau)$ (the case of discrete delay).

	We shall treat system \eqref{eq-o} for all the cases in a unified framework. 
	
	{\bf Step 1 (Linearization)}:
	The linearized system of \eqref{eq-o} evaluated at the synchronized layered steady state solution $(U(x,\ep), U(x,\ep))$ is given by
	\begin{equation}\label{eq-o-L}
		\begin{cases}
			\begin{split}
				\alpha\frac{\partial \phi_1}{\partial t}&=\ep\dfrac{\partial ^2\phi_1}{\partial x^2}+\dfrac{1}{\ep}(f_u^\ep\phi_1+f_v^\ep\psi_1)+ k_1(L_1[\phi_2]-L_2[\phi_1]),\\
				\dfrac{\partial \psi_1}{\partial t}&=d\dfrac{\partial ^2\psi_1}{\partial x^2}+g_u^\ep\phi_1+g_v^\ep\psi_1+ k_2(\psi_2-\psi_1),\\
				\alpha\frac{\partial \phi_2}{\partial t}&=\ep\dfrac{\partial ^2\phi_2}{\partial x^2}+\dfrac{1}{\ep}(f_u^\ep\phi_2+f_v^\ep\psi_2)+ k_1(L_1[\phi_1]-L_2[\phi_2]),\\
				\dfrac{\partial \psi_2}{\partial t}&=d\dfrac{\partial ^2\psi_2}{\partial x^2}+g_u^\ep\phi_2+g_v^\ep\psi_2+ k_2(\psi_1-\psi_2).
			\end{split}
		\end{cases}
	\end{equation}
	
	{\bf Step 2 (Decomposition)}: For $i=1,2$, we rewrite $\phi_i(x,t)$ and $\psi_i(x,t)$ in the following form
	\begin{equation}\label{W w Z z}
		\phi_1(x,t):=w_1(x)e^{\la t},\; \phi_2(x,t):=w_2(x)e^{\la t},\;\psi_1(x,t):=z_1(x)e^{\la t},\;\psi_2(x,t):=z_2(x)e^{\la t},
	\end{equation}
	and set $L_i[\phi_j]$ to be either $w_j(x)e^{\la t}$ (the non-delayed case), or $\frac{w_j(x)}{1+\la\tau}e^{\la t}$ (the case of distributed delay), or $w_j(x)e^{\la (t-\tau)}$ (the case of discrete delay).
	
	Substituting \eqref{W w Z z} into \eqref{eq-o-L}, we obtain the following eigenvalue problem 	
	\begin{equation}\label{D-L-2-a}
		\begin{cases}
			\begin{split}
				&\ep^2 w_1''(x)+f_u^{\ep} w_1+f_v^{\ep} z_1+\ep k_1 (\ell_1 w_2-\ell_2 w_1)=\ep\alpha\la w_1,\\
				&d z_1''(x)+g_u^{\ep} w_1+g_v^{\ep} z_1+k_2(z_2-z_1)=\la z_1,\\
				&\ep^2 w_2''(x)+f_u^{\ep} w_2+f_v^{\ep} z_2+\ep k_1 (\ell_1 w_1-\ell_2 w_2)=\ep\alpha\la w_2,\\
				&d z_2''(x)+g_u^{\ep} w_2+g_v^{\ep} z_2+k_2(z_1-z_2)=\la z_2,
			\end{split}
		\end{cases}
	\end{equation}
	where $\ell_i$ ($i=1,2$) can be either $1$ (when $L_i$ has no delay), or $\theta_\tau(\la):=e^{-\la\tau}$ (when $L_i$ has discrete delay), or $\theta_\tau(\la):=\frac{1}{1+\la\tau}$ (when $L_i$ has distributed delay).
	
	Following \cite{LE2, TMN}, we define
	\begin{equation*}
		\begin{split}
			w_s(x):=&\frac{1}{2}(w_1(x)+w_2(x)),\; w_a(x):=\frac{1}{2}(w_1(x)-w_2(x)),\\
			z_s(x):=&\frac{1}{2}(z_1(x)+z_2(x)),\;\;\;\;z_a(x):=\frac{1}{2}(z_1(x)-z_2(x)).
		\end{split}
	\end{equation*}
	
	Then, the eigenvalue problem \eqref{D-L-2-a} can be decomposed into the following two eigenvalue problems:
	\begin{enumerate}
		\item One is the {\bf symmetric} sub-eigenvalue problem
		\begin{equation}\label{EP-o-1}
			\begin{pmatrix}
				\ep^2&0\\
				0&d
			\end{pmatrix}
			\begin{pmatrix}
				w_s''(x)\\
				z_s''(x)
			\end{pmatrix}+
			\begin{pmatrix}
				f_u^\ep + \ep k_1(p+q\theta_\tau(\la))&f_v^\ep\\
				g_u^\ep&g_v^\ep\\
			\end{pmatrix}\begin{pmatrix}
				w_s(x)\\
				z_s(x)
			\end{pmatrix}=\begin{pmatrix}
				\ep\al\la&0\\
				0&\la
			\end{pmatrix}
			\begin{pmatrix}
				w_s(x)\\
				z_s(x)
			\end{pmatrix},
		\end{equation}
		\item The other is the {\bf asymmetric} sub-eigenvalue problem
		\begin{equation}\label{EP-o-2}
			\begin{pmatrix}
				\ep^2&0\\
				0&d
			\end{pmatrix}
			\begin{pmatrix}
				w_a''(x)\\
				z_a''(x)
			\end{pmatrix}+
			\begin{pmatrix}
				f_u^\ep+\ep k_1(p+q\theta_\tau(\la))&f_v^\ep\\
				g_u^\ep&g_v^\ep-2k_2\\
			\end{pmatrix}\begin{pmatrix}
				w_a(x)\\
				z_a(x)
			\end{pmatrix}=\begin{pmatrix}
				\ep\al\la&0\\
				0&\la
			\end{pmatrix}
			\begin{pmatrix}
				w_a(x)\\
				z_a(x)
			\end{pmatrix},
		\end{equation}
		where $(p,q)$ and $\theta_\tau(\la)$ are determined in the following Proposition \ref{Pthetasa}.
	\end{enumerate}
	\begin{proposition}\label{Pthetasa}
		The symmetric and asymmetric eigenvalue problems \eqref{EP-o-1} and \eqref{EP-o-2} can be written in the following unified form
		\begin{equation}\label{EP-g}
			\begin{pmatrix}
				\ep^2&0\\
				0&d
			\end{pmatrix}
			\begin{pmatrix}
				w''(x)\\
				z''(x)
			\end{pmatrix}+
			\begin{pmatrix}
				f_u^\ep + \ep k_1(p+q\theta_\tau(\la))&f_v^\ep\\
				g_u^\ep&g_v^\ep-2K\\
			\end{pmatrix}\begin{pmatrix}
				w(x)\\
				z(x)
			\end{pmatrix}=\begin{pmatrix}
				\ep\al\la&0\\
				0&\la
			\end{pmatrix}
			\begin{pmatrix}
				w(x)\\
				z(x)
			\end{pmatrix},
		\end{equation}
		where
		\begin{equation}\label{thetau}
			\theta_\tau(\la):=
			\begin{cases}
				e^{-\lambda\tau},&\text{for all of Case 1, Case 2 and Case 3},\\
				\dfrac{1}{1+\lambda\tau},&\text{for all of Case 4, Case 5 and Case 6}.
			\end{cases}
		\end{equation}
		Moreover,
		\begin{enumerate}
			\item For symmetric sub-eigenvalue problem \eqref{EP-o-1}, 
			\begin{equation}\label{eq:pqk-sym}
				(p,q,K)=
				\begin{cases}
					(0,0,0), & \text{in Cases 1 and 4},\\
					(-1,1,0), & \text{in Cases 2 and 5},\\
					(1,-1,0), & \text{in Cases 3 and 6.}
				\end{cases}
			\end{equation}
			\item For asymmetric sub-eigenvalue problem \eqref{EP-o-2}, 
			\begin{equation}\label{eq:pqk-asym}
				(p,q,K)=
				\begin{cases}
					(0,-2,k_2), & \text{in Cases 1 and 4},\\
					(-1,-1,k_2), & \text{in Cases 2 and 5},\\
					(-1,-1,k_2), & \text{in Cases 3 and 6.}
				\end{cases}
			\end{equation}
		\end{enumerate}
		In particular, as $\tau\to0^+$, for any $\la\in\C$, we have $p+q\theta_0(\la)=0$ for the symmetric sub-problem \eqref{EP-o-1} and $p+q\theta_0(\la)=-2$ for the asymmetric sub-problem \eqref{EP-o-2}.
	\end{proposition}
	\begin{proof}
		The proof is trivial and we shall omit it.
	\end{proof}
	
	{\bf Step 3 (SLEP reduction)}: 
	Define
	\begin{equation*}
		\C_{\mu^*}:=\left\{(\la,\tau)\in\C\times(0,\infty): \operatorname{Re}\la>-\mu^*,\; \tau\operatorname{Re}\la>-\frac{1}{2},  \;0<\mu^*<\inf\limits_{x\in(0,1)}\frac{ f_u^\star g_v^\star-f_v^\star g_u^\star}{- f_u^\star}\right\},
	\end{equation*}
	where $\tau\operatorname{Re}\la>-1/2$ is additionally imposed for the uniform boundedness of $\vert\theta_{\tau}(\la)\vert$ for $(\la,\tau)\in\C_{\mu^*}$ which is shown later. 
	
	For any $(\la,\tau)\in\C_{\mu^*}$, as $\ep\to0^+$, we denote $\la^*=\la_R^*+i\la_I^*:=\lim\limits_{\ep\to0^+}\la$, and
	\begin{equation}\label{Thetau*}
		\theta_\tau^*=\theta_R^*+i\theta_I^*:=\lim\limits_{\ep\to0^+}\theta_\tau(\la)=
		\begin{cases}
			e^{-\lambda^*\tau},&\text{for all of Case 1, Case 2 and Case 3},\\
			\dfrac{1}{1+\lambda^*\tau},&\text{for all of Case 4, Case 5 and Case 6}.
		\end{cases}
	\end{equation}
	
	Then, as $\ep\to0^+$, by using the SLEP reduction method (for instance, page 490-494 of \cite{NM}), we can obtain the complex SLEP system of \eqref{EP-g}:
	\begin{equation}\label{F*}
		\mathcal{F}^*(\la^*,\tau):=\rho_0^*-\alpha\la^*+k_1\left(p+q\theta_\tau^*\right)-\sum_{n=0}^{\infty}\dfrac{c_1^*c_2^*}{\gamma_n+\la^*+2K}\psi_n^2(x^*)=0,
	\end{equation}
	or equivalently (by separating the real and the imaginary parts of \eqref{F*}), the real SLEP system
	\begin{equation}\label{SLEP-g}
		\begin{cases}
			\begin{split}
				&\rho_0^*+k_1\left(p+q\theta_R^*\right)-\alpha\la_R^*=X(\la_R^*,(\la_I^*)^2,K),\\
				&\alpha\la_I^*-k_1q\,\theta_I^*=\la_I^*Y(\la_R^*,(\la_I^*)^2,K),
			\end{split}
		\end{cases}
	\end{equation}
	where $(p,q,K)$ is given by \eqref{eq:pqk-sym} and \eqref{eq:pqk-asym}, $\theta_\tau^*=\theta_R^*+i\theta_I^*$ is given by \eqref{Thetau*}.
	
	More precisely,	the symmetric and asymmetric SLEP systems corresponding to \eqref{EP-o-1} and \eqref{EP-o-2} are given by
	\begin{align}
		\mathcal{F}_s^*(\lambda^*,\tau)
		&:= \rho_0^*-\alpha\lambda^*+k_1\bigl(p+q\theta_\tau^*\bigr)
		-\sum_{n=0}^{\infty}\frac{c_1^*c_2^*}{\gamma_n+\lambda^*}\psi_n^2(x^*)=0,\label{SLEP-g-1}\\
		\mathcal{F}_a^*(\lambda^*,\tau)
		&:= \rho_0^*-\alpha\lambda^*+k_1\bigl(p+q\theta_\tau^*\bigr)
		-\sum_{n=0}^{\infty}\frac{c_1^*c_2^*}{\gamma_n+\lambda^*+2k_2}\psi_n^2(x^*)=0,\label{SLEP-g-2}
	\end{align}
	where $(p,q)$ are given by \eqref{eq:pqk-sym} and \eqref{eq:pqk-asym}, respectively.
	
	First of all, we shall identify the spectrum of \eqref{SLEP-g-1} and \eqref{SLEP-g-2} when delay vanishes ($\tau\to0^+$), as well as the existence of zero or positive eigenvalues in $\tau\in(0,\infty)$.
	\begin{proposition}\label{p.pzp}
		Suppose that $\alpha\in(\alpha_0,\infty)$. Then, in any of Cases 1--6, we have the following conclusions.
		\begin{enumerate}
			\item For $k_1>0$, the symmetric SLEP system \eqref{SLEP-g-1} has no zero eigenvalues in $\tau\in(0,\infty)$. Moreover, as $\tau\to0^+$, all the eigenvalues have  strictly negative real parts.
			\item For $(k_1,k_2)\in\Gamma_1$, the asymmetric SLEP system \eqref{SLEP-g-2} always has a positive eigenvalue in $\tau\in(0,\infty)$. While if $(k_1,k_2)\in\Gamma_2$, system \eqref{SLEP-g-2} admits no zero eigenvalues in $\tau\in(0,\infty)$; as $\tau\to0^+$, all the eigenvalues have strictly negative real parts.
		\end{enumerate}
	\end{proposition}
	\begin{proof}
		For the asymmetric SLEP system \eqref{SLEP-g-2},
		if $\la_R^*=\la_I^*=0$, then in any of Cases 1--6, it reduces to $\rho_0^*-2k_1=X(0,0,k_2)$, which is independent of $\tau$. The equality holds if and only if $(k_1,k_2)\in\Gamma$ (see Theorem \ref{th.2-1}). Thus, if $(k_1,k_2)\in\Gamma_1\cup\Gamma_2$, then system \eqref{SLEP-g-2} has no zero eigenvalue for any $\tau\in(0,\infty)$. 
		
		Moreover, as $\tau\to0^+$, if $(k_1,k_2)\in\Gamma_2$, from \cite{WY2}, all the eigenvalues of system \eqref{SLEP-g-2} have negative real parts. While if $(k_1,k_2)\in\Gamma_1$, then as $\tau\to0^+$, system \eqref{SLEP-g-2} admits a positive eigenvalue. Further, for any $\tau>0$, we have 
		\begin{equation*}
			\mathcal F_a^*(0,\tau)
			=\rho_0^*-2k_1-X(0,0,k_2)>0,\;\;
			\lim\limits_{\la\to\infty}\mathcal F_a^*(\la,\tau)=-\infty,
		\end{equation*}
		implying that for any $\tau\in(0,\infty)$, the asymmetric SLEP system admits at least one positive real eigenvalue.
		
		As a simpler case, the conclusions of symmetric SLEP system \eqref{SLEP-g-1} can be obtained similarly.
	\end{proof}
	
	Next we shall relate the spectrum of \eqref{F*} to that of \eqref{EP-g} for sufficiently small $\ep>0$.
	
	\begin{theorem}\label{p.con}
		For any $(k_1,k_2)\in\Gamma_1\cup\Gamma_2$, the following conclusions hold true.
		\begin{enumerate}
			\item If the SLEP system \eqref{F*} has a unique pair of critical complex eigenvalues $\la_R^*(\tau)\pm i\la_I^*(\tau)$ in $\tau\in(\tau_H^*-\kappa,\tau_H^*+\kappa)$ for some $\kappa>0$, which satisfies
			\begin{equation}\label{eq.eH}
				\la_R^*(\tau_H^*)=0,\; \la_I^*(\tau_H^*)=\omega_H^*,\;\frac{\pa \mathcal{F}^*}{\pa \la}( i \omega_H^*,\tau_H^*)\neq 0,\; \frac{d}{d\tau}\la_R^*(\tau_H^*)>0,	
			\end{equation}
			where $(\pm i\omega_H^*,\tau_H^*)$ is the solution of \eqref{F*}, then for sufficiently small $\ep>0$, there exist $\tau_H(\ep)>0$ and $\omega_H(\ep)\neq0$, such that system \eqref{EP-g} admits a simple pair of complex eigenvalues $\la_R(\tau)\pm i \la_I(\tau)$ for $\tau$ belonging to a small neighborhood of $\tau_H(\ep)$, satisfying
			\begin{equation*}
				\la_R(\tau_H(\ep))=0,\; \la_I(\tau_H(\ep))=\omega_H(\ep),\; \frac{d}{d\tau}\la_R(\tau_H(\ep))>0,\; \omega_H(0)=\omega_H^*,\; \tau_H(0)=\tau_H^*.
			\end{equation*}
			\item For any fixed $\tau\in(0,\infty)$, if the SLEP system \eqref{F*} admits an eigenvalue with positive real part, then \eqref{EP-g} also admits an eigenvalue with positive real part for sufficiently small $\ep>0$.
			\item For any fixed $\tau\in(0,\infty)$, if all eigenvalues of the SLEP system \eqref{F*} have strictly negative real parts, then the same conclusion holds for  \eqref{EP-g} for sufficiently small $\ep>0$.
		\end{enumerate}
	\end{theorem}
	\begin{proof}
		By \cite{NF,NM}, the eigenvalue problem \eqref{EP-g} can be reduced to
		\begin{equation*}
			\mathcal{F}\left(\la,\tau,\ep\right):= \rho(\ep)+k_1\left( p+q\theta_{\tau}\right) -\alpha\la-\left\langle (T_{\la,\tau}^{\ep})^{-1}\frac{g_u^{\ep}\varphi_0^{\ep}}{\sqrt{\ep}},\frac{-f_v^{\ep}\varphi_0^{\ep}}{\sqrt{\ep}} \right\rangle=0,
		\end{equation*}
		where $\theta_\tau$ and $(p,q,K)$ are given by \eqref{thetau}-\eqref{eq:pqk-asym}. Moreover,  $(T_{\la,\tau}^{\ep})^{-1}$ is the inverse operator of 
		\begin{equation*}
			T_{\la,\tau}^{\ep}:=-d\frac{d^2}{dx^2}-g_v^{\ep}+2K+g_u^{\ep}\left(\ep^2\dfrac{d^2}{dx^2}+f_u^{\ep}+\ep k_1\left( p+q\theta_{\tau}\right) -\ep\alpha\la\right)^{\dag}f_v^{\ep}+\la,
		\end{equation*}
		where $(\cdot)^\dagger$ denotes the inverse operator in the complementary space to $\varphi_0^\ep$  (for instance, \eqref{eq.dag}). 
		
		In particular,  $\mathcal{F}(\la,\tau,\ep)$ depends on $\tau$ and $\la$ analytically in $\C_{\mu^*}$, and depends on $\ep$ continuously up to $\ep=0$. Let $\mathcal{F}^* (\la,\tau)$ be given by \eqref{F*}. Then, as $\ep\to0^+$, we have
		\begin{equation}\label{F->F*1}
			\mathcal{F} (\la,\tau,\ep)\to \mathcal{F}^* (\la,\tau),\;
			\frac{\pa \mathcal{F}}{\pa \la}( \la,\tau,\ep)\to\frac{\pa \mathcal{F}^*}{\pa \la}(\la,\tau),\;
			\frac{\pa \mathcal{F}}{\pa \tau}(\lambda,\tau,\epsilon)\to
			\frac{\pa \mathcal{F}^*}{\pa \tau}(\lambda,\tau),
		\end{equation}
		uniformly on any compact subset of $\C_{\mu^*}$, where $\mu^*$ can be chosen smaller if necessary.
		
		For any $\tau\in(0,\infty)$, we shall show that the eigenvalues satisfying $\operatorname{Re}\la>\max\{-\mu^*,-1/(2\tau)\}$ are uniformly bounded with estimate $\vert\la\vert<C$ for some constant $C>0$ which is independent of $\tau$ and $\ep$. First, for any $(\la,\tau)\in\C_{\mu^*}$, one can check that $\vert\theta_\tau(\la)\vert<2$ uniformly for both discrete and distributed delays. Consequently, we have $\left|\epsilon k_1(p+q\theta_\tau(\lambda))\right|
		\leq \epsilon k_1(|p|+2|q|)$. Thus, by the similar arguments as in Lemma 4.1 of \cite{NM} or Lemma 3.1 of \cite{NF}, the operator  $(T_{\la,\tau}^\ep)^{-1}$ is a well-defined uniformly bounded operator from $H^{-1}(0,1)$ to $H_N^1(0,1)$ in $(\la,\tau)\in\C_{\mu^*}$. Here $H_N^1(0,1)$ denotes the subspace of space $H^1(0, 1)$ where all the functions satisfy the homogeneous Neumann boundary conditions. Then by \eqref{fvgu}, we have
		\begin{equation*}
			\left|
			\left\langle
			(T_{\lambda,\tau}^\ep)^{-1}
			\frac{g_u^\ep\varphi_0^\ep}{\sqrt{\ep}},
			\frac{-f_v^\ep\varphi_0^\ep}{\sqrt{\ep}}
			\right\rangle
			\right|\le C_1,
		\end{equation*}
		where $C_1$ is independent of $\ep\in(0,\ep_0)$ and $(\la,\tau)\in\C_{\mu^*}$.
		
		Hence, if $\la$ satisfies $\mathcal F(\lambda,\tau,\ep)=0$ and $\operatorname{Re}\la>\max\{-\mu^*,-1/(2\tau)\}$, then there exists some constant $C>0$, such that 
		\begin{equation*}
			\alpha|\lambda|
			\le |\rho(\ep)|+k_1(|p|+|q|\vert\theta_\tau(\la)\vert)+C_1\le|\rho(\ep)|+k_1(|p|+2|q|)+C_1<C,
		\end{equation*}
		which shows the boundedness independent of sufficiently small $\ep>0$ and $\tau\in(0,\infty)$. Therefore, the proof can be completed by the continuity argument \eqref{F->F*1} and implicit function theorem as in \cite{NM,TMN,WY2}.
	\end{proof}
	
	\begin{corollary}\label{cor.hopf}
		In any of Cases 1--6, for $(k_1,k_2)\in\Gamma_2$, if exactly one of systems \eqref{SLEP-g-1} and \eqref{SLEP-g-2} admits a unique pair of purely imaginary eigenvalues $\pm i\omega_H^*$ at the first critical value  $\tau_H^*>0$ $($in the sense that all the eigenvalues have negative real parts for $\tau\in(0,\tau_H^*)$$)$, and satisfies \eqref{eq.eH}, then there exists a $\tau_H(\ep)>0$ for sufficiently small $\ep>0$ satisfying $\tau_H(0)=\tau_H^*$, such that the synchronized layered steady state $(U(x,\ep),U(x,\ep))$ is stable for $\tau\in(0,\tau_H(\ep))$, but unstable in $\tau\in(\tau_H(\ep),\tau_H(\ep)+\kappa)$ for some sufficiently small $\kappa>0$. Moreover, system \eqref{eq-o} undergoes a primary Hopf bifurcation around $(U(x,\ep),U(x,\ep))$ at $\tau=\tau_H(\ep)$, and there occurs a family of periodic solutions with minimal positive period $T_s$, where $s>0$ is sufficiently small. In particular, if the Hopf bifurcation is triggered by symmetric sub-eigenvalue problem \eqref{EP-o-1} $($resp., the asymmetric sub-eigenvalue problem \eqref{EP-o-2}$)$, then the bifurcating periodic solutions are in-phase breathers $($resp., anti-phase breathers$)$.
	\end{corollary}
	\begin{proof}
		The proof of the local existence of Hopf bifurcation uses the methods appeared in \cite{SSS} (for distributed delay) and \cite{SWS} (for discrete delay). Suppose that the primary Hopf bifurcation is triggered by asymmetric sub-eigenvalue problem \eqref{EP-o-2}, which corresponds to the purely imaginary eigenvalue $i\omega_H(\ep)$ and eigenfunction $\Phi_H(x):=(\phi(x),\psi(x),-\phi(x),-\psi(x))$ at the first critical value $\tau_H(\ep)$, where $\omega_H(\ep)>0$ and $\ep>0$ is sufficiently small. We shall show that the bifurcating periodic solutions are anti-phase breathers with phase difference $T_s/2$.
		
		Define the layer-exchange operator $\mathcal{T}(u_1,v_1,u_2,v_2):=(u_2,v_2,u_1,v_1)$. Clearly, from the symmetry of system \eqref{eq-o} and \eqref{EP-o-2},  we can conclude that \eqref{eq-o} is invariant under $\mathcal{T}$, and $\mathcal{T}\Phi_H=-\Phi_H$. Let $W_s(t)$ denote the local Hopf bifurcation periodic solution, and let $T_s=2\pi/\Omega_s$ be its minimal positive period, where $s>0$ is sufficiently small. In particular, as $s\to0^+$, we have $\Omega_s\to \omega_H(\ep)$ and $T_s\to\frac{2\pi}{\omega_H(\ep)}$. 
		
		Since system \eqref{eq-o} is $\mathcal{T}$-equivariant, then by the local uniqueness of Hopf branch up to time translation, $\mathcal{T}W_s(t)$ is a periodic solution with a phase shift from $W_s(t)$. Namely, there exists a $\vartheta_s\in[0,T_s)$ such that $\mathcal{T}W_s(t)=W_s(t+\vartheta_s)$. For sufficiently small $s>0$,
		\begin{equation}\label{Ws}
			W_s(t)=W_0+s\left[\Phi_H e^{i\Omega_s t}+\overline{\Phi_H}e^{-i\Omega_s t}\right]+O(s^2),
		\end{equation}
		where $W_0=(U(x,\ep),U(x,\ep))$.
		
		Note that $\mathcal{T}W_s(t)=W_s(t+\vartheta_s)$, and $\mathcal{T}\Phi_H=-\Phi_H$. Then by applying $\mathcal{T}$ to the right hand side of \eqref{Ws}, we obtain
		\begin{equation}\label{Ws2}
			W_s(t+\vartheta_s)=\mathcal{T}W_s(t)=W_0-s\left[\Phi_H e^{i\Omega_s t}+\overline{\Phi_H}e^{-i\Omega_s t}\right]+O(s^2).
		\end{equation}
		
		Replacing $t$ by $t+\vartheta_s$ in \eqref{Ws} yields
		\begin{equation}\label{Ws3}
			W_s(t+\vartheta_s)=W_0+s\left[\Phi_H e^{i\Omega_s (t+\vartheta_s)}+\overline{\Phi_H}e^{-i\Omega_s (t+\vartheta_s)}\right]+O(s^2).
		\end{equation}
		
		By comparing the leading-order terms of \eqref{Ws2} and \eqref{Ws3}, we can derive $e^{i\Omega_s\vartheta_s}\to-1$ as $s\to0$, implying that $\vartheta_s\Omega_s\to\pi$, or equivalently, $\vartheta_s/T_s\to 1/2$.
		
		On the other hand, from the symmetry of system \eqref{eq-o}, we have $\mathcal{T}^2=I$, such that $W_s(t)=\mathcal{T}^2 W_s(t)=W_s(t+2\vartheta_s)$. Since $T_s$ is the minimal positive period and $\vartheta_s\in[0,T_s)$, we have $\vartheta_s=0$ or $T_s/2$. Then by the limit $\vartheta_s/T_s\to 1/2$ as $s\to0$, we derive $\vartheta_s\equiv T_s/2$. Hence, the bifurcating periodic solutions are anti-phase breathers with phase difference $T_s/2$.
		
		The case of in-phase breathers can be proved similarly. The proof is complete.
	\end{proof}
	
	Throughout the paper, we shall assume that $\alpha\in(\alpha_0,\infty)$ and  $(k_1,k_2)\in\Gamma_1\cup\Gamma_2$. It remains to investigate the existence of purely imaginary eigenvalues and the transversality condition with respect to SLEP systems \eqref{SLEP-g}.
	
	\subsection{Analysis of the SLEP problem with discrete delays (Cases 1--3)}   	
	In this subsection, we shall establish the general framework to analyze the spectral distribution of the SLEP system \eqref{SLEP-g} for Cases 1--3.
	
	\begin{proposition}Suppose that $\la_R^*=0$ and $\la_I^*=\omega>0$. Then, the real SLEP problem \eqref{SLEP-g} can be reduced to the following two algebraic equations in $(\omega,\tau)$:
		\begin{subequations}\label{eq-g1}
			\begin{numcases}{}
				X(0,\omega^2,K)- \rho_0^* - k_1 p= k_1 q\cos(\omega\tau),\label{u1: im-a}\\
				\omega\left(Y(0,\omega^2,K)-\alpha \right)  =  k_1 q\sin(\omega\tau),\label{u1: im-b}	
			\end{numcases}
		\end{subequations}
		where $X(0,\omega^2,K)$ and $Y(0,\omega^2,K)$ are defined in terms of \eqref{AB}; $(p,q,K)$ is given by \eqref{eq:pqk-sym} and \eqref{eq:pqk-asym}.
	\end{proposition}
	\begin{proof}
		It can be directly obtained from \eqref{SLEP-g} and Proposition \ref{Pthetasa}.
	\end{proof}
	
	In the rest of this subsection, we shall assume that $q\neq0$. Otherwise, when $q=0$, equation \eqref{u1: im-b} has no positive solution with respect to $\omega$ due to $\alpha>\alpha_0$.
	
	By $\cos^2(x)+\sin^2(x)=1$ for any $x\in\R$, system \eqref{eq-g1} is equivalent to
	\begin{equation}\label{S1}
		X(0,\omega^2,K)- \rho_0^* - k_1 p= k_1 q\cos(\omega\tau),
	\end{equation}
	$\sin(\omega\tau)q<0$, and
	\begin{equation*}
		\left( X(0,\omega^2,K)- \rho_0^* - k_1 p\right) ^2+\omega^2 \left(Y(0,\omega^2,K)-\alpha \right)^2=q^2 k_1^2,
	\end{equation*}
	which is equivalent to 
	\begin{equation}\label{S2}
		F(\omega^2):=\left( X(0,\omega^2,K)- \rho_0^* \right) ^2+\omega^2 \left(Y(0,\omega^2,K)-\alpha \right)^2=(q^2-p^2) k_1^2+2pk_1(X(0,\omega^2,K)-\rho_0^*),
	\end{equation}
	where the right hand side is denoted by $G(\omega^2;k_1)$.
	
	It turns out that the monotonicity of $F(\omega^2)$ plays a crucial role in the existence/nonexistence of purely imaginary eigenvalues for Cases 1--3 and Cases 4--6.
	
	\begin{lemma}\label{le.F}
		Suppose that $\alpha\in(\alpha_0,\infty)$. Then $F(\omega^2)$ is positive for any $\omega>0$. Moreover, 
		\begin{enumerate}
			\item If $K\in[0,k_2^*)$, then there exists an $\omega_*\in[0,\omega_0)$, where $\omega_0=\omega_0(K)$ is the unique positive solution of $X(0,\omega_0^2,K)=\rho_0^*$, such that $F(\omega^2)$ is decreasing in $\omega\in(0,\omega_*)$ and increasing in $\omega\in(\omega_*,\infty)$. In particular, $\omega_0(K)=0$ at $K=k_2^*$. See Figure \ref{Fome-1-a}.
			\item If $K\in[k_2^*,\infty)$, then $F(\omega^2)$ is increasing in $\omega\in(\omega_*,\infty)$, where $\omega_*=0$. See Figure \ref{Fome-1-b}.
		\end{enumerate}
	\end{lemma}
	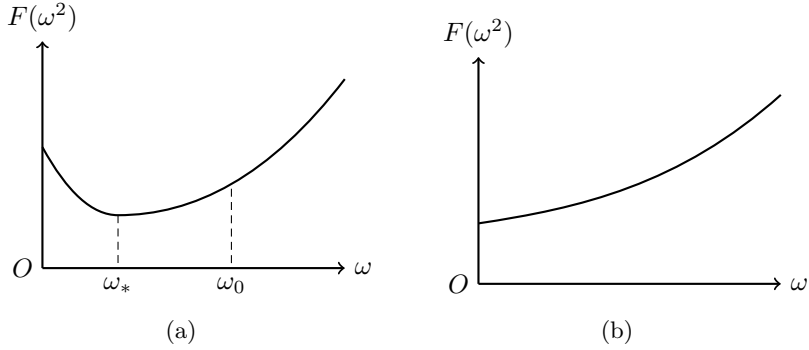
\begin{figure}
		\centering
		\subfloat[]{
			\begin{tikzpicture}
				\draw[->,thick] (0,0) --(4,0) node[right] {$\omega$};
				\draw[->,thick] (0,0) --(0,3) node [above] {$F(\omega^2)$};
				\draw  (0,0) node[left] {$O$};
				\draw[thick] (0,1.6)
				.. controls (0.33,1) and (0.66,0.7) .. (1,0.7)
				.. controls (2,0.7)    and (3,1.2) .. (4,2.5);
				\draw [densely dashed] (1,0.7)--(1,0) node [below] {$\omega_*$};
				\draw [densely dashed] (2.5,1.05)--(2.5,0) node [below] {$\omega_0$};
			\end{tikzpicture}\label{Fome-1-a}
		}\hspace{10pt}
		\subfloat[]{
			\begin{tikzpicture}
				\draw[->,thick] (0,0) --(4,0) node[right] {$\omega$};
				\draw[->,thick] (0,0) --(0,3) node [above] {$F(\omega^2)$};
				\draw  (0,0) node[left] {$O$};
				\draw[thick] (0,0.8)
				.. controls (1.33,1) and (2.66,1.3) .. (4,2.5);
			\end{tikzpicture}\label{Fome-1-b}
		}
		\caption{Graphs of $F(\omega^2)$. (a) For $K\in[0,k_2^*)$, $F(\omega^2)$ is decreasing in $\omega\in(0,\omega_*)$ and increasing in $\omega\in(\omega_*,\infty)$ for some $\omega_*\in[0,\omega_0)$, where $\omega_0=\omega_0(K)$ satisfies $X(0,\omega_0^2,K)=\rho_0^*$. Moreover, $F(\omega_*^2)>0$. (b) For $K\in[k_2^*,\infty)$, $F(\omega^2)$ is increasing and satisfies $F(\omega^2)>F(0)=(\rho_0^*-X(0,0,K))^2\ge0$ in $\omega\in(0,\infty)$, where the equality holds if and only if $K=k_2^*$.}
		\label{Fig. Fome-1}
	\end{figure}
	\begin{proof}
		The strict positivity of $F(\omega^2)$ is clear. Differentiating $F(\omega^2)$ with respect to $\omega^2$, we have
		\begin{equation}\label{F'}
			F'(\omega^2):=\frac{d F}{d \omega^2}(\omega^2)=-2\left(\rho_0^*-X\right)\px+\left(\alpha-Y\right)^2-2\omega^2\left(\alpha-Y\right)\py,
		\end{equation}
		and
		\begin{equation}\label{F''}
			\begin{split}
				F''(\omega^2):=\frac{d^2 F}{d (\omega^2)^2}(\omega^2)=&2\left(\px\right)^2-2\left(\rho_0^*-X\right)\frac{\pa^2 X}{\pa(\omega^2)^2}+2\omega^2\left( \py\right)^2  \\			
				&+2(\alpha-Y)\left(-2\py-\omega^2\frac{\pa^2 Y}{\pa(\omega^2)^2} \right),
			\end{split}
		\end{equation}
		where $X,Y,\px,\py,\frac{\pa^2 X}{\pa(\omega^2)^2},\frac{\pa^2 Y}{\pa(\omega^2)^2}$ are evaluated at $(0,\omega^2,K)$. 
		
		Moreover, from \eqref{AB}, we can obtain that
		\begin{equation*}
			-2\py-\omega^2\frac{\pa^2 Y}{\pa(\omega^2)^2}\bigg|_{(0,\omega^2,K)} =c_1^*c_2^*\sum_{n=0}^{\infty} \frac{2(\gamma_n+2K)^2}{((\gamma_n+2K)^2+\omega^2)^3}\psi_n^2(x^*)>0.
		\end{equation*}
		
		1. If $K\in[0,k_2^*)$, where $k_2^*$ satisfies $X(0,0,k_2^*)=\rho_0^*$ (see Theorem \ref{th.2-1}), then there exists a unique	
		$\omega_0=\omega_0(K)>0$ such that
		\begin{equation}\label{om0}
			\rho_0^*-X(0,\omega^2,K)
			\begin{cases}
				\begin{split}
					&<0,~&\text{if}\;\omega^2<\omega_0^2,\\
					&=0,~&\text{if}\;\omega^2=\omega_0^2,\\
					&>0,~&\text{if}\;\omega^2>\omega_0^2.
				\end{split}
			\end{cases}
		\end{equation}
		Moreover, $\omega_0\to0$ as $K\to(k_2^*)^-$.
		
		From \eqref{F'}, \eqref{F''}, \eqref{om0} and Proposition \ref{prop.AB}, we can conclude that $F''(\omega^2)>0$ in $\omega^2\in[0,\omega_0^2)$ and $F'(\omega^2)>0$ in $\omega^2\in[\omega_0^2,\infty)$. Namely, $F(\omega^2)$ is decreasing in $\omega^2\in(0,\omega_*^2)$ and then increasing in $\omega^2\in(\omega_*^2,+\infty)$ for some $\omega_*\in[0,\omega_0)$, which depends on $\alpha$ and $K$. Thus, $F(\omega^2)$ attains the minimum value $F(\omega_*^2)>0$ at $\omega^2=\omega_*^2\in[0,\omega_0^2)$.
		
		2. If $K\in[k_2^*,\infty)$, then $\rho_0^*-X(0,\omega^2,K)>0$ for $\omega^2\in(0,\infty)$, which implies $F'(\omega^2)>0$ from \eqref{F'}. Hence, $F(\omega^2)$ is increasing in $(0,\infty)$ and attains the minimum value  $F(0)=\left(\rho_0^*-X(0,0,K)\right)^2\ge0$ at $\omega=\omega_*:=0$, where $F(0)=0$ if and only if $K=k_2^*$. 
	\end{proof}
	
	\begin{lemma}\label{le.e1}
		Suppose that $\alpha\in(\alpha_0,\infty)$. Then, for any $(p,q,K)$ given in \eqref{eq:pqk-sym} and \eqref{eq:pqk-asym} satisfying $q\neq0$, there exists a  $K_1=K_1(p,q,K)\ge0$, where the equality holds if and only if $(p,q,K)=(0,-2,k_2^*)$, such that the following conclusions hold:
		\begin{enumerate}
			\item If $k_1\in(0,K_1)$, then \eqref{eq-g1} has no purely imaginary eigenvalues regardless of $\tau> 0$.
			\item If $k_1\in(K_1,+\infty)$, then \eqref{eq-g1} admits at least one pair of purely imaginary eigenvalues $\pm i\omega_H^*$ at
			\begin{equation}\label{alp2om}
				\tau=\tau^{\omega,j}:=
				\begin{cases}
					\begin{split}
						&\frac{1}{\omega}\arccos\left\lbrace \frac{ X(0,\omega^2,K)-\rho_0^*-k_1p}{qk_1}\right\rbrace +\frac{2j\pi}{\omega},~&\text{if $q<0$},\\
						&\frac{1}{\omega}\arccos\left\lbrace -\frac{ X(0,\omega^2,K)-\rho_0^*-k_1p}{qk_1}\right\rbrace +\frac{(2j+1)\pi}{\omega},~&\text{if $q>0$},\\
					\end{split}
				\end{cases}			
			\end{equation}
			where $\omega=\omega_H^*$, $j\in\mathbb{N}_0$.  Moreover, there is a $\tau_H^*>0$, such that \eqref{eq-g1} has no purely imaginary eigenvalues for $\tau\in(0,\tau_H^*)$. At $\tau=\tau_H^*$, there exists a unique pair of purely imaginary eigenvalues $\pm i \omega_H^*$.
		\end{enumerate}
		In particular, 
		\begin{enumerate}
			\item If $(p,q,K)=(0,-2,k_2)$ or $(-1,-1,k_2)$, then $K_1(p,q,K)=\xi(k_2)$ in $k_2\in(k_2^*,\infty)$, where $\xi(k_2)$ is given by Theorem \ref{th.2-1}.
			\item If $(p,q,K)=(-1,-1,k_2)$, then $K_1(-1,-1,k_2)$ is continuous in $k_2\in(0,k_2^*]$ and $(k_2^*,\infty)$, where $ K_1(-1,-1,k_2)\to 0< K_1(-1,-1,k_2^*)$ as $k_2\to(k_2^*)^+$. Moreover, $K_1(-1,-1,k_2)|_{k_2=0}= K_1(-1,1,0)$.
			\item If $(p,q,K)=(0,-2,k_2)$, then $K_1(0,-2,k_2)$ is continuous in $k_2\in(0,\infty)$.
		\end{enumerate}
	\end{lemma}
	\begin{proof}
		For any $(p,q,K)$ given in \eqref{eq:pqk-sym} and \eqref{eq:pqk-asym} with $q\neq0$, we shall determine the threshold $K_1=K_1(p,q,K)$ from \eqref{S2}.
		
		For the case $p=0$, from \eqref{eq:pqk-sym} and \eqref{eq:pqk-asym}, we have $(p,q,K)=(0,-2,k_2)$, such that \eqref{S2} reduces to $F(\omega^2)=4k_1^2$. Define 
		\begin{equation*}
			K_1(0,-2,k_2) := \frac{1}{2}\sqrt{\inf_{\omega\ge0}F(\omega^2)} \ge0.
		\end{equation*}
		By Lemma \ref{le.F}, the minimum is attained at $\omega_*$ for $k_2\in(0,k_2^*)$ and $0$ for $k_2\in[k_2^*,\infty)$. Note that $F(\omega^2)\to\infty$ as $\omega\to\infty$. Therefore, \eqref{S2} has no positive solution if $0<k_1<K_1$, whereas it has at least one positive solution whenever $k_1>K_1$. Moreover, $K_1(0,-2,k_2)=0$ if and only if $k_2=k_2^*$. For $k_2>k_2^*$, $F(\omega^2)$ is increasing for $\omega>0$, leading to $K_1(0,-2,k_2)=\sqrt{F(0)}/2=(\rho_0^*-X(0,0,k_2))/2=\xi(k_2)$. Clearly, $K_1(0,-2,k_2)$ is continuous in $k_2\in(0,\infty)$.
		
		For the case $p\neq0$, we have $(p,q,K)=(-1,1,0)$, $(1,-1,0)$, or $(-1,-1,k_2)$, such that $q^2-p^2=0$. 
		
		In $K\in[0,k_2^*)$, we denote
		\begin{equation*}
			\Omega_+ := \{\omega>0: 2pk_1(X(0,\omega^2,K)-\rho_0^*)>0\}=
			\begin{cases}
				\begin{split}
					&(0,\omega_0),&&\text{for $(p,q,K)=(1,-1,0)$},\\
					&(\omega_0,\infty),&&\text{for $(p,q,K)=(-1,1,0)$, $(-1,-1,k_2)$}.
				\end{split}
			\end{cases}
		\end{equation*}
		
		Then \eqref{S2} is equivalent to 
		\begin{equation*}
			k_1=H_1(\omega):=
			\frac{F(\omega^2)}{2p(X(0,\omega^2,K)-\rho_0^*)},\; \omega\in\Omega_+.
		\end{equation*}
		
		Define $K_1(p,q,K) := \inf\{H_1(\omega): \omega\in\Omega_+\}$. If $p=1$, then $H_1(\omega)$ is positive and continuous on $\Omega_+$, and satisfies $H_1(\omega)\to+\infty$ as $\omega\to\omega_0^-$. Thus, \eqref{S2} admits at least one positive solution for any $k_1>K_1$. If $p=-1$, then $(p,q,K)=(-1,1,0)$ or $(-1,-1,k_2)$ with $k_2<k_2^*$. In this case, $H_1(\omega)\to+\infty$ as $\omega\to\omega_0^+$ or $\omega\to\infty$. Thus, $H_1(\omega)$ attains a positive minimum $K_1(p,q,K)$ in $(\omega_0,\infty)$, such that \eqref{S2} admits at least one positive solution for any $k_1>K_1$.
		
		We next consider $K=k_2^*$ and $(p,q,K)=(-1,-1,k_2^*)$. Then, $F(\omega^2)$ and $2pk_1(X(0,\omega^2,k_2^*)-\rho_0^*)$ are positive and increasing in $\omega\in(0,\infty)$. One can check that
		\begin{equation*}
			\lim_{\omega\to0^+}H_1(\omega)=\frac{\left( \alpha-Y(0,0,k_2^*)\right) ^2}{-2\,\frac{\pa X}{\pa\omega^2}(0,0,k_2^*)}>0,\;H_1(\omega)\to+\infty\;\text{as $\omega\to\infty$},
		\end{equation*}
		which implies that $K_1(-1,-1,k_2^*) := \inf\limits_{\omega>0}H_1(\omega)>0$. 
		
		In particular, since $\omega_0(k_2)\to0$ as $k_2\to(k_2^*)^-$, we obtain
		\begin{equation*}
			\lim\limits_{k_2\to(k_2^*)^-}K_1(-1,-1,k_2) = K_1(-1,-1,k_2^*)>0=\lim\limits_{k_2\to(k_2^*)^+}K_1(-1,-1,k_2).
		\end{equation*}
		Moreover, when $k_2=0$, \eqref{S2} is the same one in cases $(p,q,K)=(-1,-1,k_2)$ with $k_2=0$ and $(-1,1,0)$, indicating that $K_1(-1,-1,k_2)|_{k_2=0}=K_1(-1,1,0)$.
		
		Finally, for the case $(p,q,K)=(-1,-1,k_2)$ and $K>k_2^*$, we have $\rho_0^*-X(0,\omega^2,k_2)>\rho_0^*-X(0,0,k_2)>0$, such that for any $\omega>0$,
		\begin{equation*}
			H_1(\omega)=\frac{\rho_0^*-X(0,\omega^2,k_2)}{2}+\frac{\omega^2\left( \alpha-Y(0,\omega^2,k_2)\right) ^2}{2(\rho_0^*-X(0,\omega^2,k_2))}>\frac{\rho_0^*-X(0,0,k_2)}{2}=\xi(k_2).
		\end{equation*}
		
		Combining $\lim\limits_{\omega\to0^+}H_1(\omega)=\xi(k_2)$, we obtain $K_1(-1,-1,k_2)=\xi(k_2)$ for $k_2\in(k_2^*,\infty)$.
		
		If \eqref{S2} admits a unique solution, then we denote it by $\omega_H^*$. Correspondingly, we can obtain \eqref{alp2om} from \eqref{S1}, where the smallest critical value of $\tau$ is denoted by $\tau_H^*$.
		
		While if \eqref{S2} has multiple positive solutions $\omega_1,\cdots,\omega_m$ ($m\ge2$) which satisfy $\omega_1<\omega_2<\cdots<\omega_m$, then by \eqref{S1}, we can obtain the critical values $\tau^{\omega_i,j}$ associated with $\omega_i$ ($i=1,\cdots,m$), $j\in\mathbb{N}_0$. For any two solution pairs $(\omega_i,\tau^{\omega_i,0})$ and $(\omega_j,\tau^{\omega_j,0})$ of \eqref{eq-g1} which satisfy $\omega_i>\omega_j$, we shall show that $\tau^{\omega_i,0}<\tau^{\omega_j,0}$. Then by \eqref{u1: im-a} and Proposition \ref{prop.AB}, we have
		\begin{equation*}
			qk_1\left( \cos(\omega_i\tau^{\omega_i,0})-\cos(\omega_j\tau^{\omega_j,0})\right) = X(0,\omega_i^2,K)-X(0,\omega_j^2,K)<0.
		\end{equation*}
		
		If $q<0$, then we have $\sin(\omega\tau)>0$ and  $\cos(\omega_j\tau^{\omega_j,0})<\cos(\omega_i\tau^{\omega_i,0})$. From \eqref{u1: im-b}, the solution pair $(\omega,\tau)$ satisfies $\sin(\omega\tau)>0$, such that $\omega\tau\in(2n\pi,\pi+2n\pi)$, $n\in\mathbb{N}_0$. Then we have $\omega_i\tau^{\omega_i,0}<\omega_j\tau^{\omega_j,0}$, such that $\frac{\tau^{\omega_i,0}}{\tau^{\omega_j,0}}<\frac{\omega_j}{\omega_i}<1$, leading to $\tau^{\omega_i,0}<\tau^{\omega_j,0}$. While if $q>0$, $\tau^{\omega_i,0}<\tau^{\omega_j,0}$ can be proved similarly. Thus, $\tau_H^*:=\tau^{\omega_m,0}$ is the smallest one among the critical values of $\tau$, such that there is no purely imaginary eigenvalue of \eqref{eq-g1} for $\tau\in(0,\tau_H^*)$, but a unique pair of purely imaginary eigenvalues $\pm i\omega_H^*$ exists at $\tau_H^*$.
	\end{proof}
	
	\begin{lemma}\label{le.trans1}
		Suppose that $\alpha\in(\alpha_0,\infty)$. For any $(p,q,K)$ given in \eqref{eq:pqk-sym} and \eqref{eq:pqk-asym} satisfying $q\neq0$, we assume that $k_1\in(K_1,\infty)$, where $K_1$ is given by Lemma \ref{le.e1}. Let $(\tau_H^*,\omega_H^*)$ be given by Lemma \ref{le.e1}. Then, there exists a $K_2=K_2(p,q,K)$ which satisfies $K_2\ge K_1$, such that for $k_1\in(K_2,\infty)$, system \eqref{SLEP-g} has a unique pair of critical complex eigenvalues $\la_R^*(\tau)\pm i\la_I^*(\tau)$ in $\tau\in(\tau_H^*-\kappa,\tau_H^*+\kappa)$ for small $\kappa>0$, satisfying
		\begin{equation}\label{eq.trans1}
			\la_R^*(\tau_H^*)=0,\; \la_I^*(\tau_H^*)=\omega_H^*,\; \frac{d}{d\tau}\la_R^*(\tau_H^*)>0.
		\end{equation}
		More precisely, $K_2(p,q,K)$ can be characterized as follows:
		\begin{enumerate}
			\item If $(p,q,K)=(0,-2,k_2)$, then \eqref{eq.trans1} holds for $k_1\in(K_2,\infty)$, where $K_2=K_1$ for $k_2\in(k_2^*,\infty)$ and $K_2=\sqrt{F(\omega_0^2)}/2$ for $k_2\in(0,k_2^*]$.
			\item If $(p,q,K)=(-1,-1,k_2)$ and $k_2\in(0,\infty)$, then \eqref{eq.trans1} holds in  $k_1\in(K_2,\infty)$ for some $K_2=K_2(-1,-1,k_2)$. In particular, if $k_2\in(k_2^*,\infty)$, then \eqref{eq.trans1} also holds in $k_1\in(K_1,2K_1)$.
			\item If $(p,q,K)=(1,-1,0)$, then \eqref{eq.trans1} holds for $k_1\in(K_2,\infty)$, where $K_2=X(0,0,0)-\rho_0^*$.
			\item If $(p,q,K)=(-1,1,0)$, then \eqref{eq.trans1} holds for $k_1\in(K_2,\infty)$ for some $K_2=K_2(-1,1,0)$ satisfying $K_2(-1,1,0)=K_2(-1,-1,k_2)|_{k_2=0}$.
		\end{enumerate}
	\end{lemma}
	\begin{proof}
		The complex SLEP system \eqref{F*} is recast as
		\begin{equation*}
			\mathcal{F}^*(\la^*,\tau):=\rho_0^*-\alpha\la^*+k_1(p+q \theta_{\tau})-\sum_{n=0}^{\infty}\frac{c_1^*c_2^*}{\gamma_n+\la^*+2K}\psi_n^2(x^*)=0,
		\end{equation*}
		where $\theta_{\tau}=e^{-\lambda^*\tau}$.
		
		Differentiating $\mathcal{F}^*(\la^*,\tau)=0$ with respect to $\la^*$ and $\tau$ at $(i\omega_H^*,\tau_H^*)$, respectively, we have
		\begin{equation*}
			\begin{split}
				\frac{ \pa\mathcal{F}^*}{\pa\la^*}(i\omega_H^*,\tau_H^*)=-\alpha-k_1 q\tau_H^*e^{-i\omega_H^*\tau_H^*}+\sum_{n=0}^{\infty}\frac{c_1^*c_2^*\psi_n^2(x^*)}{(\gamma_n+i\omega_H^*+2K)^2},\;\frac{ \pa\mathcal{F}^*}{\pa\tau}(i\omega_H^*,\tau_H^*)=-ik_1 q\omega_H^*e^{-i\omega_H^*\tau_H^*},
			\end{split}		
		\end{equation*}
		such that 
		\begin{equation*}
			\begin{split}
				&\operatorname{Re}\frac{ \pa\mathcal{F}^*}{\pa\la^*}(i\omega_H^*,\tau_H^*)=-\alpha-k_1 q \tau_H^*\cos(\omega_H^*\tau_H^*)+\sum_{n=0}^{\infty} \frac{c_1^*c_2^*[(\gamma_n + 2K)^2-(\omega_H^*)^2]}{[(\gamma_n + 2K)^2 + (\omega_H^*)^2]^2}\,\psi_n^2(x^*),\\
				&\operatorname{Im}\frac{ \pa\mathcal{F}^*}{\pa\la^*}(i\omega_H^*,\tau_H^*)=k_1 q \tau_H^*\sin(\omega_H^* \tau_H^*)-\sum_{n=0}^{\infty} \frac{2c_1^*c_2^*\omega_H^*(\gamma_n + 2K)}{[(\gamma_n + 2K)^2 + (\omega_H^*)^2]^2}\,\psi_n^2(x^*),\\
				&\operatorname{Re}\frac{ \pa\mathcal{F}^*}{\pa\tau}(i\omega_H^*,\tau_H^*)=-k_1 q\omega_H^*\sin(\omega_H^*\tau_H^*),\;
				\operatorname{Im}\frac{ \pa\mathcal{F}^*}{\pa\tau}(i\omega_H^*,\tau_H^*)=-k_1 q\omega_H^*\cos(\omega_H^*\tau_H^*).
			\end{split}
		\end{equation*}
		
		Since $\sin(\omega_H^*\tau_H^*)q<0$, we have $\operatorname{Im}\frac{\pa\mathcal{F}^*}{\pa\la^*}(i\omega_H^*,\tau_H^*)<0$. Then by the implicit function theorem, we can obtain the existence and uniqueness of critical eigenvalues $\la_R^*(\tau)\pm i\la_I^*(\tau)$ in $\tau\in(\tau_H^*-\kappa,\tau_H^*+\kappa)$ for small $\kappa>0$, which satisfy $\la_R^*(\tau_H^*)=0$, $\la_I^*(\tau_H^*)=\omega_H^*$.
		
		Next we shall show that $\frac{d}{d\tau}\la_R^*(\tau_H^*)>0$. By implicit differentiation, 
		\begin{equation}\label{eq.tra1}
			\frac{d\la_R^*}{d\tau}(\tau_H^*)=-\frac{
				\operatorname{Re}\frac{ \pa\mathcal{F}^*}{\pa\la^*}(i\omega_H^*,\tau_H^*)
				\operatorname{Re}\frac{ \pa\mathcal{F}^*}{\pa\tau}(i\omega_H^*,\tau_H^*)+
				\operatorname{Im}\frac{ \pa\mathcal{F}^*}{\pa\la^*}(i\omega_H^*,\tau_H^*)
				\operatorname{Im}\frac{ \pa\mathcal{F}^*}{\pa\tau}(i\omega_H^*,\tau_H^*)}{\left(\operatorname{Re}\frac{ \pa\mathcal{F}^*}{\pa\la^*}(i\omega_H^*,\tau_H^*)\right)^2+\left(\operatorname{Im}\frac{ \pa\mathcal{F}^*}{\pa\la^*}(i\omega_H^*,\tau_H^*)\right)^2}.
		\end{equation}				
		
		After some calculations, it is equivalent to show that $I_1+I_2>0$, where
		\begin{equation*}
			I_1=\alpha-\sum_{n=0}^{\infty} \frac{c_1^*c_2^*[(\gamma_n + 2K)^2-(\omega_H^*)^2]}{[(\gamma_n + 2K)^2 + (\omega_H^*)^2]^2}\,\psi_n^2(x^*),\;
			I_2=\sum_{n=0}^{\infty} \frac{2c_1^*c_2^*\omega_H^*(\gamma_n + 2K)}{[(\gamma_n + 2K)^2 + (\omega_H^*)^2]^2}\,\psi_n^2(x^*) \frac{\cos(\omega_H^*\tau_H^*)}{\sin(\omega_H^*\tau_H^*)}.
		\end{equation*}
		
		On the one hand, we have
		\begin{equation}\label{alp-}
			\begin{split}
				I_1> \alpha-\sum_{n=0}^{\infty} \frac{c_1^*c_2^*[(\gamma_n + 2K)^2+(\omega_H^*)^2]}{[(\gamma_n + 2K)^2 + (\omega_H^*)^2]^2}\,\psi_n^2(x^*)=\alpha-Y(0,(\omega_H^*)^2,K)>\alpha-\alpha_0>0.
			\end{split}   		
		\end{equation}
		
		On the other hand, it is sufficient to show that $\cos(\omega_H^*\tau_H^*)\sin(\omega_H^*\tau_H^*)>0$ or $I_2\to0$ as $k_1\to\infty$.
		
		To begin with, from \eqref{u1: im-b}, it follows that $\sin(\omega_H^*\tau_H^*)>0$ (resp., $<0$) and $\omega_H^*\tau_H^*\in(0,\pi)$ (resp., $(\pi,2\pi)$) if $q<0$ (resp., $q>0$). 
		
		Next we verify the transversality condition for the following cases.
		
		1. $(p,q,K)=(0,-2,k_2)$: By \eqref{S1}, we have $\cos(\omega_H^*\tau_H^*)=\frac{\rho_0^*-X(0,(\omega_H^*)^2,k_2)}{2k_1}$. For $k_2\in[k_2^*,\infty)$, it is clear that $\cos(\omega_H^*\tau_H^*)>0$ holds for any $k_1\in(K_1,\infty)$, such that $K_2=K_1$. While for $k_2\in(0,k_2^*)$, we note that $\omega=\omega_0$ is the unique solution of $\rho_0^*=X(0,\omega^2,k_2)$ and $F(\omega^2)$ is increasing in $\omega\in(\omega_0,\infty)$. Thus, if  $k_1>K_2:=\sqrt{F(\omega_0^2)}/2$, then $\omega_H^*>\omega_0$, such that $\cos(\omega_H^*\tau_H^*)>0$.
		
		2. $(p,q,K)=(-1,-1,k_2)$: By \eqref{S1}, we have $\cos(\omega_H^*\tau_H^*)=\frac{\rho_0^*-k_1-X(0,(\omega_H^*)^2,k_2)}{k_1}$. For $k_2\in(k_2^*,\infty)$, we have
		\begin{equation*}
			\rho_0^*-k_1-X(0,(\omega_H^*)^2,k_2)>\rho_0^*-k_1-X(0,0,k_2)=2K_1-k_1,
		\end{equation*}
		implying that $\cos(\omega_H^*\tau_H^*)>0$ if $k_1\in(K_1,2K_1)$.
		
		3. $(p,q,K)=(1,-1,0)$: By \eqref{S1}, we have $\cos(\omega_H^*\tau_H^*)=\frac{\rho_0^*+k_1-X(0,(\omega_H^*)^2,0)}{k_1}$. Clearly, if $k_1>K_2:=X(0,0,0)-\rho_0^*$, then $\rho_0^*+k_1-X(0,(\omega_H^*)^2,0)>0$, such that $\cos(\omega_H^*\tau_H^*)>0$.
		
		For the remaining cases, the transversality follows from $I_2\to0$ as $k_1\to\infty$, which are verified as follows.
		
		1. $(p,q,K)=(-1,1,0)$: From \eqref{eq-g1}, it follows that
		\begin{equation*}
			\begin{split}
				I_2&=\sum_{n=0}^{\infty} \frac{2c_1^*c_2^*\omega_H^*\gamma_n }{[\gamma_n ^2 + (\omega_H^*)^2]^2}\,\psi_n^2(x^*) \frac{k_1 q \cos(\omega_H^*\tau_H^*)}{k_1 q \sin(\omega_H^*\tau_H^*)}\\
				&=\sum_{n=0}^{\infty} \frac{2c_1^*c_2^*(\omega_H^*)^2\gamma_n }{[\gamma_n^2 + (\omega_H^*)^2]^2}\,\psi_n^2(x^*) \frac{X(0,(\omega_H^*)^2,0)-\rho_0^*+k_1}{(\omega_H^*)^2(Y(0,(\omega_H^*)^2,0)-\alpha)},
			\end{split}		
		\end{equation*}
		where, due to \eqref{S2},
		\begin{equation*}
			k_1=\frac{(\rho_0^*-X(0,(\omega_H^*)^2,0))^2+(\omega_H^*)^2(\alpha-Y(0,(\omega_H^*)^2,0))^2}{2(\rho_0^*-X(0,(\omega_H^*)^2,0))}.
		\end{equation*} 
		
		Then from \eqref{eq.XY1} and \eqref{eq.XY2}, combining $\omega_H^*\to\infty$ as $k_1\to\infty$, we have
		\begin{equation*}
			0\le\lim\limits_{\omega_H^*\to\infty}\vert I_2\vert\le \lim\limits_{\omega_H^*\to\infty}\frac{\alpha}{\rho_0^*} X(0,(\omega_H^*)^2,0)=0.
		\end{equation*}
		
		2. For $(p,q,K)=(-1,-1,k_2)$,  $\lim\limits_{k_1\to\infty} I_2=0$ can be shown similarly to the case of $(p,q,K)=(-1,1,0)$. Clearly, $K_2(-1,-1,k_2)|_{k_2=0}=K_2(-1,1,0)$. The proof is complete.
	\end{proof}
	
	\begin{corollary}\label{cor.K3}
		Let $(\tau_1^*,\omega_1^*)$, $(\tau_2^*,\omega_2^*)$, $(\tau_3^*,\omega_3^*)$ be the critical points $(\tau_H^*,\omega_H^*)$ given by Lemma \ref{le.trans1} for the cases of $(p,q,K)=(-1,1,0)$, $(-1,-1,k_2)$, $(1,-1,0)$, respectively. For any $k_2\in(0,\infty)$, there exists a sufficiently large $\mathcal{K}_3=\mathcal{K}_3(k_2)$, such that when $k_1>\mathcal{K}_3$, we have $\tau_3^*<\tau_2^*<\tau_1^*$.
	\end{corollary}
	\begin{proof}
		1. From \eqref{S2}, we obtain
		\begin{equation}\label{ho/o}
			\frac{\left[ \rho_0^*-X(0,(\omega_1^*)^2,0)\right] ^2+ (\omega_1^*)^2\left[Y(0,(\omega_1^*)^2,0)-\alpha\right]^2}{\left[ \rho_0^*-X(0,(\omega_2^*)^2,k_2)\right] ^2+ (\omega_2^*)^2\left[Y(0,(\omega_2^*)^2,k_2)-\alpha\right]^2}=\frac{2k_1(\rho_0^*-X(0,(\omega_1^*)^2,0))}{2k_1(\rho_0^*-X(0,(\omega_2^*)^2,k_2))}.
		\end{equation}
		
		As $k_1\to\infty$, by the proof of Lemma \ref{le.e1}, we have $\omega_1^*\to\infty$ and $\omega_2^*\to\infty$, leading to $\lim\limits_{k_1\to\infty}\omega_1^*/\omega_2^*=1$ from \eqref{ho/o}. Then by \eqref{eq-g1}, we have  $\sin(\omega_1^*\tau_1^*)<0$ and $\sin(\omega_2^*\tau_2^*)>0$ (leading to $\omega_1^*\tau_1^*\in(\pi,2\pi)$ and $\omega_2^*\tau_2^*\in(0,\pi)$), as well as $\cos(\omega_1^*\tau_1^*)\to 1$ and $\cos(\omega_2^*\tau_2^*)\to -1$ as $k_1\to\infty$. Thus, we obtain $\omega_1^*\tau_1^*\to 2\pi$ and $\omega_2^*\tau_2^*\to \pi$, such that $\lim\limits_{k_1\to\infty}\tau_1^*/\tau_2^*=2$.
		
		2. By the proof of Lemma \ref{le.e1}, we have $\omega_3^*\to\wh\omega_0$ (as $k_1\to\infty$), where $\wh\omega_0$ satisfies $X(0,\wh\omega_0^2,0)=\rho_0^*$. Further, from \eqref{eq-g1}, we obtain $\omega_3^*\tau_3^*\in(0,\pi)$, $\cos(\omega_3^*\tau_3^*)\to1$, and $\sin(\omega_3^*\tau_3^*)\to0$, such that $\omega_3^*\tau_3^*\to 0$. Then by \eqref{eq-g1}, combining $\lim\limits_{x\to0}\sin(x)/x=1$, we have 
		\begin{equation*}
			1=\lim\limits_{k_1\to\infty}\frac{\sin(\omega_3^*\tau_3^*)}{\omega_3^*\tau_3^*}=\lim\limits_{k_1\to\infty}\frac{k_1\sin(\omega_3^*\tau_3^*)}{k_1\omega_3^*\tau_3^*}=\frac{\alpha-Y(0,\wh\omega_0^2,0)}{\lim\limits_{k_1\to\infty} k_1 \tau_3^*},
		\end{equation*}
		implying that $\tau_3^*\sim C_1/k_1$ for sufficiently large $k_1$ with $C_1>0$.
		
		On the other hand, from \eqref{S2}, we have $\omega_2^*\sim C_2\sqrt{k_1}$. Then by $\omega_2^*\tau_2^*\to\pi$, we derive $\tau_2^*\sim C_3/\sqrt{k_1}$ for sufficiently large $k_1$, where $C_i>0$, $i=1,2,3$. Hence, $\tau_3^*<\tau_2^*<\tau_1^*$ for $k_1>\mathcal{K}_3$ for some sufficiently large $\mathcal{K}_3=\mathcal{K}_3(k_2)$.
	\end{proof}
	
	\subsection{Stability and Hopf bifurcations for Cases 1--3}
	In this subsection, we present the stability and bifurcation results for Cases 1--3 with sufficiently small $\ep>0$, which are given by
	\begin{align}
		\makebox[2.2cm][l]{\text{(Case 1)}}\quad
		&\left\{
		\begin{aligned}
			\alpha\frac{\partial u_i}{\partial t}
			&= \ep\frac{\partial^2 u_i}{\partial x^2}
			+ \frac{1}{\ep}f(u_i,v_i)
			+ k_1\bigl(u_j(x,t-\tau)-u_i(x,t-\tau)\bigr),\\
			\frac{\partial v_i}{\partial t}
			&= d\frac{\partial^2 v_i}{\partial x^2}
			+ g(u_i,v_i)+k_2(v_j-v_i),
		\end{aligned}
		\right.
		\label{eq-d1}\\[1.5ex]
		\makebox[2.2cm][l]{\text{(Case 2)}}\quad
		&\left\{
		\begin{aligned}
			\alpha\frac{\partial u_i}{\partial t}
			&= \ep\frac{\partial^2 u_i}{\partial x^2}
			+ \frac{1}{\ep}f(u_i,v_i)
			+ k_1\bigl(u_j(x,t-\tau)-u_i\bigr),\\
			\frac{\partial v_i}{\partial t}
			&= d\frac{\partial^2 v_i}{\partial x^2}
			+ g(u_i,v_i)+k_2(v_j-v_i),
		\end{aligned}
		\right.
		\label{eq-d2}\\[1.5ex]
		\makebox[2.2cm][l]{\text{(Case 3)}}\quad
		&\left\{
		\begin{aligned}
			\alpha\frac{\partial u_i}{\partial t}
			&= \ep\frac{\partial^2 u_i}{\partial x^2}
			+ \frac{1}{\ep}f(u_i,v_i)
			+ k_1\bigl(u_j-u_i(x,t-\tau)\bigr),\\
			\frac{\partial v_i}{\partial t}
			&= d\frac{\partial^2 v_i}{\partial x^2}
			+ g(u_i,v_i)+k_2(v_j-v_i),
		\end{aligned}
		\right.
		\label{eq-d3}
	\end{align}
	respectively, where $\tau\in(0,\infty)$, $i,j\in\{1,2\}$, $i\neq j$; $x\in(0,1)$, $t>0$.
	
	For convenience, we denote by $\mathcal{K}_i^{j,\sigma}$  the $i$-th threshold of Case $j$ associated with the $\sigma$ sub-problem, where $j\in\{1,2,3\}$, $\sigma\in\{s,a\}$ denotes symmetric or asymmetric sub-problem, $i=1,2$. In particular,  $\mathcal{K}_i^{2,a}(k_2)=\mathcal{K}_i^{3,a}(k_2)$ for any $k_2\in(0,\infty)$, $i=1,2$.
	
	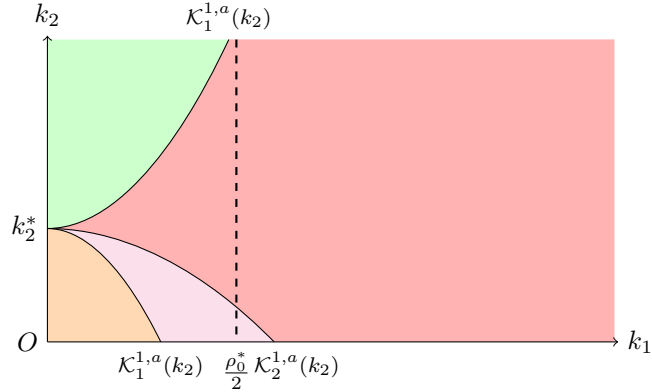
\begin{figure}[htbp]
		\centering
		\begin{tikzpicture}
			\fill [fill=green!20]  (0,4) --  (0,1.5) parabola (2.4,4)--cycle;
			\fill [fill=magenta!15]  (0,0) --  (0,1.5) parabola (2.4,4)--(7.5,4)--(7.5,0)--cycle;
			\fill [red!30]  (0,1.5) parabola (2.4,4)--(7.5,4)--(7.5,0)--(3,0) parabola bend (0,1.5) (0,1.5);
			\fill [orange!30]  (0,1.5) parabola (1.5,0)--(0,0) --cycle;
			\draw[->] (0,0) --(7.55,0) node[right] {$k_1$};
			\draw[->] (0,0) --(0,4.05) node[above] {$k_2$};
			\draw[dashed,thick] (2.5,4)--(2.5,0) node [below] {$\frac{\rho_0^*}{2}$};
			\draw (0,1.5) node[left] {$k_2^*$};
			\draw (0,0) node[left] {$O$};
			\draw  (0,1.5) parabola (2.4,4);		
			\draw (2.4,4.3) node[font=\footnotesize]{$\mathcal{K}_1^{1,a}(k_2)$};
			\draw  (0,1.5) parabola (1.5,0);
			\node[font=\footnotesize,below] at (1.5,0)  {$\mathcal{K}_1^{1,a}(k_2)$};
			\draw  (0,1.5) parabola (3,0);
			\node[font=\footnotesize,below] at (3.3,0)  {$\mathcal{K}_2^{1,a}(k_2)$};
		\end{tikzpicture}
		\caption{Stability diagram of the synchronized layered steady state $(U(x,\ep),U(x,\ep))$ with respect to system \eqref{eq-d1} in the $(k_1,k_2)$ plane. Green: instability region  $\Gamma_1$; the curve $\Gamma$ separates $\Gamma_1$ from $\Gamma_2$ (Theorem \ref{th.2-1}). Orange: stability region $\Gamma_S$. Red: Hopf bifurcation region $\Gamma_H\subset\Gamma_2$ specified by \eqref{H1}.}
		\label{Fig.s-dep1}
	\end{figure}
	
	\begin{theorem}[Case 1]\label{th-1} 
		Suppose that $\alpha\in(\alpha_0,+\infty)$. Let $\Gamma_1,\Gamma_2$ be defined by Theorem \ref{th.2-1}, $K_1(p,q,K)$ and $K_2(p,q,K)$ be given by Lemmas \ref{le.e1} and \ref{le.trans1}, respectively. Define $\mathcal{K}_i^{1,a}(k_2):=K_i(0,-2,k_2)$, $i=1,2$. Then, we have the following results $($see Fig. \ref{Fig.s-dep1}$)$: 
		\begin{enumerate} 
			\item $($Turing instability$)$ If $(k_1,k_2)\in\Gamma_1$ holds, then for any fixed  $\tau\in(0,\infty)$ and sufficiently small $\ep>0$, the synchronized layered steady state $(U(x,\ep),U(x,\ep))$ is unstable with respect to system \eqref{eq-d1}. 
			\item $($Delay-independent stability$)$ If $(k_1,k_2)\in\Gamma_S:=\{(k_1,k_2)\in\Gamma_2: k_2\in(0,k_2^*), k_1\in(0,\mathcal{K}_1^{1,a}(k_2))\}$ holds, then for any fixed  $\tau\in(0,\infty)$ and sufficiently small $\ep>0$, $(U(x,\ep),U(x,\ep))$ is stable with respect to system \eqref{eq-d1}. 
			\item $($Delay-induced Hopf bifurcation$)$ If $(k_1,k_2)\in\Gamma_2$ and satisfies 
			\begin{equation}\label{H1} 
				\begin{split} 
					&k_2\in(0,k_2^*], k_1\in(\mathcal{K}_2^{1,a}(k_2),+\infty),\;\text{or}\\ 
					&k_2\in(k_2^*,\infty), k_1\in(\mathcal{K}_1^{1,a}(k_2),+\infty), 
				\end{split} 
			\end{equation} 
			then there exists a $\tau_H(\ep)>0$ for sufficiently small $\ep>0$ with $\tau_H(0)=\tau_H^*$, where $\tau_H^*$ is determined from \eqref{eq-g1} for the case of $(p,q,K)=(0,-2,k_2)$ by Lemma \ref{le.e1}, such that $(U(x,\ep),U(x,\ep))$ is stable with respect to the system \eqref{eq-d1} when $\tau\in 
			(0,\tau_H(\ep))$ while unstable when $\tau\in(\tau_H(\ep),\tau_H(\ep)+\kappa)$ for some $\kappa>0$. Moreover, system \eqref{eq-d1} undergoes a primary Hopf bifurcation at $\tau=\tau_H(\ep)$, such that a family of anti-phase bifurcating breathers of \eqref{eq-d1} occurs around $(U(x,\ep),U(x,\ep))$.  
		\end{enumerate}	 
	\end{theorem} 
	\begin{proof}
		We first recall the argument common to all six cases. For $(k_1,k_2)\in\Gamma_1$, Proposition \ref{p.pzp} gives a positive real eigenvalue of the asymmetric SLEP system \eqref{SLEP-g-2} for every $\tau>0$, and Theorem \ref{p.con} gives an eigenvalue of \eqref{EP-o-2} with positive real part, leading to the instability of the synchronized layered steady state with respect to system \eqref{eq-o} for any fixed $\tau>0$ and sufficiently small $\ep>0$. For $(k_1,k_2)\in\Gamma_2$, both the symmetric and the asymmetric SLEP systems are stable in the limiting sense as $\tau\to0^+$ and have no zero eigenvalues for $\tau>0$.
		
		For Case 1, the symmetric and asymmetric SLEP systems are given by \eqref{SLEP-g-1} and \eqref{SLEP-g-2}, which correspond to $(p,q,K)=(0,0,0)$ and $(0,-2,k_2)$, respectively. In particular, system \eqref{SLEP-g-1} is independent of $\tau$ and remains stable by Proposition \ref{p.pzp}. For $(k_1,k_2)\in\Gamma_S$, the existence of purely imaginary eigenvalues of the asymmetric SLEP system is excluded by Lemma \ref{le.e1}. Under \eqref{H1}, Lemmas \ref{le.e1} and \ref{le.trans1} give a unique simple conjugate pair at the first critical delay $\tau_H^*$, crossing the imaginary axis transversely from left to right. Theorem \ref{p.con} further shows the existence, simplicity, and the transversality of critical complex eigenvalues of asymmetric sub-problem \eqref{EP-o-2} for sufficiently small $\ep>0$ near $\tau_H(\ep)$ with $\tau_H(0)=\tau_H^*$. Since the symmetric sub-problem remains stable, Corollary \ref{cor.hopf} provides the stability of the synchronized layered steady state $(U(x,\ep),U(x,\ep))$, the existence of primary Hopf bifurcation through asymmetric sub-problem, and corresponding anti-phase breathers.
	\end{proof}
	
	Cases 2 and 3 share the same asymmetric sub-problem but differ in their symmetric sub-problems. The next two theorems state the sufficient conditions ensuring the existence of delay-induced instability and identifying the first critical mode.
	
	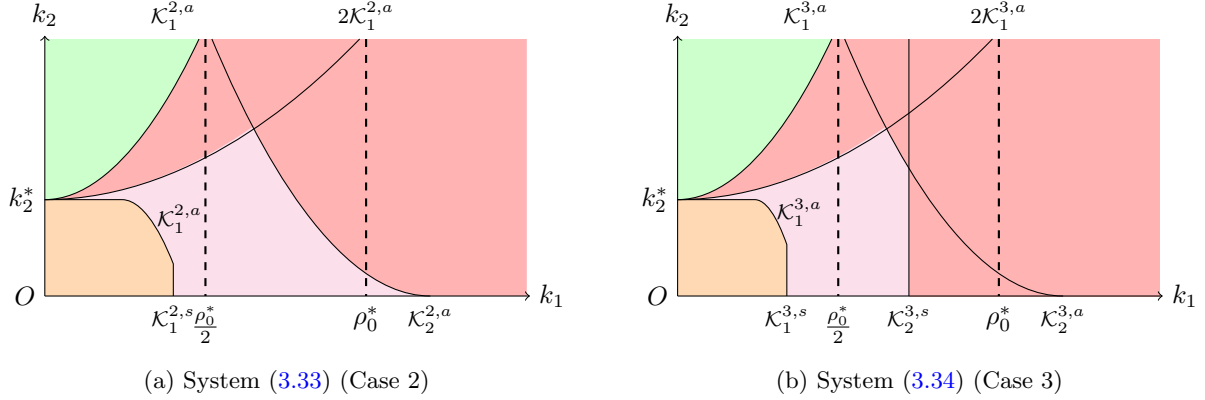
\begin{figure}[htbp]
		\centering
		\begin{subfigure}[t]{0.48\linewidth}
			\centering
			\begin{tikzpicture}[scale=0.85]
				\fill [fill=green!20]  (0,4) --  (0,1.5) parabola (2.4,4)--cycle;
				\fill [fill=magenta!15]  (0,0) --  (0,1.5) parabola (2.4,4)--(7.5,4)--(7.5,0)--cycle;
				\fill [red!30]  (0,1.5) parabola (4.8,4)--(2.4,4) parabola bend (0,1.5) (0,1.5);
				\fill [red!30]  (6,0) parabola (2.6,4)--(7.5,4)--(7.5,0)--cycle;
				\fill [orange!30] (0,1.5)-- (1.2,1.5) parabola (2,0.5)--(2,0)--(0,0)--(0,1.5)--cycle;
				\draw[->] (0,0) --(7.55,0) node[right] {$k_1$};
				\draw[->] (0,0) --(0,4.05) node[above] {$k_2$};
				\draw[dashed,thick] (2.5,4)--(2.5,0) node [below] {$\frac{\rho_0^*}{2}$};
				\draw[dashed,thick] (5,4)--(5,0)  node [below] {$\rho_0^*$};
				\draw (0,1.5) node[left] {$k_2^*$};
				\draw (0,0) node[left] {$O$};
				\draw  (0,1.5) parabola (2.4,4);
				\draw  (0,1.5) parabola (4.9,4);
				\draw  (6,0) parabola (2.6,4);
				\draw  (0,1.5)--(1.2,1.5) parabola (2,0.5)--(2,0);
				\node at (2,0)  [font=\footnotesize,below] {$\mathcal{K}_1^{2,s}$};
				\node[font=\footnotesize,right] at (1.6,1.2) {$\mathcal{K}_1^{2,a}$};
				\node[font=\footnotesize,above] at (2,4) {$\mathcal{K}_1^{2,a}$};
				\node[font=\footnotesize,below] at (6,0) {$\mathcal{K}_2^{2,a}$};
				\node[font=\footnotesize,above] at (5,4) {$2\mathcal{K}_1^{2,a}$};
			\end{tikzpicture}
			\caption{System \eqref{eq-d2} (Case 2)}
			\label{Fig.s-dep2}
		\end{subfigure}\;\;\;\;
		\begin{subfigure}[t]{0.48\linewidth}
			\centering
			\begin{tikzpicture}[scale=0.85]
				\fill [fill=green!20]  (0,4) --  (0,1.5) parabola (2.4,4)--cycle;
				\fill [fill=magenta!15]  (0,0) --  (0,1.5) parabola (2.4,4)--(7.5,4)--(7.5,0)--cycle;
				\fill [red!30]  (0,1.5) parabola (4.8,4)--(2.4,4) parabola bend (0,1.5) (0,1.5);
				\fill [red!30]  (6,0) parabola (2.6,4)--(7.5,4)--(7.5,0)--cycle;
				\fill [red!30]  (3.6,4)--(7.5,4)--(7.5,0)--(3.6,0)--cycle;
				\fill [orange!30] (0,1.5)-- (1.2,1.5) parabola (1.7,0.8)--(1.7,0)--(0,0)--(0,1.5)--cycle;
				\draw[->] (0,0) --(7.55,0) node[right] {$k_1$};
				\draw[->] (0,0) --(0,4.05) node[above] {$k_2$};
				\draw[dashed,thick] (2.5,4)--(2.5,0) node [below] {$\frac{\rho_0^*}{2}$};
				\draw[dashed,thick] (5,4)--(5,0) node [below] {$\rho_0^*$};
				\draw (0,1.5) node[left] {$k_2^*$};
				\draw (0,0) node[left] {$O$};
				\draw (0,1.5) parabola (2.4,4);
				\draw (0,1.5) parabola (4.9,4);
				\draw (6,0) parabola (2.6,4);
				\draw (0,1.5)--(1.2,1.5) parabola (1.7,0.8)--(1.7,0);
				\draw (3.6,4) -- (3.6,0) node [font=\footnotesize,below] {$\mathcal{K}_2^{3,s}$};
				\node at (1.7,0)  [font=\footnotesize,below] {$\mathcal{K}_1^{3,s}$};
				\node[font=\footnotesize,right] at (1.4,1.3) {$\mathcal{K}_1^{3,a}$};
				\node[font=\footnotesize,above] at (2,4) {$\mathcal{K}_1^{3,a}$};
				\node[font=\footnotesize,below] at (6,0) {$\mathcal{K}_2^{3,a}$};
				\node[font=\footnotesize,above] at (5,4) {$2\mathcal{K}_1^{3,a}$};
			\end{tikzpicture}
			\caption{System \eqref{eq-d3} (Case 3)}
			\label{Fig.s-dep3}
		\end{subfigure}
		\caption{Stability diagrams of the synchronized layered steady state $(U(x,\ep),U(x,\ep))$ in the $(k_1,k_2)$ plane. (a) Case 2, system \eqref{eq-d2}; (b) Case 3, system \eqref{eq-d3}. Green: delay-independent Turing instability region $\Gamma_1$; the curve $\Gamma$ separates $\Gamma_1$ from $\Gamma_2$ (Theorem \ref{th.2-1}). Orange: sufficient region $\Gamma_S$ for delay-independent stability. Red: sufficient delay-induced instability regions in $\Gamma_2$ defined by \eqref{H2} and \eqref{H3}, respectively, where at least one pair of eigenvalues crosses the imaginary axis transversally from left to right as the delay increases, in either the symmetric or asymmetric sub-problem. Sufficient conditions for primary Hopf bifurcation and the associated critical mode are given in \eqref{pri2} and \eqref{pri3}, respectively.}
		\label{Fig.s-dep23}
	\end{figure}
	
	\begin{theorem}[Case 2]\label{th-2} 
		Suppose that $\alpha\in(\alpha_0,+\infty)$. Let $\Gamma_1,\Gamma_2$ be defined by Theorem \ref{th.2-1},  $K_1(p,q,K)$ and $K_2(p,q,K)$ be given by Lemmas \ref{le.e1} and \ref{le.trans1}, respectively. Let $\tau_1^*$, $\tau_2^*$, and $\mathcal{K}_3=\mathcal{K}_3(k_2)$ be stated by Corollary \ref{cor.K3}. Define $\mathcal{K}_i^{2,s}:=K_i(-1,1,0)$ and $\mathcal{K}_i^{2,a}(k_2):=K_i(-1,-1,k_2)$, $i=1,2$. Then, we have the following results $($see Figure \ref{Fig.s-dep2}$)$: 
		\begin{enumerate} 
			\item $($Turing instability$)$ If $(k_1,k_2)\in\Gamma_1$, then for any fixed  $\tau\in(0,\infty)$ and sufficiently small $\ep>0$, the synchronized layered steady state $(U(x,\ep),U(x,\ep))$ is unstable with respect to system \eqref{eq-d2}. 
			\item $($Delay-independent stability$)$ If 
			\begin{equation*}
				(k_1,k_2)\in\Gamma_S:=\{(k_1,k_2)\in\Gamma_2: k_2\in(0,k_2^*], k_1\in(0,\min\{\mathcal{K}_1^{2,a}(k_2),\mathcal{K}_1^{2,s}\})\},
			\end{equation*}
			then for any fixed $\tau\in(0,\infty)$ and sufficiently small $\ep>0$, $(U(x,\ep),U(x,\ep))$ is stable with respect to system \eqref{eq-d2}. 
			\item $($Delay-induced Hopf bifurcation$)$ If $(k_1,k_2)\in\Gamma_2$ and satisfies 
			\begin{equation}\label{H2} 
				\begin{split} 
					&k_2\in(0,k_2^*], k_1\in(\min\{\mathcal{K}_2^{2,s},\mathcal{K}_2^{2,a}(k_2)\},\infty),\;\text{or}\\ 
					&k_2\in(k_2^*,\infty), k_1\in(\mathcal{K}_1^{2,a}(k_2),2\mathcal{K}_1^{2,a}(k_2))\cup(\min\{\mathcal{K}_2^{2,s},\mathcal{K}_2^{2,a}(k_2)\},\infty), 
				\end{split} 
			\end{equation} 
			then there exists a $\tau_H(\ep)>0$ for sufficiently small $\ep>0$, such that eigenvalue problem \eqref{D-L-2-a} admits a pair of critical eigenvalues which crosses the imaginary axis from left to right transversally at $\tau=\tau_H(\ep)$. Under the following additional conditions, $\tau_H(\ep)$ is the first critical delay, with
			\begin{equation}\label{pri2} 
				\tau_H(0)= 
				\begin{cases} 
					\begin{split} 
						&\tau_1^*,&&\text{if $k_2\in(0,k_2^*]$ and $k_1\in(\mathcal{K}_2^{2,s},\mathcal{K}_1^{2,a}(k_2))$ $($provided $\mathcal{K}_2^{2,s}<\mathcal{K}_1^{2,a}(k_2)$$)$},\\ 
						&\tau_2^*,&&\text{if $k_2\in(0,\infty)$ and $k_1\in(\mathcal{K}_3(k_2),\infty)$, or}\\
						&\;&&\;\;\;\;\text{$k_2\in(k_2^*,\infty)$ and $k_1\in(\mathcal{K}_1^{2,a}(k_2),2\mathcal{K}_1^{2,a}(k_2))\cap(0,\mathcal{K}_1^{2,s})$}. 
					\end{split} 
				\end{cases} 
			\end{equation} 
			Under one of the conditions in \eqref{pri2}, system \eqref{eq-d2} undergoes a primary Hopf bifurcation at $\tau_H(\ep)$. Namely, the synchronized layered steady state $(U(x,\ep),U(x,\ep))$ is stable for $\tau\in(0,\tau_H(\ep))$ and unstable when $\tau\in(\tau_H(\ep),\tau_H(\ep)+\kappa)$ for some $\kappa>0$. Moreover, if $\tau_H(0)=\tau_1^*$ $($resp., $\tau_2^*$$)$, then system \eqref{eq-d2} admits a family of in-phase $($resp., anti-phase$)$ breathers bifurcating at $\tau=\tau_H(\ep)$.   
		\end{enumerate} 
	\end{theorem} 
	
	\begin{proof}
		For Case 2, the symmetric and asymmetric SLEP sub-problems correspond to $(p,q,K)=(-1,1,0)$ and $(-1,-1,k_2)$, respectively. The instability or stability in $\Gamma_1$ or $\Gamma_S$ can be obtained similarly to Theorem \ref{th-1}. Following Lemmas \ref{le.e1} and \ref{le.trans1}, together with Theorem \ref{p.con}, we can directly obtain the sufficient condition \eqref{H2} ensuring the existence of delay-induced instability through at least one of symmetric or asymmetric sub-problems for sufficiently small $\ep>0$. 
		
		Further, we shall identify the first critical mode. If $k_2\in(0,\infty)$ and $k_1\in(\mathcal{K}_3(k_2),\infty)$, then by Corollary \ref{cor.K3}, we have $\tau_2^*<\tau_1^*$. Moreover, if $k_2\in(k_2^*,\infty)$ and $k_1\in(\mathcal{K}_1^{2,a}(k_2),2\mathcal{K}_1^{2,a}(k_2))\cap(0,\mathcal{K}_1^{2,s})$, then SLEP system \eqref{eq-g1} has purely imaginary eigenvalues only in asymmetric mode. Thus, the primary Hopf bifurcation is triggered by the asymmetric sub-problem \eqref{EP-o-2} at some $\tau_H(\ep)>0$ for sufficiently small $\ep>0$ with $\tau_H(0)=\tau_2^*$, which generates a family of anti-phase breathers from Corollary \ref{cor.hopf}. Similarly, if $k_2\in(0,k_2^*]$, $\mathcal{K}_2^{2,s}<\mathcal{K}_1^{2,a}(k_2)$ (see Remark \ref{re.k2k1}), and $k_1\in(\mathcal{K}_2^{2,s},\mathcal{K}_1^{2,a}(k_2))$, then the Hopf bifurcation is only induced by the symmetric sub-problem \eqref{EP-o-1} at some $\tau_H(\ep)>0$ for sufficiently small $\ep>0$ with $\tau_H(0)=\tau_1^*$, which triggers a family of in-phase breathers from Corollary \ref{cor.hopf}.
	\end{proof}
	
	\begin{theorem}[Case 3]\label{th-3} 
		Suppose that $\alpha\in(\alpha_0,+\infty)$. Let $\Gamma_1,\Gamma_2$ be defined by Theorem \ref{th.2-1},  $K_1(p,q,K)$ and $K_2(p,q,K)$ be given by Lemmas \ref{le.e1} and \ref{le.trans1}, respectively. Let $\tau_2^*$, $\tau_3^*$, and $\mathcal{K}_3=\mathcal{K}_3(k_2)$ be stated by Corollary \ref{cor.K3}. Define $\mathcal{K}_i^{3,s}:=K_i(1,-1,0)$ and $\mathcal{K}_i^{3,a}(k_2):=K_i(-1,-1,k_2)$, $i=1,2$. Then, we have the following results $($see Figure \ref{Fig.s-dep3}$)$: 
		\begin{enumerate} 
			\item $($Turing instability$)$ If $(k_1,k_2)\in\Gamma_1$, then for any fixed  $\tau\in(0,\infty)$ and sufficiently small $\ep>0$, the synchronized layered steady state $(U(x,\ep),U(x,\ep))$ is unstable with respect to system \eqref{eq-d3}. 
			\item $($Delay-independent stability$)$ If 
			\begin{equation*}
				(k_1,k_2)\in\Gamma_S:=\{(k_1,k_2)\in\Gamma_2: k_2\in(0,k_2^*], k_1\in(0,\min\{\mathcal{K}_1^{3,a}(k_2),\mathcal{K}_1^{3,s}\})\},
			\end{equation*}
			then for any fixed  $\tau\in(0,\infty)$ and sufficiently small $\ep>0$, $(U(x,\ep),U(x,\ep))$ is stable with respect to system \eqref{eq-d3}. 
			\item $($Delay-induced Hopf bifurcation$)$ If $(k_1,k_2)\in\Gamma_2$ and satisfies 
			\begin{equation}\label{H3} 
				\begin{split} 
					&k_2\in(0,k_2^*], k_1\in(\min\{\mathcal{K}_2^{3,a}(k_2),\mathcal{K}_2^{3,s}\},+\infty),\;\text{or}\\ 
					&k_2\in(k_2^*,\infty), k_1\in(\mathcal{K}_1^{3,a}(k_2),2\mathcal{K}_1^{3,a}(k_2))\cup(\min\{\mathcal{K}_2^{3,a}(k_2),\mathcal{K}_2^{3,s}\},\infty), 
				\end{split}			 
			\end{equation} 
			then there exists a $\tau_H(\ep)>0$ for sufficiently small $\ep>0$, such that eigenvalue problem \eqref{D-L-2-a} admits a pair of critical eigenvalues which crosses the imaginary axis from left to right transversally at $\tau=\tau_H(\ep)$. Under the following additional conditions, $\tau_H(\ep)$ is the first critical delay, with 
			\begin{equation}\label{pri3} 
				\tau_H(0)= 
				\begin{cases} 
					\begin{split} 
						&\tau_3^*,&&\text{if $k_2\in(0,\infty)$ and $k_1\in( \mathcal{K}_3(k_2),\infty)$, or}\\
						&\;&&\;\;\text{ $k_2\in(0,k_2^*]$ and $k_1\in(\mathcal{K}_2^{3,s}, \mathcal{K}_1^{3,a}(k_2))$ $($provided $\mathcal{K}_2^{3,s}< \mathcal{K}_1^{3,a}(k_2)$$)$},\\ 
						&\tau_2^*,&&\text{if $k_2\in(k_2^*,\infty)$ and $k_1\in(\mathcal{K}_1^{3,a}(k_2),2\mathcal{K}_1^{3,a}(k_2))\cap(0,\mathcal{K}_1^{3,s})$}, 
					\end{split} 
				\end{cases} 
			\end{equation} 
			Under one of the conditions in \eqref{pri3}, system \eqref{eq-d3} undergoes a primary Hopf bifurcation at $\tau_H(\ep)$. Namely, the synchronized layered steady state $(U(x,\ep),U(x,\ep))$ is stable for $\tau\in(0,\tau_H(\ep))$ and unstable when $\tau\in(\tau_H(\ep),\tau_H(\ep)+\kappa)$ for some $\kappa>0$. Moreover, if $\tau_H(0)=\tau_3^*$ $($resp., $\tau_2^*$$)$, then system \eqref{eq-d3} admits a family of in-phase $($resp., anti-phase$)$ bifurcating breathers near $\tau=\tau_H(\ep)$. 
		\end{enumerate} 
	\end{theorem} 
	\begin{proof}
		For Case 3, the asymmetric SLEP sub-problem is identical to that in Case 2, whereas the symmetric sub-problem now corresponds to $(p,q,K)=(1,-1,0)$. Similar to Case 2, one can obtain the sufficient conditions \eqref{H3}, such that system \eqref{eq-d3} undergoes delay-induced instability through at least one of the symmetric and asymmetric sub-problems at some $\tau_H(\ep)>0$. Further, the first critical delay and corresponding modes can be determined under \eqref{pri3}, which is derived by combining Lemma \ref{le.trans1} and Corollary \ref{cor.K3}.
	\end{proof}  
	
	\begin{remark}\label{re.k2k1}
		Whenever the interval $(\mathcal{K}_2^{j,s},\mathcal{K}_1^{j,a}(k_2))$ ($j=2,3$) is nonempty for some $k_2\in(0,k_2^*]$, conditions \eqref{pri2} and \eqref{pri3} guarantee that the first loss of stability occurs through the symmetric mode in the corresponding cases. Numerical examples are presented in Figures \ref{Fig.dia_2} and \ref{Fig.dia_3}.
	\end{remark}
	
	\subsection{Analysis of the SLEP problem with distributed delays (Cases 4--6)}    
	In this subsection, we shall establish the general framework to analyze the spectral distribution of the SLEP system \eqref{SLEP-g} for Cases 4--6.
	
	For $p,q\in\R$, $k_1>0,\tau>0, \omega>0$, we define
	\begin{equation*}
		P(\omega^2):= \rho_0^* + k_1 p + \dfrac{k_1 q}{1+\tau^2\omega^2},\;Q(\omega^2):= \alpha +\dfrac{k_1 q\tau}{1+\tau^2\omega^2}.
	\end{equation*}
	
	\begin{proposition}Suppose that $\la_R^*=0$ and $\la_I^*=\omega>0$. Then, the real SLEP problem \eqref{SLEP-g} can be reduced to the following two algebraic equations in $\omega^2$:
		\begin{subequations}\label{eq-g2}
			\begin{numcases}{}
				X(0,\omega^2,K) = P(\omega^2),\label{u2: im-a}\\
				Y(0,\omega^2,K) =
				Q(\omega^2),\label{u2: im-b}	
			\end{numcases}
		\end{subequations}
		where  $X(0,\omega^2,K)$ and $Y(0,\omega^2,K)$ are defined in terms of \eqref{AB}, $(p,q,K)$ is given by \eqref{eq:pqk-sym} and \eqref{eq:pqk-asym}. Moreover, $P(\omega^2)$ and $Q(\omega^2)$ are strictly increasing $($resp., decreasing$)$ with respect to $\omega>0$ if $q<0$ $($resp., $q>0$$)$.
	\end{proposition}
	\begin{proof}
		It can be easily obtained from Proposition \ref{Pthetasa}.
	\end{proof}
	
	\begin{lemma}\label{le.no1}
		Suppose that $\alpha\in(\alpha_0,\infty)$. Then, \eqref{eq-g2} admits no positive root $\omega^2$ if one of the following conditions holds:
		\begin{enumerate}
			\item $q\ge 0$;
			\item $q<0$ and $\tau\in(0,\tau_0)$, where $\tau_0:=\frac{\alpha-Y(0,0,K)}{-qk_1}>0$;
			\item $q<0$, $\rho_0^*+k_1(p+q) > X(0,0,K)$;
			\item $q<0$, $\rho_0^*+k_1 p < 0$.
		\end{enumerate}		 		
	\end{lemma}
	\begin{proof}
		1. If $q\ge 0$, then the right-hand side of \eqref{u2: im-b} satisfies $Q(\omega^2) \ge \alpha > \alpha_0 > Y(0,\omega^2,K)$ for any $\omega\neq0$ and $K=0$ or $k_2$. Then, in this case,  \eqref{u2: im-b} has no positive root $\omega^2$.
		
		2. If $q<0$, then $Q(\omega^2)$ is increasing in $\omega\in(0,\infty)$. Since $Y(0,\omega^2,K)$ is decreasing in $\omega\in(0,\infty)$, thus, \eqref{u2: im-b} has no positive root $\omega^2$ when $Y(0,0,K) < Q(0)=\alpha + k_1 q \tau$, or equivalently, $\tau < \tau_0=\frac{\alpha-Y(0,0,K)}{-qk_1}$.
		
		3. If $q<0$, then $P(\omega^2)$ is increasing in $\omega$. Therefore, if $X(0,0,K) < P(0)=\rho_0^*+k_1(p+q)$, then \eqref{u2: im-a} has no positive root $\omega^2$.
		
		4. Similarly, if $q<0$, then \eqref{u2: im-a} has no positive root $\omega^2$ if $X(0,\infty,K) = 0 > P(\infty)=\rho_0^*+k_1 p$.
	\end{proof}
	
	\begin{lemma}\label{le.om1}
		Suppose that $\alpha\in(\alpha_0,\infty)$ and that $q<0$, $\rho_0^*+k_1 p>0$, $\rho_0^*+k_1(p+q) < X(0,0,K)$. Then \eqref{u2: im-a} has a unique solution $\omega=\omega_1(\tau)$ which is decreasing for $\tau\in(0,\infty)$. Moreover,
		\begin{enumerate}
			\item If $\rho_0^*+k_1 p - X(0,0,K) + \frac{qk_1}{2} \ge 0$, then $\tau\omega_1(\tau)<1$ for all $\tau\in(0,\infty)$.
			\item If $\rho_0^*+k_1 p + \frac{qk_1}{2} < 0$, then $\tau\omega_1(\tau)>1$ for all $\tau\in(0,\infty)$.
			\item If $\rho_0^*+k_1 p - X(0,0,K) + \frac{qk_1}{2} < 0$, and $\rho_0^*+k_1 p + \frac{qk_1}{2} > 0$, then there exists a unique $\tau_1=\tau_1(k_1,K)$, such that 
			\begin{equation}\label{om1}
				\tau\omega_1(\tau)
				\begin{cases}
					<1, & \tau\in(0,\tau_1),\\
					=1, & \tau=\tau_1,\\
					>1, & \tau\in(\tau_1,\infty),
				\end{cases}
			\end{equation}
			where $\tau_1$ is decreasing $($resp., increasing$)$ in $k_1$ if $p+ \frac{q}{2}<0$ $($resp., $p+ \frac{q}{2}>0$$)$. In particular, when $K=k_2$ and $0<\rho_0^*+k_1 p + \frac{qk_1}{2}< X(0,0,0)$, there is a unique solution $k_2=\wh\zeta(k_1)>0$ of $\rho_0^*+k_1 p - X(0,0,k_2) + \frac{qk_1}{2}=0$, such that $\tau_1(k_1,k_2)\to\infty$ as $k_2\to\wh\zeta(k_1)^-$.
		\end{enumerate}
	\end{lemma}
	\begin{proof}
		By the assumptions, we have $P(0)<X(0,0,K)$, $\lim\limits_{\omega\to\infty}P(\omega^2)=\rho_0^*+k_1p>0=\lim\limits_{\omega\to\infty}X(0,\omega^2,K)$. Since $P(\omega^2)$ is strictly increasing while $X(0,\omega^2,K)$ is strictly decreasing with respect to $\omega>0$, equation \eqref{u2: im-a} has a unique positive solution $\omega=\omega_1(\tau)$.
		
		That is,
		\begin{equation*}
			\rho_0^*+k_1 p + k_1 q \frac{1}{1+\tau^2\omega_1^2}-X(0,\omega_1^2,K)=0,
		\end{equation*}
		where $\omega_1=\omega_1(\tau)$. Differentiating it with respect to $\tau$, we have 
		\begin{equation*}
			\frac{\partial \omega_1(\tau)}{\partial \tau}=
			\frac{ (-q) k_1 \tau \omega_1}{(\tau^2\omega_1^2+1)^2 \frac{\partial X(0,\omega_1^2, K)}{\partial (\lambda_I^*)^2} +qk_1 \tau^2}<0,
		\end{equation*} 
		where $\frac{\partial X(0,(\lambda_I^*)^2, K)}{\partial (\lambda_I^*)^2}<0$ and $q<0$ are used. Thus, $\omega_1(\tau)$ is decreasing with respect to $\tau\in(0,+\infty)$.
		
		For $\tau\in(0,+\infty)$, we define
		\begin{equation*}
			F_1(\eta,\tau,k_1,K):=\rho_0^*+k_1 p + \frac{q k_1}{1+\eta} - X(0,\eta/\tau^2,K),\quad
			\eta_1(\tau):=\big(\tau\omega_1(\tau)\big)^2,
		\end{equation*}
		where $\eta>0$. Then, the relation $F_1(\eta_1(\tau),\tau,k_1,K)=0$ holds.
		
		By Proposition \ref{prop.AB}, we have
		\begin{equation*}
			\frac{\partial F_1}{\partial \eta}= -\frac{qk_1}{(1+\eta)^2} - \frac{1}{\tau^2}\frac{\partial X}{\partial(\lambda_I^*)^2}(0,\eta/\tau^2,K)>0,
		\end{equation*}
		which implies that
		\begin{equation}\label{eta1}
			F_1(\eta,\tau,k_1,K)
			\begin{cases}
				<0, &\text{if } \eta<\eta_1(\tau),\\
				=0, &\text{if } \eta=\eta_1(\tau),\\
				>0, &\text{if } \eta>\eta_1(\tau).
			\end{cases}
		\end{equation}
		
		Setting $\eta=1$, we have
		\begin{equation*}
			F_1(1,\tau,k_1,K)=\rho_0^*+k_1 p + \frac{qk_1}{2} - X(0,1/\tau^2,K),
		\end{equation*}
		which is decreasing with respect to $\tau\in(0,+\infty)$ due to Proposition \ref{prop.AB}. Moreover,
		\begin{equation*}
			F_1(1,0^+,k_1,K)=\rho_0^*+k_1 p + \frac{qk_1}{2},\;
			F_1(1,\infty,k_1,K)=\rho_0^*+k_1 p + \frac{qk_1}{2} - X(0,0,K).
		\end{equation*}
		
		1. If $F_1(1,\infty,k_1,K) \ge 0$, i.e., $\rho_0^*+k_1 p + \frac{qk_1}{2} - X(0,0,K)\ge 0$, then for any $\tau\in(0,\infty)$, we have $F_1(1,\tau,k_1,K) > 0$, implying that $1 > \eta_1(\tau)$, i.e., $\tau\omega_1(\tau)<1$. 
		
		2. If $F_1(1,0^+,k_1,K) < 0$, i.e., $\rho_0^*+k_1 p + \frac{qk_1}{2} < 0$, then for any $\tau\in(0,\infty)$, we have $F_1(1,\tau,k_1,K) < 0$, suggesting that $1 < \eta_1(\tau)$, i.e., $\tau\omega_1(\tau)>1$.
		
		3. If $F_1(1,0^+,k_1,K) > 0$ and $F_1(1,\infty,k_1,K) < 0$, i.e.,
		\begin{equation*}
			\rho_0^*+k_1 p + \frac{qk_1}{2} > 0 > \rho_0^*+k_1 p + \frac{qk_1}{2} - X(0,0,K),
		\end{equation*}
		then there exists a unique $\tau_1=\tau_1(k_1,K)$ such that
		\begin{equation*}
			F_1(1,\tau_1,k_1,K)=\rho_0^*+k_1 p + \frac{qk_1}{2} - X(0,1/\tau_1^2,K)=0.
		\end{equation*}
		Thus, by \eqref{eta1}, we can derive \eqref{om1}.
		
		Clearly, by implicit differentiation of $F_1 (1,\tau_1,k_1,K)=0$, we can obtain the monotonicity of $\tau_1$ with respect to $k_1$. Moreover, when $K=k_2$, if $F_1(1,\infty,k_1,0)=\rho_0^*+k_1 p + \frac{qk_1}{2} - X(0,0,0)<0$ and $F_1(1,\infty,k_1,\infty)=\rho_0^*+k_1 p + \frac{qk_1}{2}>0$, then by the fact that $F_1(1,\infty,k_1,k_2)$ is increasing in $k_2$, there exists a unique solution $k_2=\widehat{\zeta}(k_1)$ of $\rho_0^*+k_1 p - X(0,0,k_2) + \frac{qk_1}{2}=0$, such that $\tau_1(k_1,k_2)\to\infty$ as $k_2\to(\widehat\zeta(k_1))^-$.
	\end{proof}
	
	\begin{lemma}\label{le.om2}
		Suppose that $\alpha\in(\alpha_0,\infty)$ and that $q<0$, $\tau\in(\tau_0,\infty)$, where $\tau_0:=\frac{\alpha-Y(0,0,K)}{-qk_1}$ is defined in Lemma \ref{le.no1}. Then \eqref{u2: im-b} admits a unique solution $\omega=\omega_2(\tau)$ for $\tau\in(\tau_0,\infty)$. Moreover,
		\begin{enumerate}
			\item $\lim\limits_{\tau\to\tau_0^+}\omega_2(\tau) = \lim\limits_{\tau\to\infty}\omega_2(\tau) = 0$.
			\item There exists a unique $\tau_2=\tau_2(k_1,K)>\tau_0$, such that $\omega_2(\tau)$ is increasing in $\tau\in(\tau_0,\tau_2)$ and decreasing in $\tau\in(\tau_2,\infty)$. In particular,
			\begin{equation}\label{om2}
				\tau\omega_2(\tau)
				\begin{cases}
					<1, & \tau\in(\tau_0,\tau_2),\\
					=1, & \tau=\tau_2,\\
					>1, & \tau\in(\tau_2,\infty).
				\end{cases}
			\end{equation}
			\item $\tau_2$ is decreasing in $k_1$. Moreover, if $K=k_2$, then $\tau_2$ is increasing in $k_2$ and satisfies $\tau_2\to \frac{2\alpha}{-qk_1}>0$ as $k_2\to+\infty$.
		\end{enumerate}
	\end{lemma}
	\begin{proof}
		By the assumptions, we have $Q(0)<Y(0,0,K)$, $\lim\limits_{\omega\to\infty}Q(\omega^2)=\alpha>0=\lim\limits_{\omega\to\infty}Y(0,\omega^2,K)$. Since $Q(\omega^2)$ is strictly increasing while $Y(0,\omega^2,K)$ is strictly decreasing with respect to $\omega>0$, equation \eqref{u2: im-b} has a unique positive solution $\omega_2=\omega_2(\tau)$.
		
		That is,
		\begin{equation}\label{eq:ome2}
			Y(0, \omega_2^2, K) = \alpha + \frac{q k_1 \tau}{1 + \tau^2 \omega_2^2},
		\end{equation}
		where $\omega_2=\omega_2(\tau)$.
		
		1. Substituting $\tau_0$ into \eqref{eq:ome2} gives $\lim\limits_{\tau\to\tau_0^+}\omega_2(\tau)=0$, where the monotonicity properties of $Q(\omega^2)$ and $Y(0,\omega^2,K)$ with respect to $\omega>0$ are used. 
		
		We next show that $\omega_2(\tau)\to 0$ as $\tau\to\infty$. From \eqref{eq:ome2}, we have
		\begin{equation*}
			\frac{(-q) k_1 \tau}{1 + \tau^2 \omega_2^2} = \alpha - Y(0, \omega_2^2, K)\ge \alpha - Y(0, 0, K)>0.
		\end{equation*}
		
		It leads to
		\begin{equation*}
			1 + \tau^2 \omega_2^2 \le \frac{(-q) k_1 \tau}{\alpha - Y(0, 0, K)},
		\end{equation*}
		such that
		\begin{equation*}
			0< \omega_2^2 \le \frac{(-q) k_1 }{\alpha - Y(0, 0, K)}\frac{1}{\tau}-\frac{1}{\tau^2}\to0\;\;\text{(as $\tau\to\infty$)},
		\end{equation*}
		implying that $\omega_2(\tau)\to 0$ as $\tau\to\infty$.
		
		2. For $\tau\in(\tau_0,\infty)$, define
		\begin{equation*}
			F_2(\eta,\tau,k_1,K):=\alpha + \frac{qk_1\tau}{1+\eta} - Y\left( 0,\frac{\eta}{\tau^2},K\right) ,\quad
			\eta_2(\tau):=\big(\tau\omega_2(\tau)\big)^2,
		\end{equation*}
		where $\eta>0$. Then by \eqref{eq:ome2}, $F_2(\eta_2(\tau),\tau,k_1,K)=0$. Moreover,
		\begin{equation*}
			\frac{\partial F_2}{\partial \eta}= -\frac{qk_1\tau}{(1+\eta)^2} - \frac{1}{\tau^2}\frac{\partial Y}{\partial(\lambda_I^*)^2}\left( 0,\frac{\eta}{\tau^2},K\right)>0,
		\end{equation*}
		so that
		\begin{equation}\label{the02}
			F_2(\eta,\tau,k_1,K)
			\begin{cases}
				<0, &\eta<\eta_2(\tau),\\
				=0, &\eta=\eta_2(\tau),\\
				>0, &\eta>\eta_2(\tau).
			\end{cases}
		\end{equation}
		
		Setting $\eta=1$ gives
		\begin{equation}\label{F_22}
			F_2(1,\tau,k_1,K)=\alpha + \frac{qk_1\tau}{2} - Y\left( 0,\frac{1}{\tau^2},K\right).
		\end{equation}
		
		By Proposition \ref{prop.AB}, $F_2(1,\tau,k_1,K)$ is decreasing in $\tau\in(\tau_0,\infty)$. Since
		\begin{equation*}
			\lim_{\tau\to\tau_0^+}F_2(1,\tau,k_1,K)=\frac{\alpha+Y(0,0,K)-2Y(0,1/\tau_0^2,K)}{2}>0,\;
			\lim_{\tau\to\infty}F_2(1,\tau,k_1,K)=-\infty,
		\end{equation*}
		there exists a unique $\tau_2=\tau_2(k_1,K)\in(\tau_0,\infty)$ such that
		\begin{equation}\label{sig_22}
			F_2(1,\tau,k_1,K)
			\begin{cases}
				>0, &\tau<\tau_2,\\
				=0, &\tau=\tau_2,\\
				<0, &\tau>\tau_2.
			\end{cases}
		\end{equation}
		Combining \eqref{sig_22} and \eqref{the02}, we obtain \eqref{om2}.
		
		Now we show the monotonicity of $\omega_2(\tau)$. Differentiating \eqref{eq:ome2} with respect to $\tau$ gives
		\begin{equation*}
			\frac{d\omega_2(\tau)}{d\tau}=
			\frac{(-q)k_1}{2\omega_2(\tau)}\,
			\frac{(1+\tau\omega_2(\tau))(1-\tau\omega_2(\tau))}
			{(-q)k_1\tau^3 - (1+\tau^2\omega_2^2)^2\frac{\partial Y}{\partial(\lambda_I^*)^2}}
			\begin{cases}
				>0, &\tau\in(\tau_0,\tau_2),\\
				=0, &\tau=\tau_2,\\
				<0, &\tau\in(\tau_2,\infty),
			\end{cases}
		\end{equation*}
		where $\frac{\partial Y}{\partial(\lambda_I^*)^2}<0$ is used.
		
		3. When $K=k_2$, differentiating $F_2(1,\tau_2,k_1,k_2)=0$ with respect to $k_2$ (resp. $k_1$) shows that $\tau_2$ is increasing in $k_2$ (resp. decreasing in $k_1$). In particular, from \eqref{F_22} and \eqref{sig_22}, as $k_2\to\infty$,  $Y(0,1/\tau_2^2,k_2)\to0$ and $\alpha + \frac{qk_1\tau_2}{2}\to0$ gives $\tau_2\to \frac{2\alpha}{-qk_1}$.
	\end{proof}
	
	\begin{figure}[htbp]
		\centering
		\subfloat[]{
			\begin{tikzpicture}[font=\footnotesize]
				\draw[->] (0,0) --(3,0) node[right] {$\tau$};
				\draw[->] (0,0) --(0,3);
				\draw (0,0) node[left] {$O$};
				\draw[thick] (0.2, 3) .. controls (0.7, 0.6) and (1.5, 0.4) .. (3, 0.2);
				\node at (0.9,2) {$\ds{\frac{1}{\tau}}$};	
				\draw[blue, thick] (0.3,0) .. controls (0.7, 0.7) and (0.9, 0.9) .. (1.1, 0.9) .. controls (1.2, 0.9) and (2, 0.6)  .. (3, 0.4);
				\draw[red,thick] (0.2, 2) .. controls (0.6, 0.3) and (0.9, 0.2) .. (3, 0.1);
				\draw (0.3,0) node[below]{$\tau_0$};
				\draw[densely dashed] (1.08,0.9)--(1.08,0) node[below]{$\tau_2$};
				\draw[densely dashed] (0.73,0.68)--(0.73,0.1) node[below]{$\tau_H^*$};
				\draw[densely dashed] (0.73,0.66)--(0,0.66) node[left]{$\omega_H^*$};
		\end{tikzpicture}}
		\hspace{10pt}
		\subfloat[]{
			\begin{tikzpicture}[font=\footnotesize]
				\draw[->] (0,0) --(3,0) node[right] {$\tau$};
				\draw[->] (0,0) --(0,3);
				\draw (0,0) node[left] {$O$};
				\draw[thick] (0.2, 3) .. controls (0.7, 0.6) and (1.5, 0.4) .. (3, 0.2);
				\node at (0.9,2) {$\ds{\frac{1}{\tau}}$};	
				\draw[blue, thick] (0.3,0) .. controls (0.7, 0.7) and (0.9, 0.9) .. (1.1, 0.9) .. controls (1.2, 0.9) and (2, 0.6)  .. (3, 0.4);
				\draw[red,thick] (0.1, 2) .. controls (0.36, 0.45) and (0.9, 0.4) .. (3, 0.3);
				\draw (0.3,0) node[below]{$\tau_0$};
				\draw[densely dashed] (1.08,0.9)--(1.08,0) node[below]{$\tau_2$};
				\draw[densely dashed] (0.75,0.68)--(0.75,0.1) node[below]{$\tau_H^*$};
				\draw[densely dashed] (0.75,0.66)--(0,0.66) node[left]{$\omega_H^*$};
				\draw[densely dashed] (2.15,0.3)--(2.15,0) node[below]{$\tau_1$};
		\end{tikzpicture}}
		\hspace{10pt}
		\subfloat[]{
			\begin{tikzpicture}[font=\footnotesize]
				\draw[->] (0,0) --(3,0) node[right] {$\tau$};
				\draw[->] (0,0) --(0,3);
				\draw (0,0) node[left] {$O$};
				\draw[thick] (0.2, 3) .. controls (0.7, 0.6) and (1.5, 0.4) .. (3, 0.2);
				\node at (0.3,1.5) {$\ds{\frac{1}{\tau}}$};	
				\draw[blue, thick] (0.3,0) .. controls (0.7, 0.7) and (0.9, 0.9) .. (1.1, 0.9) .. controls (1.2, 0.9) and (2, 0.6)  .. (3, 0.4);
				\draw[red,thick] (0.5, 3) .. controls (1.36, 0.45) and (1.9, 0.4) .. (3, 0.3);
				\draw (0.3,0) node[below]{$\tau_0$};
				\draw[densely dashed] (1.08,0.9)--(1.08,0) node[below]{$\tau_2$};
				\draw[densely dashed] (1.65,0.7)--(1.65,0.1) node[below]{$\tau_H^*$};
				\draw[densely dashed] (1.65,0.7)--(0,0.7) node[left]{$\omega_H^*$};
		\end{tikzpicture}}
		\caption{Graphs of $\omega_1(\tau)$ (red) and $\omega_2(\tau)$ (blue) under $q<0$, $\rho_0^*+k_1 p>0$, and $\rho_0^*+k_1(p+q) < X(0,0,K)$, where $\omega_1(\tau)$ is decreasing in $\tau\in(0,\infty)$; $\omega_2(\tau)$ is increasing in $\tau\in(\tau_0,\tau_2)$ and decreasing in $\tau\in(\tau_2,\infty)$ with $\lim_{\tau\to\tau_0^+}\omega_2(\tau)=\lim_{\tau\to\infty}\omega_2(\tau)=0$. (a) If $\rho_0^*+k_1 p - X(0,0,K) + \frac{qk_1}{2} \ge 0$, then $\omega_1(\tau)<1/\tau$ for $\tau\in(0,\infty)$, and there exists a unique positive root $\tau_H^*\in(\tau_0,\tau_2)$ such that $\omega_1(\tau_H^*)=\omega_2(\tau_H^*)<1/\tau_H^*$. (b) If $\rho_0^*+k_1 p - X(0,0,K) + \frac{qk_1}{2} < 0$, and $\rho_0^*+k_1 p + \frac{qk_1}{2} > 0$, then there exists a unique $\tau_1=\tau_1(k_1,K)\in(0,\infty)$ such that $\omega_1(\tau_1)=1/\tau_1$. Further, if $\tau_1>\tau_2$ (the sufficient conditions are given by the second and third parts of Lemma \ref{le.e2}), then there is a root $\tau_H^*$ of $\omega_1(\tau)=\omega_2(\tau)$ which is unique in $(\tau_0,\tau_2)$ and satisfies $\omega_1(\tau_H^*)=\omega_2(\tau_H^*)<1/\tau_H^*$. (c) When $\rho_0^*+k_1 p + \frac{qk_1}{2} < 0$, then $\omega_1(\tau)>1/\tau$ for $\tau\in(0,\infty)$. Further, if $\rho_0^*+k_1 p>0$, $\rho_0^*+k_1p+k_1q<0$, $K=k_2$, and $k_2>M(k_1)$ for some sufficiently large $M(k_1)$, then $\omega_1(\tau)=\omega_2(\tau)$ at some $\tau_H^*$ with asymptotic behavior \eqref{omsi}.}
		\label{Fig. tau_2}
	\end{figure}
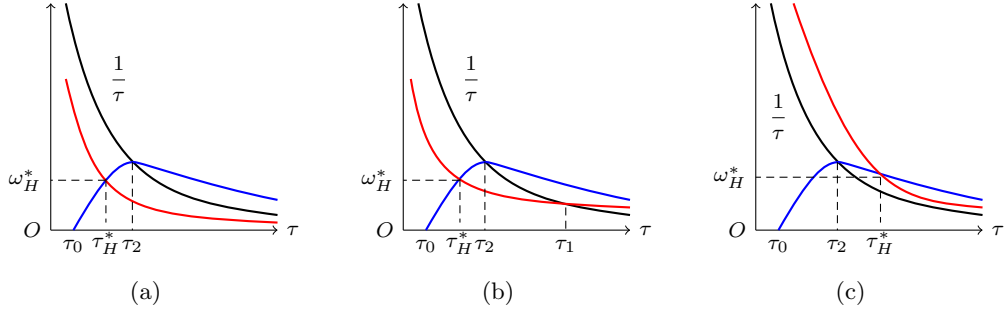
	
	\begin{lemma}\label{le.e2}
		Suppose that $\alpha\in(\alpha_0,\infty)$ and that $q<0$, $\rho_0^*+k_1 p>0$, $\rho_0^*+k_1(p+q) < X(0,0,K)$, $\tau\in(\tau_0,\infty)$. Let $\omega_1=\omega_1(\tau)$ and $\omega_2=\omega_2(\tau)$ be defined by Lemmas \ref{le.om1} and \ref{le.om2}, respectively. 
		\begin{enumerate}
			\item If $\rho_0^*+k_1 p - X(0,0,K) + \frac{q k_1}{2} \ge 0$, then system \eqref{eq-g2} has a critical value $\tau_H^*$ which is unique in $(\tau_0,\tau_2)$, such that $\omega_1(\tau_H^*)=\omega_2(\tau_H^*)=:\omega_H^*$ and $\tau_H^*\omega_H^*<1$.
			\item If $K=k_2$, $p+\frac{q}{2}<0$, $\rho_0^*+k_1 p + \frac{q k_1}{2} > 0$ and $\rho_0^*+k_1 p - X(0,0,k_2) + \frac{q k_1}{2} < 0$, then there exists a $\wt\zeta(k_1)\in[0,\wh\zeta(k_1))$, where $\wh\zeta(k_1)$ is given by Lemma \ref{le.om1}, such that $\tau_2<\tau_1$ for $k_2\in(\wt\zeta(k_1),\wh\zeta(k_1))$. Then for $k_2\in(\wt\zeta(k_1),\wh\zeta(k_1))$, system \eqref{eq-g2} has a critical value $\tau_H^*$ which is unique in $(\tau_0,\tau_2)$ and satisfies $\omega_H^*:=\omega_1(\tau_H^*)=\omega_2(\tau_H^*)$, $\tau_H^*\omega_H^*<1$. In particular, $\wt\zeta(k_1)\to k_2^*$ as $k_1\to 0^+$, $\wt\zeta(k_1)\to+\infty$ as $k_1\to \left( \frac{-2\rho_0^*}{2p+q} \right)^-$.
			\item If $K=0$ and $p+\frac{q}{2}>0$, then there exist unique $k_1^*:=2(X(0,0,0)-\rho_0^*)/(2p+q)>0$ and $\wh k_1^*\in(0,k_1^*)$, which satisfy $\tau_1(\wh k_1^*,0)-\tau_2(\wh k_1^*,0)=0$ and $\tau_1(k_1,0)>\tau_2(k_1,0)$ in $k_1\in(\wh k_1^*,k_1^*)$. Consequently, for $k_1\in(\wh k_1^*,k_1^*)$, system \eqref{eq-g2} has a critical value $\tau_H^*$ which is unique in $(\tau_0,\tau_2)$ and satisfies $\omega_H^*:=\omega_1(\tau_H^*)=\omega_2(\tau_H^*)$ and $\tau_H^*\omega_H^*<1$.
			\item If $K=k_2$, $\rho_0^*+k_1 p>0$, and $\rho_0^*+k_1p+k_1q<0$, then there exists a $M(k_1)>0$, such that for $k_2>M(k_1)$, system \eqref{eq-g2} has no solution in $\tau\in(0,\tau_H^*)$ but admits a unique solution $(\omega_H^*)^2$ at $\tau_H^*$, which satisfies
			\begin{equation}\label{omsi}
				\omega_H^*\to \frac{\sqrt{-(\rho_0^*+k_1 p)(\rho_0^*+k_1p+k_1q)}}{\alpha},\; \tau_H^*\to \frac{\alpha}{\rho_0^*+k_1 p},\;\text{as $k_2\to\infty$}.
			\end{equation}
		\end{enumerate}
	\end{lemma}
	\begin{proof}
		1. Suppose that $\rho_0^*+k_1 p - X(0,0,K) + \frac{q k_1}{2} \ge 0$. Then by Lemma \ref{le.om1}, we conclude that $\omega_1(\tau)$ is decreasing for all $\tau\in(0,\infty)$ with $\omega_1(\tau)<1/\tau$. By Lemma \ref{le.om2}, $\omega_2(\tau)$ is increasing in $(\tau_0,\tau_2)$ and decreasing in $(\tau_2,\infty)$ with $\omega_2(\tau_2)=1/\tau_2$, and $\omega_2(\tau)\to0$ as $\tau\to\tau_0^+$ and $\tau\to\infty$. Thus, in $(\tau_0,\tau_2)$, there is a unique root $\tau_H^*$ satisfying $\omega_H^*:=\omega_1(\tau_H^*)=\omega_2(\tau_H^*)$ and $\tau_H^*\omega_H^*<1$.
		
		2. Suppose that $K=k_2$, $p+\frac{q}{2}<0$, $\rho_0^*+k_1 p + \frac{q k_1}{2} > 0$ and $\rho_0^*+k_1 p - X(0,0,k_2) + \frac{q k_1}{2} < 0$. Then by Lemmas \ref{le.om1} and \ref{le.om2}, we have 
		\begin{equation*}
			\tau_2(k_1,\wh\zeta(k_1)) < \lim\limits_{k_2\to\infty}\tau_2(k_1,k_2)=\frac{2\alpha}{-qk_1}<\infty=\lim\limits_{k_2\to(\wh\zeta(k_1))^-}\tau_1(k_1,k_2).
		\end{equation*}	
		
		Define $\wt{\zeta}(k_1) := \inf \{ \xi \in [0, \wh{\zeta}(k_1)) : \tau_2(k_1, k_2) < \tau_1(k_1, k_2) \text{ for all } k_2 \in (\xi, \wh{\zeta}(k_1)) \}$. Then, we have $\tau_2(k_1,k_2)<\tau_1(k_1,k_2)$ for $k_2\in(\wt\zeta(k_1),\wh\zeta(k_1))$. Since $\omega_1(\tau_0)>0=\omega_2(\tau_0)$ and $\omega_1(\tau_2)<1/\tau_2=\omega_2(\tau_2)$, combining the monotonicity of $\omega_1(\tau)$ and $\omega_2(\tau)$ in $\tau\in(\tau_0,\tau_2)$, then there exists a critical value $\tau_H^*$ which is unique in $(\tau_0,\tau_2)$ and satisfies $\omega_H^*:=\omega_1(\tau_H^*)=\omega_2(\tau_H^*)$ and $\omega_H^*\tau_H^*<1$.
		
		In particular, we shall show that $\wt\zeta(k_1)\to k_2^*$ as $k_1\to 0^+$ and $\wt\zeta(k_1)\to+\infty$ as $k_1\to \left( \frac{-2\rho_0^*}{2p+q} \right)^-$. Recall that $\omega_1(\tau)$ and $\omega_2(\tau)$ are the solutions of \eqref{u2: im-a} and \eqref{u2: im-b}, and $\tau_1=\tau_1(k_1,k_2)$ and $\tau_2=\tau_2(k_1,k_2)$ are determined by $F_1(1,\tau_1,k_1,K)=0$ and $F_2(1,\tau_2,k_1,k_2)=0$, respectively. Moreover, for $i=1,2$, $\tau_i\omega_i(\tau_i)=1$ at $\tau=\tau_i$ (since $\eta=1$). Thus, we have
		\begin{subequations}
			\begin{numcases}{}
				X(0,1/\tau_1^2(k_1,\wt\zeta(k_1)),\wt\zeta(k_1))=\rho_0^*+k_1 p + \frac{qk_1}{2},\label{3.29a}\\
				Y(0,1/\tau_2^2(k_1,\wt\zeta(k_1)),\wt\zeta(k_1))=\alpha + \frac{qk_1}{2}\tau_2(k_1,\wt\zeta(k_1)),\label{3.29b}\\
				\tau_1(k_1,\wt\zeta(k_1))\ge\tau_2(k_1,\wt\zeta(k_1)).\label{3.29c}
			\end{numcases}
		\end{subequations}
		
		Firstly, we show that $\lim\limits_{k_1\to0^+}\wt\zeta(k_1)=k_2^*$. By \eqref{3.29b}, we have
		\begin{equation}\label{tau2>}
			\tau_2(k_1,\wt\zeta(k_1))=\dfrac{2(\alpha-Y(0,1/\tau_2^2(k_1,\wt\zeta(k_1)),\wt\zeta(k_1)))}{(-q)k_1}>\dfrac{2(\alpha-\alpha_0)}{(-q)k_1}\to\infty\;\;\text{(as $k_1\to0^+$)},
		\end{equation}
		such that $\lim\limits_{k_1\to0^+}\tau_1(k_1,\wt\zeta(k_1))=\infty$ due to \eqref{3.29c}. Then from \eqref{3.29a}, combining $X(0,0,k_2^*)=\rho_0^*$, one can obtain that $\lim\limits_{k_1\to0^+}\wt\zeta(k_1)=k_2^*$.
		
		Secondly, as $k_1\to\left( \frac{-2\rho_0^*}{2p+q} \right)^-$, it follows from \eqref{3.29a} that $X(0,1/\tau_1^2(k_1,\wt\zeta(k_1)),\wt\zeta(k_1))\to0$. By \eqref{3.29c} and \eqref{tau2>}, $\tau_1(k_1,\wt\zeta(k_1))$ is bounded away from zero. Consequently, $\wt\zeta(k_1)\to+\infty$ as $k_1\to\left( \frac{-2\rho_0^*}{2p+q} \right)^-$.
		
		3. Suppose that $K=0$ and $p+\frac{q}{2}>0$. By Lemma \ref{le.om1}, $\tau_1$ is increasing in $k_1$. Denote the unique solution of $\rho_0^*+k_1p-X(0,0,0)+\frac{qk_1}{2}=0$ by $k_1^*:=2(X(0,0,0)-\rho_0^*)/(2p+q)>0$. Then as $k_1\to (k_1^*)^-$, we have $\tau_1\to \infty$. On the other hand, from Lemma \ref{le.om2}, $\tau_2$ is decreasing in $k_1$ with $\tau_2\to\infty$ as $k_1\to0^+$. Then $\wh\tau(k_1):=\tau_1(k_1,0)-\tau_2(k_1,0)$ is increasing in $k_1\in(0,k_1^*)$, which satisfies $\wh\tau(k_1)\to-\infty$ as $k_1\to0^+$ and $\wh\tau(k_1)\to+\infty$ as $k_1\to(k_1^*)^-$. Thus, there exists a unique $\wh k_1^*\in(0,k_1^*)$, such that $\tau_2=\tau_1$ at $k_1=\wh k_1^*$ and $\tau_2<\tau_1$ for $k_1\in(\wh k_1^*,k_1^*)$. It leads to the conclusions for $k_1\in(\wh k_1^*,k_1^*)$.
		
		4. Suppose that $\rho_0^*+k_1 p>0$, $\rho_0^*+k_1p+k_1q<0$, and $K=k_2$, then $\omega_1(\tau)$ and $\omega_2(\tau)$ exist.  As $k_2\to\infty$, system \eqref{eq-g2} reduces to
		\begin{equation*}
			0 = \rho_0^* + k_1 p + k_1 q\,\dfrac{1}{1+\tau^2\omega^2},\qquad
			0 = \alpha + k_1 q\,\dfrac{\tau}{1+\tau^2\omega^2},
		\end{equation*}
		
		Define $\omega_{1,\infty}(\tau):=\lim\limits_{k_2\to\infty}\omega_1(\tau)$, $\omega_{2,\infty}(\tau):=\lim\limits_{k_2\to\infty}\omega_2(\tau)$, and $\tau^*:=\dfrac{\alpha}{\rho_0^*+k_1 p}>0$. Then, 
		\begin{equation*}
			\omega_{1,\infty}(\tau)=\frac{1}{\tau}\sqrt{\frac{-qk_1}{\rho_0^*+k_1 p}-1},\;\;
			\omega_{2,\infty}(\tau)=\frac{1}{\tau}\sqrt{\frac{-qk_1\tau}{\alpha}-1},
		\end{equation*}
		which satisfies
		\begin{equation*}
			\omega_{1,\infty}(\tau)
			\begin{cases}
				\begin{split}
					&>\omega_{2,\infty}(\tau),&&\tau\in\left( \frac{\alpha}{-qk_1},\tau^*\right), \\
					&=\omega_{2,\infty}(\tau),&&\tau=\tau^*,\\
					&<\omega_{2,\infty}(\tau),&&\tau\in\left(\tau^*,\infty\right).
				\end{split}
			\end{cases}
		\end{equation*}
		
		Thus, for any fixed $\tau_-$ and $\tau_+$ satisfying $\frac{\alpha}{-qk_1}<\tau_-<\tau^*<\tau_+$, as long as $k_2>M(k_1)$ for some sufficiently large $M(k_1)$, we have $\omega_1(\tau_-)>\omega_2(\tau_-)$ and $\omega_1(\tau_+)<\omega_2(\tau_+)$, such that there exists at least one solution of $\omega_1(\tau)=\omega_2(\tau)$, where the smallest one is denoted by $\tau_H^*$. 
		
		On the other hand, in $\tau\in(0,\tau_0)$, there is no solution of \eqref{eq-g2} due to Lemma \ref{le.no1}. 
		
		Finally, in $\tau\in(\tau_0,\tau_-)$, from \eqref{eq-g2}, we have
		\begin{equation*}
			\rho_0^*+k_1 p+\frac{qk_1}{1+(\omega_1\tau)^2}=X(0,\omega_1^2,k_2)>0,\;\alpha+\frac{q k_1\tau}{1+(\omega_2\tau)^2}=Y(0,\omega_2^2,k_2)<Y(0,0,k_2),
		\end{equation*}
		such that
		\begin{equation*}
			(\omega_1\tau)^2>\frac{-qk_1}{\rho_0^*+k_1 p}-1,\;\;(\omega_2\tau)^2<\frac{-qk_1\tau}{\alpha-Y(0,0,k_2)}-1=\frac{\tau}{\tau_0}-1<\frac{\tau_-}{\tau_0}-1,
		\end{equation*}
		
		Note that $\frac{-q k_1}{\rho_0^*+k_1 p}\tau_0 \to \frac{\alpha}{\rho_0^*+k_1 p}=\tau^*>\tau_-$ as $k_2\to\infty$. Thus, for sufficiently large $k_2$, we have $\frac{-q k_1}{\rho_0^*+k_1 p}>\frac{\tau_-}{\tau_0}$, indicating that $\omega_1(\tau)>\omega_2(\tau)$ in $\tau\in(\tau_0,\tau_-)$. 
		
		Hence, $\tau_H^*\in(\tau_-,\tau_+)$ is the first critical value. Correspondingly, we obtain the asymptotic behavior $\tau_H^*\to\tau^*$ and $\omega_H^*\to\omega_{1,\infty}(\tau^*)$ as $k_2\to\infty$.
	\end{proof}
	
	\begin{lemma}\label{le.trans2}
		Let the critical point $(\tau_H^*,\omega_H^*)$ be stated by Lemma \ref{le.e2}. In particular, there exists a small $\kappa>0$, such that the SLEP system \eqref{F*} admits a unique pair of critical complex eigenvalues $\la_R^*(\tau)\pm i\la_I^*(\tau)$ in $\tau\in(\tau_H^*-\kappa,\tau_H^*+\kappa)$ which satisfies $\la_R^*(\tau_H^*)=0$, $\la_I^*(\tau_H^*)=\omega_H^*$, and $\frac{d}{d\tau}\la_R^*(\tau_H^*)>0$.
	\end{lemma}
	\begin{proof}
		The complex SLEP system \eqref{F*} is recast as
		\begin{equation*}
			\mathcal{F}^*(\la^*,\tau):=\rho_0^*-\alpha\la^*+k_1(p+q \theta_{\tau}^*)-\sum_{n=0}^{\infty}\frac{c_1^*c_2^*}{\gamma_n+\la^*+2K}\psi_n^2(x^*)=0,
		\end{equation*}
		where $\theta_{\tau}^*=\frac{1}{1+\lambda^*\tau}$.
		
		It suffices to prove that
		\begin{equation}\label{eq.tra2}
			\frac{d\la_R^*}{d\tau}(\tau_H^*)=-\frac{
				\operatorname{Re}\frac{ \pa\mathcal{F}^*}{\pa\la^*}(i\omega_H^*,\tau_H^*)
				\operatorname{Re}\frac{ \pa\mathcal{F}^*}{\pa\tau}(i\omega_H^*,\tau_H^*)+
				\operatorname{Im}\frac{ \pa\mathcal{F}^*}{\pa\la^*}(i\omega_H^*,\tau_H^*)
				\operatorname{Im}\frac{ \pa\mathcal{F}^*}{\pa\tau}(i\omega_H^*,\tau_H^*)}{\left(\operatorname{Re}\frac{ \pa\mathcal{F}^*}{\pa\la^*}(i\omega_H^*,\tau_H^*)\right)^2+\left(\operatorname{Im}\frac{ \pa\mathcal{F}^*}{\pa\la^*}(i\omega_H^*,\tau_H^*)\right)^2}>0,
		\end{equation}
		where 
		\begin{equation*}
			\begin{split}
				&\operatorname{Re}\frac{ \pa\mathcal{F}^*}{\pa\la^*}(i\omega_H^*,\tau_H^*)=-\alpha+\frac{k_1 q \tau_H^*[(\tau_H^*\omega_H^*)^2-1]}{[1+(\tau_H^*\omega_H^*)^2]^2}+\sum_{n=0}^{\infty} \frac{c_1^*c_2^*[(\gamma_n + 2K)^2-(\omega_H^*)^2]}{[(\gamma_n + 2K)^2 + (\omega_H^*)^2]^2}\,\psi_n^2(x^*),\\
				&\operatorname{Im}\frac{\pa\mathcal{F}^*}{\pa\la^*}(i\omega_H^*,\tau_H^*)=\frac{2k_1 q \tau_H^{*2} \omega_H^*}{[1+(\tau_H^*\omega_H^*)^2]^2}-\sum_{n=0}^{\infty} \frac{2c_1^*c_2^*(\gamma_n + 2K)\omega_H^*}{[(\gamma_n + 2K)^2 + (\omega_H^*)^2]^2}\,\psi_n^2(x^*),\\
				&\operatorname{Re}\frac{ \pa\mathcal{F}^*}{\pa\tau}(i\omega_H^*,\tau_H^*)=-\frac{2 k_1 q \tau_H^* \omega_H^{*2}}{[1+(\tau_H^*\omega_H^*)^2]^2},\;
				\operatorname{Im}\frac{ \pa\mathcal{F}^*}{\pa\tau}(i\omega_H^*,\tau_H^*)=-\frac{k_1 q \omega_H^*[1-(\tau_H^*\omega_H^*)^2]}{[1+(\tau_H^*\omega_H^*)^2]^2}.
			\end{split}
		\end{equation*}
		
		Moreover, we have $\operatorname{Im}\frac{\pa\mathcal{F}^*}{\pa\la^*}(i\omega_H^*,\tau_H^*)<0$ due to $q<0$.
		
		One can check that $\frac{d\la_R^*}{d\tau}(\tau_H^*)>0$ is equivalent to $I_3+I_4>0$, where 
		\begin{equation}\label{I34}
			\begin{split}
				&I_3=\left[ \alpha-\sum_{n=0}^{\infty} \frac{c_1^*c_2^*[(\gamma_n + 2K)^2-(\omega_H^*)^2]}{[(\gamma_n + 2K)^2 + (\omega_H^*)^2]^2}\,\psi_n^2(x^*)\right] \frac{2 k_1 (-q)\tau_H^*\omega_H^{*2}}{[1+(\tau_H^*\omega_H^*)^2]^2},\\
				&I_4=\left[ \sum_{n=0}^{\infty} \frac{2c_1^*c_2^*(\gamma_n + 2K)\omega_H^*}{[(\gamma_n + 2K)^2 + (\omega_H^*)^2]^2}\,\psi_n^2(x^*)\right] \frac{k_1 (-q)\omega_H^*[1-(\tau_H^*\omega_H^*)^2]}{[1+(\tau_H^*\omega_H^*)^2]^2}.
			\end{split}
		\end{equation}
		where $I_3>0$ can be obtained by \eqref{alp-}.

		On the other hand, under the first three conditions of Lemma \ref{le.e2}, we have $\omega_H^*\tau_H^*<1$, such that $I_4>0$. Under the last condition of Lemma \ref{le.e2}, as $k_2\to\infty$, we have $I_4\to 0$ ($I_3\to c$ for some $c>0$), such that $\frac{d\la_R^*}{d\tau}(\tau_H^*)>0$ holds for $k_2>M(k_1)$, where $M(k_1)$ can be a larger one if necessary. The proof is complete.
	\end{proof}
	
	We next characterize the stability regions for Cases 4–6. For convenience, we denote by $\mathcal{K}_i^{j,\sigma}$ the $i$-th threshold of Case $j$ in the $\sigma$ sub-problem, where $j\in\{4,5,6\}$, $\sigma\in\{s,a\}$ denotes the symmetric or asymmetric sub-problem, and $i=1,2$ when applicable. In particular, $(j,\sigma)$ corresponds to a specific $(p,q,K)$ (see \eqref{eq:pqk-sym} and \eqref{eq:pqk-asym}). 
	
	\begin{lemma}\label{le.e3} 
		Suppose that $\alpha\in(\alpha_0,\infty)$, $q<0$. For any $k_2\in(0,k_2^*]$, let $F(\omega^2)$ be given by \eqref{S2} and let $\omega_0=\omega_0(k_2)$ be the unique nonnegative solution of $X(0,\omega_0^2,k_2)=\rho_0^*$. Define $\Delta:=9(\rho_0^*)^2-8F(\omega_0^2)$, and 
		\begin{enumerate} 
			\item $\mathcal{K}_1^{4,a}(k_2):=\mathcal{K}_1^{2,a}(k_2)=\mathcal{K}_1^{3,a}(k_2)$. 
			\item $\mathcal{K}_1^{6,s}:=2\mathcal{K}_1^{3,s}$. 
			\item $\mathcal{K}_1^{j,a}(k_2):= 
			\begin{cases} 
				\frac{3\rho_0^*-\sqrt{\Delta}}{4}, & \Delta>0,\\[6pt] 
				+\infty, & \Delta\le 0, 
			\end{cases}$,\;
			$\mathcal{K}_2^{j,a}(k_2):= 
			\begin{cases} 
				\dfrac{3\rho_0^*+\sqrt{\Delta}}{4}, & \Delta>0,\\[6pt] 
				+\infty, & \Delta\le 0, 
			\end{cases}$ $(j=5,6)$. 
		\end{enumerate} 
		Then for the mode $\sigma$ of Case $j$ $($$(j,\sigma)=(4,a),(5,a),(6,s),(6,a)$$)$, if $k_1<\mathcal{K}_1^{j,\sigma}$, there are no purely imaginary eigenvalues for any $\tau\in(0,\infty)$. Moreover, for $(j,\sigma)=(5,a),(6,a)$, if $k_1>\mathcal{K}_2^{j,\sigma}$, the nonexistence of purely imaginary eigenvalues for any $\tau\in(0,\infty)$ also holds. In particular, as $k_2\to (k_2^*)^-$, we have $\mathcal{K}_1^{4,a}(k_2)\to \mathcal{K}_1^{4,a}(k_2^*)>0$, and $\mathcal{K}_1^{j,a}(k_2)\to 0,\;\mathcal{K}_2^{j,a}(k_2)\to \frac{3\rho_0^*}{2}\; (j=5,6)$.
	\end{lemma} 
	
	\begin{proof} 
		From \eqref{eq-g2}, we have 
		\begin{equation}\label{tau} 
			\tau = \frac{Y(0,\omega^2,K)-\alpha}{X(0,\omega^2,K)-\rho_0^*-k_1 p}>0. 
		\end{equation} 
		
		Substituting \eqref{tau} into \eqref{u2: im-a}, one can obtain 
		\begin{equation}\label{F-eq_tri} 
			F(\omega^2)=(q+2p)k_1(X(0,\omega^2,K)-\rho_0^*)-p(p+q)k_1^2, 
		\end{equation} 
		where $F(\omega^2)=(X(0,\omega^2,K)-\rho_0^*)^2+\omega^2 (Y(0,\omega^2,K)-\alpha)^2$ is given by \eqref{S2} and is analyzed in Lemma \ref{le.F}. 
		
		Thus, the SLEP system \eqref{eq-g2} is equivalent to \eqref{tau} and \eqref{F-eq_tri}.  
		
		1. For the asymmetric mode of Case 4 ($(p,q,K)=(0,-2,k_2)$), \eqref{F-eq_tri} reduces to 
		\begin{equation}\label{tM-4} 
			F(\omega^2)=2k_1(\rho_0^*-X(0,\omega^2,k_2)), 
		\end{equation} 
		which coincides with the asymmetric modes of Cases 2 and 3 (\eqref{S2} with $(p,q,K)=(-1,-1,k_2)$). Thus, we have $\mathcal{K}_1^{4,a}(k_2)
		=\mathcal{K}_1^{2,a}(k_2)
		=\mathcal{K}_1^{3,a}(k_2)$.  
		
		2. For the symmetric mode of Case 6 ($(p,q,K)=(1,-1,0)$), \eqref{F-eq_tri} reduces to 
		\begin{equation*} 
			F(\omega^2)=k_1(X(0,\omega^2,0)-\rho_0^*). 
		\end{equation*} 
		
		It is similar to the symmetric mode of Case 3 (self-discrete-delayed coupling), that is, $F(\omega^2)=2k_1(X(0,\omega^2,0)-\rho_0^*)$, which has the critical value $\mathcal{K}_1^{3,s}$, leading to the factor $2$ of $\mathcal{K}_1^{6,s}=2\mathcal{K}_1^{3,s}$. 
		
		3. For the asymmetric modes of Cases 5 and 6 ($(p,q,K)=(-1,-1,k_2)$), \eqref{F-eq_tri} reduces to 
		\begin{equation}\label{F=G_a2} 
			F(\omega^2)=3k_1(\rho_0^*-X(0,\omega^2,k_2))-2k_1^2. 
		\end{equation} 
		
		For $\omega\in(0,\omega_0)$, we have $\rho_0^*-X(0,\omega^2,k_2)<0$. Thus, $F(\omega^2)>0>3k_1(\rho_0^*-X(0,\omega^2,k_2))-2k_1^2$, such that \eqref{F=G_a2} has no solutions. For $\omega\in(\omega_0,\infty)$, $F(\omega^2)$ is increasing such that $F(\omega^2)>F(\omega_0^2)$. On the other hand, we have $3k_1(\rho_0^*-X(0,\omega^2,k_2))-2k_1^2<3k_1\rho_0^*-2k_1^2$. Thus, if $3k_1\rho_0^*-2k_1^2<F(\omega_0^2)$, or equivalently, $2k_1^2-3k_1\rho_0^*+F(\omega_0^2)>0$, then \eqref{F=G_a2} has no solution with respect to $\omega\in(\omega_0,\infty)$. As a quadratic polynomial in $k_1$, if $\Delta=9(\rho_0^*)^2-8F(\omega_0^2)>0$, one can immediately derive nonnegative
		\begin{equation*}
			\mathcal{K}_1^{j,a}
			=\frac{3\rho_0^*-\sqrt{\Delta}}{4},
			\;\;
			\mathcal{K}_2^{j,a}
			=\frac{3\rho_0^*+\sqrt{\Delta}}{4}
			\;\; (j=5,6),
		\end{equation*} 
		such that for $k_1<\mathcal{K}_1^{j,a}$ or $k_1>\mathcal{K}_2^{j,a}$, \eqref{F=G_a2} has no solution. While if $\Delta\le 0$, then \eqref{F=G_a2} has no solution for any $k_1>0$. Moreover, as $k_2\to(k_2^*)^-$, we have $\omega_0\to 0$, which leads to $F(\omega_0^2)\to 0$ and $\sqrt{\Delta}\to 3\rho_0^*>0$. Thus, the proof is complete. 
	\end{proof} 
	
	\begin{remark}\label{re.K234}
		\begin{enumerate} 
			\item For all Cases 1–6, sufficiently small $k_1$ and $k_2$ cannot activate the delay-induced instability mechanism, regardless of the type or placement of delay.
			\item From the proof, we can conclude that $\mathcal{K}_1^{4,a}(k_2)=\mathcal{K}_1^{2,a}(k_2)=\mathcal{K}_1^{3,a}(k_2)$ for all $k_2\in(0,\infty)$. Thus, if $k_1<\mathcal{K}_1^{4,a}$, then the asymmetric SLEP system of Case 4 has no purely imaginary eigenvalues for any $\tau\in(0,\infty)$. 
			\item For the asymmetric modes of Cases 5 and 6 ($(p,q,K)=(-1,-1,k_2)$), the preceding nonexistence conclusion for $k_1>\rho_0^*$ in Lemma \ref{le.no1} can also be verified directly by \eqref{tau}. Indeed, if $k_1>\rho_0^*$, from \eqref{tau}, we have 
			\begin{equation*} 
				\tau =\frac{Y(0,\omega^2,k_2)-\alpha}{X(0,\omega^2,k_2)+k_1-\rho_0^*}<0, 
			\end{equation*} 
			due to $\alpha\in(\alpha_0,\infty)$. Thus, there are no purely imaginary eigenvalues for any  $\tau\in(0,\infty)$. 
		\end{enumerate} 
	\end{remark}
	
	\begin{corollary}\label{cor.tM}
		For any $k_2\in(0,\infty)$, in Case 4, there exists a sufficiently large 
		$\mathcal{K}_2^{4,a}(k_2)\ge\mathcal{K}_1^{4,a}(k_2)$, such that if 
		$k_1\in(\mathcal{K}_2^{4,a}(k_2),\infty)$, then system \eqref{eq-g2} 
		has no solution in $\tau\in(0,\tau_H^*)$ and admits a unique solution 
		$(\omega_H^*)^2$ at $\tau_H^*$ satisfying
		\begin{equation*}
			\tau_H^*\to\frac{\alpha}{\rho_0^*},
			\quad
			\omega_H^*\to\infty
			\quad\text{as }k_1\to\infty.
		\end{equation*}
		Moreover, the conclusions of Lemma \ref{le.trans2} hold.
	\end{corollary}
	\begin{proof}
		Rewrite \eqref{tM-4} as $H_2(s;k_1):=
		F(s)-2k_1\big(\rho_0^*-X(0,s,k_2)\big)=0$, where $s=\omega^2$. 
		
		For any fixed $c_-$ and $c_+$ satisfying $0<c_-<2\rho_0^*/\alpha^2<c_+$, as $k_1\to\infty$, we have
		\begin{equation*}
			\frac{H_2(c_- k_1;k_1)}{k_1}
			\to
			\alpha^2c_-  -2\rho_0^*<0,\;\frac{H_2(c_+ k_1;k_1)}{k_1}
			\to
			\alpha^2c_+ -2\rho_0^*>0.
		\end{equation*}
		Thus, for sufficiently large $k_1$, we have $H_2(c_-k_1;k_1)<0<H_2(c_+k_1;k_1)$, such that  $H_2(s;k_1)=0$ admits at least one solution in $(c_-k_1,c_+k_1)$. Further, from \eqref{eq.XY1} and \eqref{eq.XY2}, one can check that $\pa H_2(s;k_1)/\pa s\to\alpha^2>0$ as $k_1\to\infty$ uniformly for $s\ge c_-k_1$. Hence, for sufficiently large $k_1$, $H_2(s;k_1)=0$ admits a unique solution in $(c_-k_1,\infty)$, which is denoted by $(\omega_H^*)^2$. 
		
		Let $\omega$ be any other positive solution of \eqref{tM-4}. If $\omega$ were unbounded as $k_1\to\infty$, then \eqref{tM-4} implies that $\frac{\omega^2}{k_1} \to\frac{2\rho_0^*}{\alpha^2}$, and $\omega^2>c_-k_1$ for sufficiently large $k_1$, which contradicts the uniqueness of $(\omega_H^*)^2\in(c_-k_1,\infty)$. While if $\omega$ is bounded, from \eqref{tM-4} and \eqref{tau}, we have
		\begin{equation*}
			\rho_0^*-X(0,\omega^2,k_2)
			=
			\frac{F(\omega^2)}{2k_1}\to0,\;\;\tau
			=
			\frac{\alpha-Y(0,\omega^2,k_2)}
			{\rho_0^*-X(0,\omega^2,k_2)}
			\to\infty.
		\end{equation*}
		
		Therefore, it follows from \eqref{tM-4} and \eqref{tau} that
		\begin{equation*}
			\frac{(\omega_H^*)^2}{k_1}
			\to\frac{2\rho_0^*}{\alpha^2},
			\;
			\omega_H^*\to\infty,\;\tau_H^*
			:=
			\frac{\alpha-Y(0,(\omega_H^*)^2,k_2)}
			{\rho_0^*-X(0,(\omega_H^*)^2,k_2)}
			\to\frac{\alpha}{\rho_0^*}\;\text{as $k_1\to\infty$},
		\end{equation*}
		where $\tau_H^*$ is the smallest critical value for sufficiently large $k_1$, because all the other roots correspond to critical values tending to $\infty$. Thus, system \eqref{eq-g2} has no solution in $\tau\in(0,\tau_H^*)$ and admits the unique solution $(\omega_H^*)^2$ at $\tau=\tau_H^*$.
		
		Finally, from \eqref{I34}, we obtain $I_3\to2(\rho_0^*)^2>0$, $I_4\to0$.  Hence, for any $k_2\in(0,\infty)$, when $k_1>\mathcal{K}_2^{4,a}(k_2)$ for some sufficiently large $\mathcal{K}_2^{4,a}(k_2)\ge\mathcal{K}_1^{4,a}(k_2)$, the conclusions of Lemma \ref{le.trans2} hold.
	\end{proof}
	
	\subsection{Stability and Hopf bifurcations for Cases 4--6}
	In this subsection, we present the stability and bifurcation results for Cases 4--6 with sufficiently small $\ep>0$, which are given by
	\begin{align}
		\makebox[2.2cm][l]{\text{(Case 4)}}\quad
		&\left\{
		\begin{aligned}
			\alpha\frac{\partial u_i}{\partial t}
			&= \ep\dfrac{\partial^2 u_i}{\partial x^2}
			+ \dfrac{1}{\ep}f(u_i,v_i)
			+ k_1\int_{-\infty}^{t}\dfrac{1}{\tau}
			e^{-\frac{t-s}{\tau}}
			\bigl(u_j(x,s)-u_i(x,s)\bigr)\,ds,\\
			\dfrac{\partial v_i}{\partial t}
			&= d\dfrac{\partial^2 v_i}{\partial x^2}
			+ g(u_i,v_i)+k_2(v_j-v_i),
		\end{aligned}
		\right.
		\label{eq-d4}\\
		\makebox[2.2cm][l]{\text{(Case 5)}}\quad
		&\left\{
		\begin{aligned}
			\alpha\frac{\partial u_i}{\partial t}
			&= \ep\dfrac{\partial^2 u_i}{\partial x^2}
			+ \dfrac{1}{\ep}f(u_i,v_i)
			+ k_1\left(
			\int_{-\infty}^{t}\dfrac{1}{\tau}
			e^{-\frac{t-s}{\tau}}u_j(x,s)\,ds-u_i
			\right),\\
			\dfrac{\partial v_i}{\partial t}
			&= d\dfrac{\partial^2 v_i}{\partial x^2}
			+ g(u_i,v_i)+k_2(v_j-v_i),
		\end{aligned}
		\right.
		\label{eq-d5}\\
		\makebox[2.2cm][l]{\text{(Case 6)}}\quad
		&\left\{
		\begin{aligned}
			\alpha\frac{\partial u_i}{\partial t}
			&= \ep\dfrac{\partial^2 u_i}{\partial x^2}
			+ \dfrac{1}{\ep}f(u_i,v_i)
			+ k_1\left(
			u_j-\int_{-\infty}^{t}\dfrac{1}{\tau}
			e^{-\frac{t-s}{\tau}}u_i(x,s)\,ds
			\right),\\
			\dfrac{\partial v_i}{\partial t}
			&= d\dfrac{\partial^2 v_i}{\partial x^2}
			+ g(u_i,v_i)+k_2(v_j-v_i),
		\end{aligned}
		\right.
		\label{eq-d6}
	\end{align}
	where $\tau\in(0,\infty)$, $i,j\in\{1,2\}$, $i\neq j$; $x\in(0,1)$, $t>0$.
	
	\begin{figure}[!t]
		\centering
		\begin{tikzpicture}
			\fill [fill=green!20]  (0,4) --  (0,1) parabola (2.4,4)--cycle;
			\fill [fill=magenta!15]  (0,0) --  (0,1) parabola (2.4,4)--(2.5,4)--(2.5,0)--cycle;
			\fill [fill=magenta!15]  (2.5,0) rectangle (5,4);
			\fill [fill=magenta!15]  (5,0) rectangle (7.5,4);
			
			\fill [red!30]  (0,1) parabola (2.4,4)--(4.9,4) parabola bend (0,1) (0,1);
			
			\fill [red!30]  (3,2.1) parabola bend (5.25,2.5) (5.25,2.5) parabola (7.5,2) -- (7.5,4) -- (4.9,4);
			\draw (0,1) parabola (2.4,4);
			\draw[densely dashed] (2.5,4)--(2.5,0) node [below] {$\frac{\rho_0^*}{2}$};
			\fill [orange!30]
			(0,1) -- (1,1)
			.. controls (1.6,0.5) and (1.6,0.2) .. (1.2,0)-- (0,0) -- (0,1) -- cycle;
			
			\draw
			(0,1) -- (1,1)
			.. controls (1.6,0.5) and (1.6,0.2) .. (1.2,0);
			\node[font=\footnotesize,below] at (1.5,0) {$\mathcal{K}_1^{4,a}(k_2)$};
			\draw (0,1) node[left] {$k_2^*$};
			\node[font=\footnotesize,above] at (4.9,4)  {$\wt\zeta_4(k_1)$};
			
			\fill [red!30]  (6,0) parabola  (4.6,2.48) -- (4.6,4) -- (7.5,4)--(7.5,0)--cycle;
			\draw (6,0) parabola  (4.6,2.48);
			\node[font=\footnotesize,below] at (6,0) {$\mathcal{K}_2^{4,a}(k_2)$};
			\node[font=\footnotesize,right] at (7.5,2)  {$M_4(k_1)$};
			\draw[densely dashed] (5,4)--(5,0) node [below] {$\rho_0^*$};
			\draw (0,1) parabola (4.9,4);
			\draw (3,2.1) parabola bend (5.25,2.5) (5.25,2.5) parabola (7.5,2);
			\node[font=\footnotesize,above] at (2.4,4)  {$\zeta(k_1)$};
			\draw[->] (0,0) --(7.55,0) node[right] {$k_1$};
			\draw[->] (0,0) --(0,4.05) node[above] {$k_2$};
			\draw  (0,0) node[left] {$O$};
		\end{tikzpicture}
		\caption{Stability diagram of the synchronized layered steady state $(U(x,\ep),U(x,\ep))$ with respect to \eqref{eq-d4} in the $(k_1,k_2)$ plane. Green: instability region  $\Gamma_1$; the curve $\Gamma$ separates $\Gamma_1$ from $\Gamma_2$ (Theorem \ref{th.2-1}). Orange: sufficient region $\Gamma_S$ for delay-independent stability. Red: Hopf bifurcation region $\Gamma_H$. See Theorem \ref{th-4} for precise definitions.}
		\label{Fig.s-dep4}
	\end{figure}
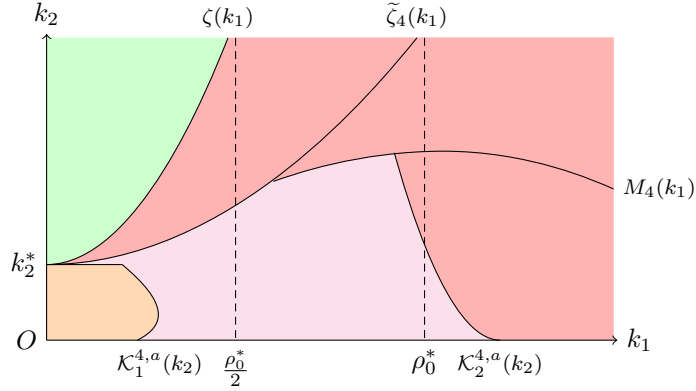
	
	\begin{theorem}[Case 4]\label{th-4}
		Let $\alpha\in(\alpha_0,\infty)$ and $\Gamma_1$ and $\Gamma_2$ be defined by Theorem \ref{th.2-1}. Let $\wt\zeta_4(k_1)$ and $M_4(k_1)$ be $\wt\zeta(k_1)$ and $M(k_1)$ determined by Lemma \ref{le.e2} for $(p,q,K)=(0,-2,k_2)$, respectively. For $k_2\in(0,k_2^*]$, let $\mathcal{K}_1^{4,a}(k_2)$ be given by Lemma \ref{le.e3}. Then, we have the following conclusions $($see Fig. \ref{Fig.s-dep4}$)$:
		\begin{enumerate}
			\item $($Turing instability$)$ If $(k_1,k_2)\in\Gamma_1$, then for any fixed  $\tau\in(0,\infty)$ and sufficiently small $\ep>0$, the synchronized layered steady state $(U(x,\ep),U(x,\ep))$ is unstable with respect to \eqref{eq-d4}.
			\item $($Delay-independent Stability$)$ If $(k_1,k_2)\in\Gamma_S:=\{(k_1,k_2)\in\Gamma_2: k_2\in(0,k_2^*], k_1<\mathcal{K}_1^{4,a}(k_2)\}$, then for any fixed  $\tau\in(0,\infty)$ and sufficiently small $\ep>0$, the synchronized layered steady state $(U(x,\ep),U(x,\ep))$ is stable with respect to \eqref{eq-d4}. 
			\item $($Delay-induced Hopf bifurcation$)$ If $(k_1,k_2)\in\Gamma_2$ and satisfies
			\begin{equation}\label{H4}
				\begin{split}
					&k_1\in(0,\rho_0^*/2),\;k_2\in(\wt\zeta_4(k_1),\zeta(k_1)),\;\text{or}\\
					&k_1\in[\rho_0^*/2,\rho_0^*),\;k_2\in(\wt\zeta_4(k_1),+\infty),\;\text{or}\\
					&k_1\in(\rho_0^*/2,+\infty),\;k_2\in(M_4(k_1),+\infty),\;\text{or}\\
					&k_2\in(0,\infty),\;k_1\in(\mathcal{K}_2^{4,a}(k_2),\infty),
				\end{split}
			\end{equation}
			then for sufficiently small $\epsilon>0$, there exists a critical value $\tau_H(\ep)>0$ with $\tau_H(0)=\tau_H^*>0$, where  $\mathcal{K}_2^{4,a}(k_2)$ is given by Corollary \ref{cor.tM}, $\tau_H^*$ is determined from \eqref{eq-g2} for the case of $(p,q,K)=(0,-2,k_2)$ by Lemma \ref{le.e2} and  Corollary \ref{cor.tM}, such that the synchronized layered steady state $(U(x,\ep),U(x,\ep))$ of \eqref{eq-d4} is asymptotically stable for all $\tau<\tau_H(\ep)$, while unstable when $\tau\in(\tau_H(\ep),\tau_H(\ep)+\kappa)$ for some $\kappa>0$. In particular, system \eqref{eq-d4} undergoes a primary Hopf bifurcation at $\tau=\tau_H(\ep)$, such that a family of bifurcating anti-phase breathers of \eqref{eq-d4} occurs around $(U(x,\ep),U(x,\ep))$.
		\end{enumerate}
	\end{theorem}
	\begin{proof}
		For Case 4, the symmetric and asymmetric SLEP systems are given by \eqref{SLEP-g-1} and \eqref{SLEP-g-2}, which correspond to $(p,q,K)=(0,0,0)$ and $(0,-2,k_2)$, respectively. It suffices to consider the asymmetric sub-problem \eqref{EP-o-2}. Following Theorem \ref{p.con}, Turing instability condition for $\Gamma_1$ and sufficient stability condition for $\Gamma_S$ are given by Proposition \ref{p.pzp} and  Lemma \ref{le.e3}, respectively.
		
		For $(k_1,k_2)\in\Gamma_2$, we shall show how the alternatives in \eqref{H4} are assembled.
		\begin{enumerate}
			\item The first part in Lemma \ref{le.e2} reduces to $\rho_0^*-k_1-X(0,0,k_2)\ge0$, that is, $k_1<\rho_0^*$ and $k_2\ge\wh \zeta(k_1)$.
			\item The second part in Lemma \ref{le.e2} is $\rho_0^*-k_1>0$, $\rho_0^*-X(0,0,k_2)-k_1<0$, and $k_2\in(\wt\zeta_4(k_1),\wh\zeta(k_1))$, which is equivalent to $k_1<\rho_0^*$ and $k_2\in(\wt\zeta_4(k_1),\wh \zeta(k_1))$.
			\item The fourth part in Lemma \ref{le.e2} is $\rho_0^*-2k_1<0$ and $k_2>M_4(k_1)$, which is equivalent to $k_1>\rho_0^*/2$ and $k_2>M_4(k_1)$.
			\item Corollary \ref{cor.tM} provides the last condition of \eqref{H4}.
		\end{enumerate}
		
		Note that the symmetric sub-problem \eqref{EP-o-1} remains stable. Under each alternative, there exists a unique pair of purely imaginary eigenvalues at the critical delay $\tau_H^*$. By Theorem \ref{p.con} and Corollary \ref{cor.hopf}, we can obtain the existence of primary Hopf bifurcation through asymmetric sub-problem \eqref{EP-o-2} at some critical delay $\tau_H(\ep)>0$ with $\tau_H(0)=\tau_H^*$, which generates a family of anti-phase breathers.
	\end{proof}
	
	Cases 5 and 6 share the same asymmetric sub-problem but differ in their symmetric sub-problems, where the symmetric SLEP system in Case 5 remains stable for every $\tau>0$.
	
	\begin{theorem}[Case 5]\label{th-5} 
		Let $\alpha\in(\alpha_0,\infty)$ and $\Gamma_1$ and $\Gamma_2$ be defined by Theorem \ref{th.2-1}. Let $\wt\zeta_5(k_1)$ and $M_5(k_1)$ be $\wt\zeta(k_1)$ and $M(k_1)$ determined by Lemma \ref{le.e2} for $(p,q,K)=(-1,-1,k_2)$, respectively. For $k_2\in(0,k_2^*]$, let $\mathcal{K}_1^{5,a}(k_2)$ and $\mathcal{K}_2^{5,a}(k_2)$ be given by Lemma \ref{le.e3}. Then we have the following conclusions 
		$($see Fig. \ref{Fig.s-dep5}$)$: 
		\begin{enumerate}
			\item $($Turing instability$)$ If $(k_1,k_2)\in\Gamma_1$, then for fixed  $\tau\in(0,\infty)$ and sufficiently small $\ep>0$, the synchronized layered steady state $(U(x,\ep),U(x,\ep))$ is unstable with respect to \eqref{eq-d5}.
			
			\item $($Delay-independent stability$)$ If $(k_1,k_2)\in\Gamma_S:=\{(k_1,k_2)\in\Gamma_2: k_1\in (\rho_0^*,\infty)\; \text{and}\; k_2\in(0,\infty),\; \text{or},\; k_2\in(0,k_2^*]\; \text{and}\; k_1\in (0,\mathcal{K}_1^{5,a}(k_2))\cup(\mathcal{K}_2^{5,a}(k_2),\infty)\}$, then for any fixed  $\tau\in(0,\infty)$ and sufficiently small $\ep>0$, $(U(x,\ep),U(x,\ep))$ is stable with respect to \eqref{eq-d5}.
			
			\item $($Delay-induced Hopf bifurcation$)$ If $(k_1,k_2)\in\Gamma_2$ and satisfies
			\begin{equation}\label{H5}
				\begin{split}
					&k_1\in(0,\rho_0^*/2),\;\;k_2\in(\wt\zeta_5(k_1),\zeta(k_1)),\;\text{or}\\
					&k_1\in\left[\rho_0^*/2,2\rho_0^*/3\right),\;  k_2\in(\wt\zeta_5(k_1),\infty),\;\text{or}\\
					&k_1\in\left(\rho_0^*/2,\rho_0^*\right),\;  k_2\in(M_5(k_1),+\infty),
				\end{split}
			\end{equation}
			then for sufficiently small $\epsilon>0$, there exists a critical value $\tau_H(\ep)>0$ with $\tau_H(0)=\tau_H^*$, where $\tau_H^*$ is determined from \eqref{eq-g2} for the case of $(p,q,K)=(-1,-1,k_2)$ by Lemma \ref{le.e2}, such that $(U(x,\ep),U(x,\ep))$ is stable with respect to system \eqref{eq-d5} when $\tau<\tau_H(\ep)$, while unstable for $\tau\in(\tau_H(\ep),\tau_H(\ep)+\kappa)$ for some $\kappa>0$. In particular, system \eqref{eq-d5} undergoes a primary Hopf bifurcation at $\tau=\tau_H(\ep)$, such that a family of bifurcating anti-phase breathers of \eqref{eq-d5} occurs around $(U(x,\ep),U(x,\ep))$.
		\end{enumerate}	
	\end{theorem}
	\begin{proof}
		For Case 5, the symmetric and asymmetric SLEP systems are given by \eqref{SLEP-g-1} and \eqref{SLEP-g-2}, which correspond to $(p,q,K)=(-1,1,0)$ and $(-1,-1,k_2)$, respectively. The Turing instability conclusion follows as before.
		
		For the symmetric sub-problem, from Lemma \ref{le.no1} and Theorem \ref{p.con}, it remains stable for every $\tau>0$ and sufficiently small $\ep>0$. For the asymmetric sub-problem, the fourth part of Lemma \ref{le.no1} becomes $\rho_0^*+k_1p=\rho_0^*-k_1<0$. That is, if $k_1>\rho_0^*$, then asymmetric SLEP system admits no purely imaginary eigenvalue for any $\tau>0$. Combining this observation with Lemma \ref{le.e3}, we can obtain the stability region $\Gamma_S$, where the synchronized layered steady state $(U(x,\ep),U(x,\ep))$ is stable for any fixed $\tau>0$ and sufficiently small $\ep>0$ due to Theorem \ref{p.con}. Analogous to that of \eqref{H4}, the condition \eqref{H5} can be derived directly from the first, second, and fourth parts of Lemma \ref{le.e2}. By Theorem \ref{p.con} and Corollary \ref{cor.hopf}, there exists a primary Hopf bifurcation through asymmetric sub-problem at some $\tau_H(\ep)>0$ which generates a family of anti-phase breathers.
	\end{proof}

	\begin{figure}[htbp]
		\centering
		\begin{subfigure}[t]{0.48\linewidth}
			\centering
			\begin{tikzpicture}[scale=0.85]
				\fill [fill=green!20]  (0,4) --  (0,1) parabola (2.4,4)--cycle;
				\fill [fill=magenta!15]  (0,0) --  (0,1) parabola (2.4,4)--(7.5,4)--(7.5,0)--cycle;
				\fill [orange!30]  (5,0) rectangle (7.5,4);
				\fill [red!30]  (0,1) ..  controls (2,1.05) and (2.8,1.2) .. (3.3,4) -- (2.4,4) parabola bend (0,1) (0,1);
				\fill [red!30]  (2.9,2.5) ..controls (3.6,2.55) and (4.6,2.8) .. (4.95,4) -- (3.3,4);
				\fill [orange!30]
				(0,1)
				.. controls (1.25,1) and (3,{23/30}) .. (3,0)
				-- (0,0) -- cycle;
				
				\draw
				(0,1)
				.. controls (1.25,1) and (3,{23/30}) .. (3,0);
				\fill [orange!30]
				(5,0.6)
				.. controls ({169/36},{61/135}) and (4.5,{23/90}) .. (4.5,0)
				-- (5,0) -- cycle;
				
				\draw
				(5,0.6)
				.. controls ({169/36},{61/135}) and (4.5,{23/90}) .. (4.5,0);
				\node[font=\footnotesize,below] at (1.5,0.9)  {$\mathcal{K}_1^{5,a}$};
				\node[font=\footnotesize,below] at (4.5,0)  {$\mathcal{K}_2^{5,a}$};
				\draw (5,4)--(5,0.6);
				\draw[densely dashed] (5,0.6)--(5,0) node [below] {$\rho_0^*$};
				\draw (0,1) node[left] {$k_2^*$};
				\draw[->] (0,0) --(7.55,0) node[right] {$k_1$};
				\draw[->] (0,0) --(0,4.05) node[above] {$k_2$};
				\draw (0,1) parabola (2.4,4);
				\draw  (0,1) ..  controls (2,1.05) and (2.8,1.2) .. (3.3,4);
				\draw (2.9,2.5) ..controls (3.6,2.55) and (4.6,2.8) .. (4.95,4);
				\draw[densely dashed] (2.5,4)--(2.5,0) node [below] {$\frac{\rho_0^*}{2}$};
				\draw[densely dashed] (3.33,4) -- (3.33,0) node[below]{$\frac{2\rho_0^*}{3}$};
				\draw  (0,0) node[left] {$O$};
				\node[font=\footnotesize,above] at (4.9,4)  {$M_5(k_1)$};
				\node[font=\footnotesize,above] at (2.4,4)  {$\zeta(k_1)$};
				\node[font=\footnotesize,above] at (3.4,4)  {$\wt\zeta_5(k_1)$};
			\end{tikzpicture}
			\caption{System \eqref{eq-d5} (Case 5)}
			\label{Fig.s-dep5}
		\end{subfigure}\;\;\;\;
		\begin{subfigure}[t]{0.48\linewidth}
			\centering
			\begin{tikzpicture}[scale=0.85]
				\fill [fill=green!20]  (0,4) --  (0,1) parabola (2.4,4)--cycle;
				\fill [fill=magenta!15]  (0,0) --  (0,1) parabola (2.4,4)--(7.5,4)--(7.5,0)--cycle;
				\fill [red!30]  (0,1) ..  controls (2,1.05) and (2.8,1.2) .. (3.3,4) -- (2.4,4) parabola bend (0,1) (0,1);
				\fill [red!30]  (2.9,2.5) ..controls (3.6,2.55) and (4.6,2.8) .. (4.95,4) -- (3.3,4);
				\fill [red!30]  (6.2,0) -- (7.5,0) -- (7.5,4) -- (6.2,4);
				\draw (6.2,4)--(6.2,0);
				\node [font=\footnotesize] at (6.3,-0.3) {$\wh k_1^*$};
				\fill [orange!30]
				(0,1)
				.. controls (1.25,1) and (3,{23/30}) .. (3,0)
				-- (0,0) -- cycle;
				
				\draw
				(0,1)
				.. controls (1.25,1) and (3,{23/30}) .. (3,0);
				\fill [orange!30]
				(5,0.6)
				.. controls ({169/36},{61/135}) and (4.5,{23/90}) .. (4.5,0)
				-- (5,0) -- cycle;
				
				\draw
				(5,0.6)
				.. controls ({169/36},{61/135}) and (4.5,{23/90}) .. (4.5,0);
				\fill [orange!30]  (5,0) -- (5.7,0) --(5.7,4) --(5,4)--cycle;
				\draw (5.7,4) -- (5.7,0) node [font=\footnotesize,below] {$\mathcal{K}_1^{6,s}$};
				\node[font=\footnotesize,below] at (1.5,0.9)  {$\mathcal{K}_1^{6,a}$};
				\node[font=\footnotesize,below] at (4.5,0) {$\mathcal{K}_2^{6,a}$};
				\draw (0,1) node[left] {$k_2^*$};
				\draw[->] (0,0) --(7.55,0) node[right] {$k_1$};
				\draw[->] (0,0) --(0,4.05) node[above] {$k_2$};
				\draw  (0,0) node[left] {$O$};
				\draw (5,4)--(5,0.6);
				\draw[densely dashed] (5,0.6)--(5,0) node [below] {$\rho_0^*$};
				\draw (0,1) parabola (2.4,4);
				\node[font=\footnotesize,above] at (4.9,4)  {$M_6(k_1)$};
				\node[font=\footnotesize,above] at (2.4,4)  {$\zeta(k_1)$};
				\node[font=\footnotesize,above] at (3.4,4)  {$\wt\zeta_6(k_1)$};
				\draw  (0,1) ..  controls (2,1.05) and (2.8,1.2) .. (3.3,4);
				\draw (2.9,2.5) ..controls (3.6,2.55) and (4.6,2.8) .. (4.95,4);
				\draw[densely dashed] (2.5,4)--(2.5,0) node [below] {$\frac{\rho_0^*}{2}$};
				\draw[densely dashed] (3.33,4) -- (3.33,0) node[below]{$\frac{2\rho_0^*}{3}$};
			\end{tikzpicture}
			\caption{System \eqref{eq-d6} (Case 6)}
			\label{Fig.s-dep6}
		\end{subfigure}
		\caption{Stability diagrams of the synchronized layered steady state $(U(x,\ep),U(x,\ep))$ in the $(k_1,k_2)$ plane. (a) Case 5, system \eqref{eq-d5}; (b) Case 6, system \eqref{eq-d6}. Green: delay-independent Turing instability region $\Gamma_1$; the curve $\Gamma$ separates $\Gamma_1$ from $\Gamma_2$ (Theorem \ref{th.2-1}). Orange: sufficient region $\Gamma_S$ for delay-independent stability. Red: sufficient delay-induced instability regions in $\Gamma_2$, defined by \eqref{H5} and \eqref{H6}, respectively, where $\wt\zeta_5=\wt\zeta_6$ and $M_5=M_6$. In Case 5, the instability is only induced by primary asymmetric Hopf bifurcation, whereas in Case 6 the first destabilizing mode is determined from \eqref{pri6}. See Theorems \ref{th-5} and \ref{th-6} for details. In particular, we shall remark that the relative positions of the thresholds in (b) are just chosen for illustration.}
		\label{Fig.s-dep56}
	\end{figure}
	
	\begin{theorem}[Case 6]\label{th-6}
		Let $\alpha\in(\alpha_0,\infty)$ and $\Gamma_1$ and $\Gamma_2$ be defined by Theorem \ref{th.2-1}. Let $\wt\zeta_6(k_1):=\wt\zeta_5(k_1)$ and $M_6(k_1):=M_5(k_1)$, where $\wt\zeta_5(k_1)$ and $M_5(k_1)$ are given by Theorem \ref{th-5}. For $k_2\in(0,k_2^*]$, let $\mathcal{K}_1^{6,s}$, $\mathcal{K}_1^{6,a}(k_2)$ and $\mathcal{K}_2^{6,a}(k_2)$ be given by Lemma \ref{le.e3}. Then we have the following conclusions
		$($see Fig. \ref{Fig.s-dep6}$)$:
		\begin{enumerate}
			\item $($Turing instability$)$ If $(k_1,k_2)\in\Gamma_1$, then for any fixed  $\tau\in(0,\infty)$ and sufficiently small $\ep>0$, the synchronized layered steady state $(U(x,\ep),U(x,\ep))$ is unstable with respect to \eqref{eq-d6}.
			\item $($Delay-independent stability$)$ If $(k_1,k_2)\in\Gamma_S:=\Gamma_S^s \cap \Gamma_S^a$, where $\Gamma_S^s:=(0,\mathcal{K}_1^{6,s})\times(0,\infty)$, $\Gamma_S^a:=\{(k_1,k_2)\in\Gamma_2: k_1\in (\rho_0^*,\infty)\; \text{and}\; k_2\in(0,\infty),\; \text{or},\; k_2\in(0,k_2^*]\; \text{and}\; k_1\in (0,\mathcal{K}_1^{6,a}(k_2))\cup(\mathcal{K}_2^{6,a}(k_2),\infty)\}$, then for any fixed  $\tau\in(0,\infty)$ and sufficiently small $\ep>0$, $(U(x,\ep),U(x,\ep))$ is stable with respect to \eqref{eq-d6}.
			\item $($Delay-induced Hopf bifurcation$)$ If $(k_1,k_2)\in\Gamma_2$ and satisfies
			\begin{equation}\label{H6}
				\begin{split}
					&k_1\in(0,\rho_0^*/2),\;\;k_2\in(\wt\zeta_6(k_1),\zeta(k_1)),\;\text{or}\\
					&k_1\in\left[\rho_0^*/2,2\rho_0^*/3\right),\;  k_2\in(\wt\zeta_6(k_1),\infty),\;\text{or}\\
					&k_1\in\left(\rho_0^*/2,\rho_0^*\right),\;  k_2\in(M_6(k_1),+\infty),\;\text{or}\\
					&k_1\in(\wh k_1^*,\infty),\;k_2\in(0,+\infty),
				\end{split}
			\end{equation}
			then there is a $\tau_H(\ep)>0$ for sufficiently small $\ep>0$, where $\wh k_1^*$ is given by Lemma \ref{le.e2}, such that eigenvalue problem \eqref{D-L-2-a} admits a pair of critical eigenvalues which crosses the imaginary axis transversely from left to right at $\tau=\tau_H(\ep)$. Under the following additional conditions, $\tau_H(\ep)$ is the first critical delay, with
			\begin{equation}\label{pri6}
				\tau_H(0)=
				\begin{cases}
					\begin{split}
						&\tau_s^*,&&\text{if $k_1\in(\max\{\wh k_1^*,\rho_0^*\},\infty)$ and $k_2\in(0,\infty)$},\\
						&\tau_a^*,&&\text{if $k_1<\mathcal{K}_1^{6,s}$ and satisfies at least one of the first three alternatives in  \eqref{H6}},
					\end{split}
				\end{cases}
			\end{equation}
			where $\tau_s^*$ and $\tau_a^*$ are determined from \eqref{eq-g2} in the case of $(p,q,K)=(1,-1,0)$ and $(-1,-1,k_2)$ under the corresponding conditions, respectively. Under one of the conditions in \eqref{pri6}, system \eqref{eq-d6} undergoes a primary Hopf bifurcation at $\tau_H(\ep)$. Namely, $(U(x,\ep),U(x,\ep))$ is stable for $\tau\in(0,\tau_H(\ep))$ and unstable for $\tau\in(\tau_H(\ep),\tau_H(\ep)+\kappa)$ for some $\kappa>0$. Moreover, if $\tau_H(0)=\tau_s^*$ $($resp., $\tau_a^*$$)$, then system \eqref{eq-d6} admits a family of in-phase $($resp., anti-phase$)$ breathers bifurcating at $\tau=\tau_H(\ep)$.
		\end{enumerate}
	\end{theorem}
	\begin{proof}
		For Case 6, the asymmetric SLEP sub-problem is the same as in Case 5, while the symmetric sub-problem corresponds to $(p,q,K)=(1,-1,0)$. Thus, the Turing instability conclusion, the asymmetric stability region $\Gamma_S^a$, and the first three alternatives in \eqref{H6} follow from Theorem \ref{th-5}.
		
		For the symmetric SLEP system, the stability region $\Gamma_S^s$ can be obtained from Lemma \ref{le.e3}. Moreover, by the third part of Lemma \ref{le.e2}, the existence condition of delay-induced instability is $k_1\in(\wh k_1^*,k_1^*)$, where $k_1^*:=2(X(0,0,0)-\rho_0^*)$. While if $k_1\ge k_1^*$, it exactly falls into the first part of Lemma \ref{le.e2}. This leads to the fourth condition of \eqref{H6}. Analogous to Theorems \ref{th-2} and \ref{th-3}, the first critical delay and corresponding mode can be determined by \eqref{pri6}. By Theorem \ref{p.con} and Corollary \ref{cor.hopf}, one can obtain the primary Hopf bifurcations, stability changes, and in-phase or anti-phase breathers.
	\end{proof}
	
	The sufficient stability and Hopf bifurcation conditions for all six coupling configurations are given by Theorems \ref{th-1}--\ref{th-3} and \ref{th-4}--\ref{th-6}. We first collect the results for the symmetric and asymmetric SLEP sub-problems, and then combine them to determine stability and primary Hopf bifurcations for sufficiently small $\ep>0$.
	
	\begin{theorem}\label{th.mode}
		Suppose that $\alpha\in(\alpha_0,\infty)$. Let $\Gamma_1$ and $\Gamma_2$ be given by Theorem \ref{th.2-1}. Denote by $\Gamma_S^{j,\sigma}$ and $\Gamma_H^{j,\sigma}$ the sufficient parameter regions for delay-independent stability and delay-induced instability of the $\sigma$-SLEP system in Case $j$, respectively, where $j\in\{1,\ldots,6\}$,  $\sigma\in\{s,a\}$. Then, $\Gamma_S^{j,\sigma}$ and $\Gamma_H^{j,\sigma}$ are characterized as follows.
		
		\noindent
		{1. Cases 1--3: discrete delays.}
		
		For Case 1,
		\begin{equation*}
			\begin{aligned}
				&\Gamma_S^{1,s}=\{k_1>0\},\quad \Gamma_H^{1,s}=\varnothing,\quad \Gamma_S^{1,a} = \Gamma_2\cap \left\{ 0<k_2\le k_2^*,\; 0<k_1<\mathcal{K}_1^{1,a}(k_2) \right\},\\ &\Gamma_H^{1,a} = \Gamma_2\cap \left( \left\{ 0<k_2\le k_2^*,\; k_1>\mathcal{K}_2^{1,a}(k_2) \right\} \cup \left\{ k_2>k_2^*,\; k_1>\mathcal{K}_1^{1,a}(k_2) \right\} \right).
			\end{aligned}
		\end{equation*}
		
		For Case $j$ $($$j=2,3$$)$,
		\begin{equation*}
			\begin{aligned}
				&\Gamma_S^{j,s}=\left\{0<k_1<\mathcal{K}_1^{j,s}\right\},\quad \Gamma_H^{j,s}=\left\{k_1>\mathcal{K}_2^{j,s}\right\}, \\
				&\Gamma_S^{2,a}=\Gamma_S^{3,a}=\Gamma_2\cap\left\{0<k_2\le k_2^*,\;0<k_1<\mathcal{K}_1^{2,a}(k_2)\right\},\\
				&\Gamma_H^{2,a}=\Gamma_H^{3,a}=\Gamma_2\cap\Big(\left\{0<k_2\le k_2^*,\;k_1>\mathcal{K}_2^{2,a}(k_2)\right\}\\
				&\qquad \qquad \qquad \qquad \quad\cup \left\{k_2>k_2^*,\;k_1\in(\mathcal{K}_1^{2,a}(k_2),2\mathcal{K}_1^{2,a}(k_2)) \cup (\mathcal{K}_2^{2,a}(k_2),\infty)\right\}\Big),
			\end{aligned}
		\end{equation*}
		where $\mathcal K_i^{2,a}(k_2)=\mathcal K_i^{3,a}(k_2)$, $i=1,2$.
		
		\medskip
		\noindent
		{2. Cases 4--6: distributed delays.}
		
		For the symmetric sub-problem of Cases 4--6,
		\begin{equation*}
			\Gamma_S^{4,s}=\Gamma_S^{5,s} = \{k_1>0\}, \quad \Gamma_H^{4,s}=\Gamma_H^{5,s} = \varnothing, \quad \Gamma_S^{6,s}=\left\{0<k_1<\mathcal{K}_1^{6,s}\right\}, \quad \Gamma_H^{6,s}=\left\{k_1>\wh k_1^*\right\}.
		\end{equation*}
		
		For the asymmetric sub-problem of Case 4,
		\begin{equation*}
			\begin{aligned}
				\Gamma_S^{4,a}
				&= \Gamma_2\cap\bigl\{\,0<k_2\le k_2^*,\;0<k_1<\mathcal{K}_1^{4,a}(k_2)\,\bigr\},\\[2pt]
				\Gamma_H^{4,a}
				&= \Gamma_2\cap\Bigl(
				\{\,0<k_1<\rho_0^*/2,\;\wt\zeta_4(k_1)<k_2<\zeta(k_1)\,\}\cup
				\{\,\rho_0^*/2\le k_1<\rho_0^*,\;k_2>\wt\zeta_4(k_1)\,\}\\
				&\qquad\qquad\qquad\cup
				\{\,k_1>\rho_0^*/2,\;k_2>M_4(k_1)\,\}\cup
				\{\,k_2>0,\;k_1>\mathcal{K}_2^{4,a}(k_2)\,\}
				\Bigr).
			\end{aligned}
		\end{equation*}
		where $\wt\zeta_4(k_1)$ and $M_4(k_1)$ are $\wt\zeta(k_1)$ and $M(k_1)$ determined by Lemma \ref{le.e2} for $(p,q,K)=(0,-2,k_2)$, respectively.
		
		For the asymmetric sub-problems of Cases 5 and 6,
		\begin{equation*}
			\begin{aligned}
				&\Gamma_S^{5,a}=\Gamma_S^{6,a}=\Gamma_2\cap\left(\{k_1>\rho_0^*,\;k_2>0\}\cup\left\{0<k_2\le k_2^*,\;k_1\in(0,\mathcal{K}_1^{5,a}(k_2))\cup(\mathcal{K}_2^{5,a}(k_2),\infty)\right\}\right),\\
				&\Gamma_H^{5,a} = \Gamma_H^{6,a}=\Gamma_2\cap\Big(\{0<k_1<\rho_0^*/2,\; \wt\zeta_5(k_1)<k_2<\zeta(k_1)\}\cup \{\rho_0^*/2\le k_1<2\rho_0^*/3,\; k_2>\wt\zeta_5(k_1)\}\\
				&\qquad\qquad\qquad\qquad\qquad \cup \{\rho_0^*/2<k_1<\rho_0^*,\; k_2>M_5(k_1)\}\Big),
			\end{aligned}
		\end{equation*}
		where $\mathcal K_i^{5,a}(k_2) = \mathcal K_i^{6,a}(k_2)$, $i=1,2$; $\wt\zeta_5(k_1)=\wt\zeta_6(k_1)$ and $M_5(k_1)=M_6(k_1)$ are $\wt\zeta(k_1)$ and $M(k_1)$ determined by Lemma \ref{le.e2} for $(p,q,K)=(-1,-1,k_2)$, respectively.
	\end{theorem}
	
	\begin{theorem}\label{th.assembly}
		Suppose that $\alpha\in(\alpha_0,\infty)$. Let $\Gamma_S^{j,\sigma}$ and $\Gamma_H^{j,\sigma}$ be stated by Theorem \ref{th.mode}, where $j\in\{1,\cdots,6\}$ and $\sigma\in\{s,a\}$. Denote the first critical values of symmetric SLEP system \eqref{SLEP-g-1} and asymmetric SLEP system \eqref{SLEP-g-2} by $\tau_s^{j,*}$ and $\tau_a^{j,*}$ for $k_1\in\Gamma_H^{j,s}$ and $(k_1,k_2)\in\Gamma_H^{j,a}$, respectively. Define
		\begin{equation*}
			\wh\Gamma_S^{j,s}:=(\Gamma_S^{j,s}\times\R_+)\cap\Gamma_2,\qquad \wh\Gamma_H^{j,s}:=(\Gamma_H^{j,s}\times\R_+)\cap\Gamma_2.
		\end{equation*}
		Then in Case $j$ $($$j=1,\cdots,6$$)$, we have the following conclusions.
		\begin{enumerate}
			\item {\rm(Turing instability)}
			If $(k_1,k_2)\in\Gamma_1$, then for any fixed $\tau\in(0,\infty)$, the synchronized layered steady state $(U(x,\ep),U(x,\ep))$ is unstable for sufficiently small $\ep>0$.
			
			\item {\rm(Delay-independent stability)}
			If $(k_1,k_2)\in\Gamma_S^j:=\wh\Gamma_S^{j,s}\cap\Gamma_S^{j,a}$, then for any fixed $\tau\in(0,\infty)$, $(U(x,\ep),U(x,\ep))$ is stable for sufficiently small $\ep>0$.
			
			\item {\rm(Delay-induced Hopf bifurcation)}  If $(k_1,k_2)\in\widehat\Gamma_H^{j,s}\cup\Gamma_H^{j,a}$, then for sufficiently small $\ep>0$, at least one of the symmetric and asymmetric sub-problems admits a simple conjugate pair crossing the imaginary axis transversely from left to right at a positive delay. In addition, suppose that one of the following conditions holds:
			\begin{equation}\label{taus}
				(k_1,k_2)\in\wh\Gamma_H^{j,s}\cap\Gamma_S^{j,a},\;\text{or}\; (k_1,k_2)\in\wh\Gamma_H^{j,s}\cap\Gamma_H^{j,a}\;\text{with}\; \tau_s^{j,*}<\tau_a^{j,*},
			\end{equation}
			\begin{equation}\label{taua}
				(k_1,k_2)\in\wh\Gamma_S^{j,s}\cap\Gamma_H^{j,a}\;\text{or}\; (k_1,k_2)\in\wh\Gamma_H^{j,s}\cap\Gamma_H^{j,a}\;\text{with}\; \tau_s^{j,*}>\tau_a^{j,*}.
			\end{equation}
			Then there exists a $\tau_H^j(\ep)>0$ for sufficiently small $\ep>0$, such that $(U(x,\ep),U(x,\ep))$ is stable in $\tau\in(0,\tau_H^j(\ep))$, but unstable in $\tau\in(\tau_H^j(\ep),\tau_H^j(\ep)+\kappa)$ for some small $\kappa>0$. Moreover, system \eqref{eq-o} undergoes a primary Hopf bifurcation at $\tau=\tau_H^j(\ep)$, such that system \eqref{eq-o} admits a family of periodic solutions $(U_1(x,t,s),U_2(x,t,s))$ with minimal positive period $T_s$ near $\tau=\tau_H^{j}(\ep)$. In particular, $\tau_H^j(0)=\tau_s^{j,*}$ or $\tau_H^j(0)=\tau_a^{j,*}$ if $(k_1,k_2)$ satisfies \eqref{taus} or \eqref{taua}, respectively. Moreover, under \eqref{taus}, the bifurcating breathers are in-phase satisfying $U_1(x,t,s)=U_2(x,t,s)$ for any $x\in\Omega$ and $t>0$. While under \eqref{taua}, the bifurcating breathers are anti-phase satisfying $U_1(x,t,s)=U_2(x,t+T_s/2,s)$ for any $x\in\Omega$ and $t>0$.
		\end{enumerate}
	\end{theorem}
	
	\begin{remark}\label{re.pri}
		We summarize the parameter regions of \eqref{taus} and \eqref{taua} where the primary mode can be identified explicitly.
		\begin{enumerate}
			\item For Case $j=1,4,5$, since $\wh\Gamma_H^{j,s}=\varnothing$, we have $\tau_H^j(0)=\tau_a^{j,*}$ if $(k_1,k_2)\in\Gamma_H^{j,a}$. That is, only the asymmetric mode can trigger a Hopf bifurcation in Cases 1, 4, and 5.
			
			\item For Case $j=2,3,6$, by \eqref{pri2}, \eqref{pri3}, and \eqref{pri6}, we have
			\begin{align*}
				\makebox[2cm][l]{\text{(\(j=2\))}}\quad \tau_H^j(0)=
				&\left\{
				\begin{aligned}
					&\tau_s^{2,*}, &&\text{if $k_2\in(0,k_2^*]$ and $k_1\in(\mathcal{K}_2^{2,s},\mathcal{K}_1^{2,a}(k_2))$ $($provided $\mathcal{K}_2^{2,s}<\mathcal{K}_1^{2,a}(k_2)$$)$},\\
					&\tau_a^{2,*}, &&\text{if $k_2\in(0,\infty)$ and $k_1\in(\mathcal{K}_3(k_2),\infty)$, or,}\\
					&\; &&\;\;\;\;\text{$k_2\in(k_2^*,\infty)$ and $k_1\in(\mathcal{K}_1^{2,a}(k_2),2\mathcal{K}_1^{2,a}(k_2))\cap(0,\mathcal{K}_1^{2,s})$},
				\end{aligned}
				\right.\\
				\makebox[2cm][l]{\text{(\(j=3\))}}\quad \tau_H^j(0)=
				&\left\{
				\begin{aligned}
					&\tau_s^{3,*}, &&\text{if $k_2\in(0,\infty)$ and $k_1\in(\mathcal{K}_3(k_2),\infty)$, or,}\\
					&\; &&\;\;\text{$k_2\in(0,k_2^*]$ and $k_1\in(\mathcal{K}_2^{3,s},\mathcal{K}_1^{3,a}(k_2))$ $($provided $\mathcal{K}_2^{3,s}<\mathcal{K}_1^{3,a}(k_2)$$)$},\\
					&\tau_a^{3,*}, &&\text{if $k_2\in(k_2^*,\infty)$ and $k_1\in(\mathcal{K}_1^{3,a}(k_2),2\mathcal{K}_1^{3,a}(k_2))\cap(0,\mathcal{K}_1^{3,s})$},
				\end{aligned}
				\right.\\
				\makebox[2cm][l]{\text{(\(j=6\))}}\quad \tau_H^j(0)=
				&\left\{
				\begin{aligned}
					&\tau_s^{6,*}, &&\text{if $k_1\in(\max\{\wh k_1^*,\rho_0^*\},\infty)$ and $k_2\in(0,\infty)$},\\
					&\tau_a^{6,*}, &&\text{if $k_1<\mathcal{K}_1^{6,s}$ and $(k_1,k_2)\in\Gamma_H^{6,a}$},
				\end{aligned}
				\right.
			\end{align*}
			where   $\mathcal{K}_i^{2,a}(k_2) =\mathcal{K}_i^{3,a}(k_2)$, $i=1,2$; $\mathcal{K}_3(k_2)$ is given by Corollary \ref{cor.K3}. In particular, the interval $(\mathcal{K}_2^{j,s},\mathcal{K}_1^{j,a}(k_2))$ ($j=2,3$) may be nonempty (see Figures \ref{Fig.dia_2} and \ref{Fig.dia_3}).
		\end{enumerate}
	\end{remark}
	
	\section{Numerical Simulations}\label{NS}
	In this section, we shall perform the numerical simulations to support and complement the theoretical results on the stability/instability and Hopf bifurcation through a specific example. Moreover, we can dig out more information on the similarities and differences among Cases 1--6.
	
	The delayed-coupled FitzHugh-Nagumo (FHN) RD model and the uncoupled subsystem are respectively given by
	\begin{equation}\label{BVDP}
		\begin{cases}
			\begin{split}
				\alpha\dfrac{\partial u_1}{\partial t}&=\ep\Delta u_1+\dfrac{1}{\ep}(u_1(u_1+a) (1-u_1)-v_1)+ k_1 \left(L_1[u_2]-L_2[u_1]\right),&&\text{in}\;\Omega\times(0,\infty),\\
				\dfrac{\partial v_1}{\partial t}&=d\Delta v_1+u_1-\gamma v_1+\theta+ k_2(v_2-v_1),&&\text{in}\;\Omega\times(0,\infty),\\
				\alpha\dfrac{\partial u_2}{\partial t}&=\ep\Delta u_2+\dfrac{1}{\ep}(u_2(u_2+a) (1-u_2)-v_2)+ k_1 \left(L_1[u_1]-L_2[u_2]\right),&&\text{in}\;\Omega\times(0,\infty),\\
				\dfrac{\partial v_2}{\partial t}&=d\Delta v_2+u_2-\gamma v_2+\theta+ k_2(v_1-v_2),&&\text{in}\;\Omega\times(0,\infty),\\
				\partial_\nu u_1&=\partial_\nu v_1=\partial_\nu u_2=\partial_\nu v_2=0,&&\text{on}\;\Po\times(0,\infty),
			\end{split}
		\end{cases}
	\end{equation}
	and
	\begin{equation}\label{BVDP-s}
		\begin{cases}
			\begin{split}
				\alpha\dfrac{\partial u}{\partial t}&=\ep\Delta u+\dfrac{1}{\ep}(u(u+a) (1-u)-v),&&\text{in}\;\Omega\times(0,\infty),\\
				\dfrac{\partial v}{\partial t}&=d\Delta v+u-\gamma v+\theta,&&\text{in}\;\Omega\times(0,\infty),\\
				\partial_\nu u&=\partial_\nu v=0,&&\text{on}\;\Po\times(0,\infty),
			\end{split}
		\end{cases}
	\end{equation}
	where $\Omega=(0,1)$. 
	
	Moreover, for numerical simulations, we set
	\begin{equation*}
		a=1,\; \gamma=8,\; \theta=0,\; \ep=0.05,\; d=1,\; \alpha=0.065.
	\end{equation*}
	See \cite{TMN} and the references therein for more details about the FitzHugh-Nagumo model.
	
	As $\ep\to0^+$, the reduced solution $(u^\star(x),v^\star(x))$ of \eqref{BVDP-s} stated in Theorem \ref{th.1} satisfies
	\begin{equation}\label{BVDP-s-r}
		\begin{cases}
			\begin{split}
				0&=u(1+u)(1-u)-v,\\
				0&=\Delta v+u-8 v.
			\end{split}
		\end{cases}
	\end{equation}
	
	To compute the SLEP characteristic equations, we shall provide the key terms stated in Lemma \ref{le.AA1} with respect to reduced system \eqref{BVDP-s-r}. See Table \ref{tab.1} (\cite{WY2}), where we only show the smallest five eigenvalues $\gamma_i$ ($0\leq i\leq4$) for simplicity. In \cite{WY2}, the numerical computations were performed with $c_1^*=c_2^*=2$. Under the normalization in \eqref{c*}, these constants should instead be $c_1^*=c_2^*=2\kappa^*=2.0598$. All numerical results reported below are computed using the corrected values.
	
	\begin{table}[htbp]
		\centering
		\caption{The table for system parameters \cite{WY2}, where $\gamma_i$ ($0\leq i\leq4$) are the smallest five eigenvalues of system \eqref{gamman}, and $\psi_i(x)$ are the corresponding normalized eigenfunctions; $x^*$ is the layer position.}
		\label{tab.1}
		\begin{tabular}{lc@{\hskip 1.2cm}lc@{\hskip 1.2cm}lc}
			\toprule
			\textbf{Term} & \textbf{Value} & \textbf{Term} & \textbf{Value} &
			\textbf{Term} & \textbf{Value} \\
			\midrule
			$v^*$         & 0           & $\gamma_0$    & 8.5379  &    $\psi_0(x^*)$ &  1.0006    \\ 
			$x^*$      & 0.5           & $\gamma_1$    & 18.4181   &  $\psi_1(x^*)$ &  $-5.6483\times10^{-12}$      \\
			$\rho_0^*$       & 0.6522           & $\gamma_2$    & 48.0094  &  $\psi_2(x^*)$ &  $1.4140$         \\
			$\kappa^*$      & 1.0299           & $\gamma_3$    & 97.3569 &  $\psi_3(x^*)$ &  $1.4951\times10^{-10}$          \\
			$c_1^*$      & 2.0598           & $\gamma_4$    & 166.4350    &  $\psi_4(x^*)$ &  $-1.4142$       \\
			$c_2^*$    & 2.0598           &               &       &   &          \\
			\bottomrule
		\end{tabular}
	\end{table}
	
	\subsection{Computations on SLEP systems}\label{sec.NS-SLEP}
	In this subsection, we shall calculate the SLEP system \eqref{SLEP-g} for the coupled FHN model \eqref{BVDP} in Cases 1--6. 
	Through Table \ref{tab.1}, we can derive $ k_2^*= 1.6684$, $\alpha_0=0.0624$ from Theorem \ref{th.2-1}. Thus, we choose $\alpha=0.065>\alpha_0$ so that the layered steady state $U(x,\ep)$ is stable with respect to the uncoupled subsystem \eqref{BVDP-s} for small $\ep=0.05$.
	
	Next we plot the stability diagram of \eqref{BVDP} in Cases 1--6. Firstly, by Theorem \ref{th.2-1}, $\Gamma_1$ is the Turing instability region. Secondly, we recall that $F(\omega^2)=\left( X(0,\omega^2,K)- \rho_0^*\right) ^2+\omega^2 \left(\alpha-Y(0,\omega^2,K) \right)^2$. The SLEP systems \eqref{SLEP-g} admit purely imaginary eigenvalues $\pm i\omega_H^*$ at $\tau_H^*$ if and only if $(\omega_H^*,\tau_H^*)$ satisfies \eqref{eq-g1} or \eqref{eq-g2}, which are respectively equivalent to 
	\begin{align}
		\makebox[2.8cm][l]{\text{Cases 1--3}}\quad
		&\left\{
		\begin{aligned}
			&F(\omega^2)
			= (q^2-p^2) k_1^2+2pk_1\bigl(X(0,\omega^2,K)-\rho_0^*\bigr),\\
			&k_1 q\cos(\omega\tau)
			= X(0,\omega^2,K)- \rho_0^* - k_1 p,\\
			&\sin(\omega\tau)q
			<0,
		\end{aligned}
		\right.
		\label{eq.NS13}\\[1.5ex]
		\intertext{or}
		\makebox[2.8cm][l]{\text{Cases 4--6}}\quad
		&\left\{
		\begin{aligned}
			&F(\omega^2)
			= -p(p+q)k_1^2+(q+2p)k_1\bigl(X(0,\omega^2,K)-\rho_0^*\bigr),\\
			&\tau
			= \frac{Y(0,\omega^2,K)-\alpha}{X(0,\omega^2,K)-\rho_0^*-k_1 p},
		\end{aligned}
		\right.
		\label{eq.NS46}
	\end{align}
	where $\omega>0$ and $\tau>0$; $(p,q,K)$ is given by \eqref{eq:pqk-sym} and \eqref{eq:pqk-asym}.
	
	\begin{figure}[htbp]
		\centering
		\begin{subfigure}[b]{0.32\textwidth}
			\centering
			\includegraphics[width=\linewidth]{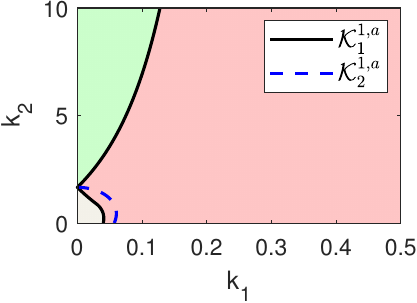}
			\caption{Case 1}	
			\label{Fig.dia_1}
		\end{subfigure}
		\begin{subfigure}[b]{0.32\textwidth}
			\centering
			\includegraphics[width=\linewidth]{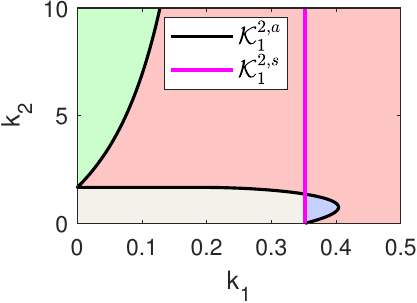}
			\caption{Case 2}	
			\label{Fig.dia_2}
		\end{subfigure}
		\begin{subfigure}[b]{0.32\textwidth}
			\centering
			\includegraphics[width=\linewidth]{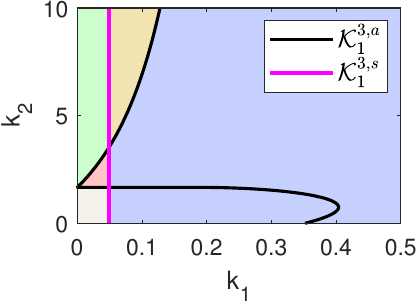}
			\caption{Case 3}	
			\label{Fig.dia_3}
		\end{subfigure}
		\begin{subfigure}[b]{0.32\textwidth}
			\centering
			\includegraphics[width=\linewidth]{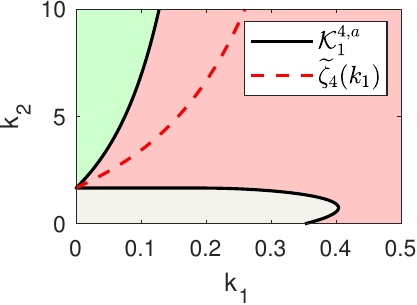}
			\caption{Case 4}	
			\label{Fig.dia_4}
		\end{subfigure}
		\begin{subfigure}[b]{0.32\textwidth}
			\centering
			\includegraphics[width=\linewidth]{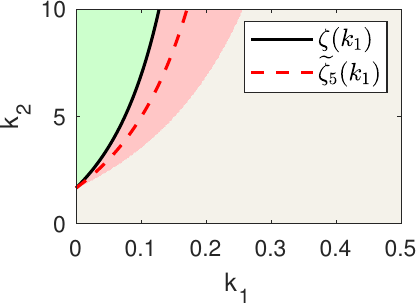}
			\caption{Case 5}	
			\label{Fig.dia_5}
		\end{subfigure}
		\begin{subfigure}[b]{0.32\textwidth}
			\centering
			\includegraphics[width=\linewidth]{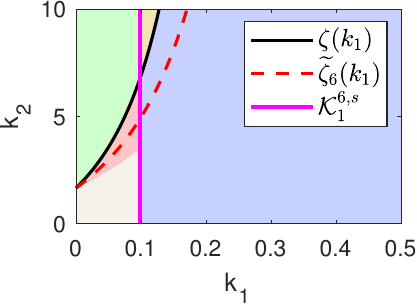}
			\caption{Case 6}	
			\label{Fig.dia_6}
		\end{subfigure}
		\caption{Numerically identified SLEP stability diagram of \eqref{BVDP} in Cases 1--6, where $\mathcal{K}_i^{j,a}=\mathcal{K}_i^{j,a}(k_2)$, $i=1,2$, $j=1,\cdots,4$. Horizontal axis: $k_1$; Vertical axis: $k_2$.  Green: Turing instability region $\Gamma_1$; Grey: stability region $\Gamma_S$; Red: Hopf bifurcation region where the first loss of stability occurs through the asymmetric mode; Blue: Hopf bifurcation region where the first loss of stability occurs through the symmetric mode; Yellow: the overlap of Turing instability region and Hopf bifurcation region in symmetric mode. In particular, $\mathcal{K}_1^{2,a}(k_2)=\mathcal{K}_1^{3,a}(k_2)=\mathcal{K}_1^{4,a}(k_2)$ for $k_2\in(0,\infty)$ in (b), (c), and (d) (see Remark \ref{re.K234}). The purple lines in (b), (c) and (f) are the thresholds of symmetric subsystems, where $\mathcal{K}_1^{2,s}=0.3521>\rho_0^*/2=0.3261$,  $\mathcal{K}_1^{6,s}=2\mathcal{K}_1^{3,s}=0.0987$. In (e) and (f), $\wt\zeta_5(k_1)=\wt\zeta_6(k_1)$.}
		\label{Fig.SLEP_dia}
	\end{figure}
	
	The stability diagrams for Cases 1--6 are presented in Figure \ref{Fig.SLEP_dia} numerically. More precisely, if $(k_1,k_2)\in\Gamma_1$, then it falls into Turing instability region. In $\Gamma_2$, we shall calculate $\omega_H^*$ from the first equation of \eqref{eq.NS13} and \eqref{eq.NS46}, and then substitute $\omega_H^*$ into the other equation (or inequality) to solve the critical values of $\tau>0$. If there are multiple critical values, then the minimum one is denoted by $\tau_H^*$, where the transversality condition $\frac{d\la_R^*}{d\tau}(\tau_H^*)>0$ can be calculated through \eqref{eq.tra1} and \eqref{eq.tra2}. Numerical results indicate that $\frac{d\la_R^*}{d\tau}(\tau_H^*)>0$ holds whenever a pair of purely imaginary eigenvalues $\pm i\omega_H^*$ exists at the first critical value $\tau_H^*$, thereby complementing the analytical sufficient conditions obtained above. Moreover, if $\tau_H^*$ arises from symmetric (resp., asymmetric) SLEP system, then $(k_1,k_2)$ belongs to the symmetric (resp., asymmetric) Hopf bifurcation region. If \eqref{eq.NS13} or \eqref{eq.NS46} has no purely imaginary eigenvalues for any $\tau\in(0,\infty)$, then it falls into stability region. Finally, within the parameter ranges examined, we remark that there is no first critical value $\tau_H^*$ which arises from the symmetric and asymmetric SLEP system simultaneously. 
	
	From Figure \ref{Fig.SLEP_dia}, we have the following observations:
	\begin{enumerate}
		\item {\bf Same delay placement, different delay types.} Although the delay parameter $\tau$ has different physical meanings in the discrete and distributed cases, we can qualitatively compare the structures of their stability diagrams under the same delay placement. The stability diagrams of Cases 1 and 4 (fully delayed coupling) are similar, and the same is true for Cases 3 and 6 (self-delayed coupling). It indicates that, under each of these two delay placements, discrete delays and distributed delays with weak kernels have similar effects on the dynamics of the synchronized layered steady state. However, the stability diagrams of Cases 2 and 5 (cross-delayed coupling) are quite different, where Case 2 exhibits Hopf instability in a large parameter region, but Case 5 maintains stability in a larger region and even exhibits delay-independent stability when the activator coupling strength is sufficiently large.
		
		\item {\bf Same delay type, different delay placements.} For both discrete and distributed delays, the stability diagrams of the fully delayed couplings have relatively simple structures, where Hopf bifurcations can only emerge through the asymmetric mode. For the cross-delayed and self-delayed couplings, they share the same asymmetric sub-problems (see \eqref{eq:pqk-asym}), so the differences are mainly determined by their symmetric sub-problems. More precisely, from Figure \ref{Fig.SLEP_dia}, we observe that for both delay types, fully delayed coupling only admits asymmetric Hopf bifurcation regions. For cross-delayed coupling, the asymmetric mode dominates in most of the Hopf bifurcation regions  (although a small symmetric Hopf bifurcation region exists in Case 2). In contrast, for self-delayed coupling, the symmetric mode is dominant in most of the Hopf bifurcation regions, while small asymmetric Hopf bifurcation regions also exist. In particular, in Cases 3 and 6, the Turing instability region overlaps with the parameter region admitting a symmetric Hopf crossing. This suggests a possible interaction between stationary symmetry-breaking and delay-induced oscillatory instabilities.
		
		\item {\bf Delayed coupling versus non-delayed coupling.} Comparing Figure \ref{Fig.SLEP_dia} with Figure \ref{Fig.1}, we observe that delays have significant effects on the stability of the synchronized layered steady state. More precisely, by Theorem \ref{th.2-1}, the activator coupling strength $k_1$ has a stabilizing effect in the non-delayed coupled system, where the synchronized layered steady state becomes stable when $k_1$ exceeds the Turing instability threshold. However, once delay is introduced into the coupling, large $k_1$ can promote delay-induced oscillatory instability (except for the cross-distributed-delayed coupling). On the other hand, although delays cannot eliminate the Turing instability induced by inter-layer diffusion, they can interact with inter-layer diffusion to produce more complex dynamical behaviors (see Figure \ref{Fig.TH_3} in the next subsection).
		
	\end{enumerate}
	
	\begin{figure}[htbp]
		\centering
		\begin{subfigure}[b]{0.4\textwidth}
			\centering
			\includegraphics[width=\linewidth]{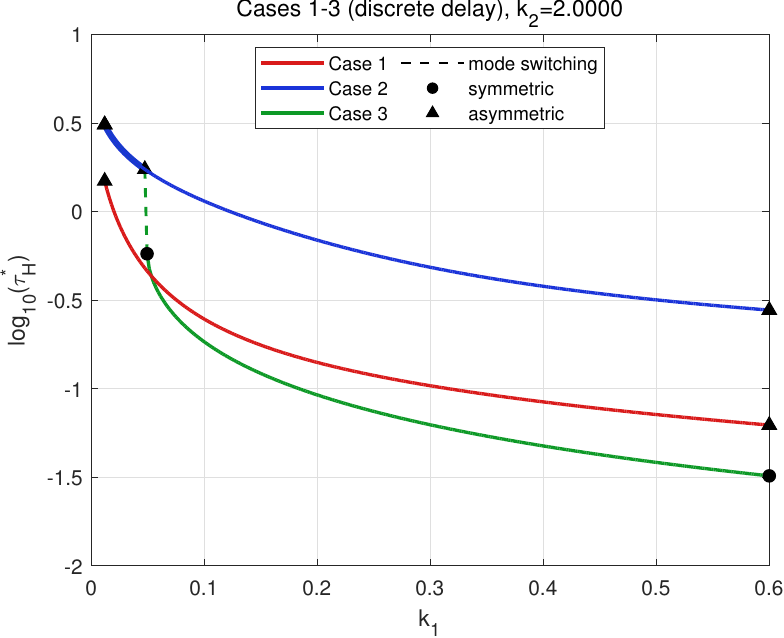}
			\caption{Cases 1--3 with $k_2=2$}	
			\label{Fig.tau_1}
		\end{subfigure}
		\begin{subfigure}[b]{0.4\textwidth}
			\centering
			\includegraphics[width=\linewidth]{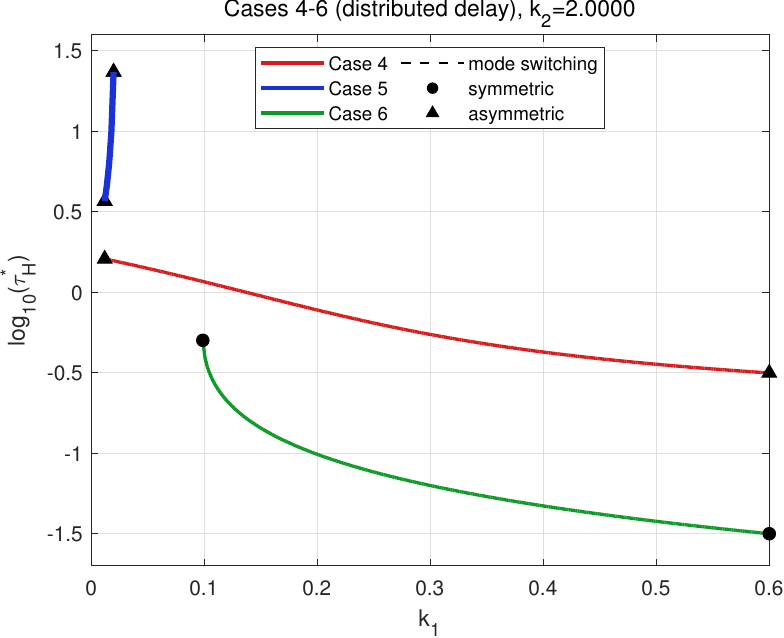}
			\caption{Cases 4--6 with $k_2=2$}	
			\label{Fig.tau_2}
		\end{subfigure}
		\begin{subfigure}[b]{0.4\textwidth}
			\centering
			\includegraphics[width=\linewidth]{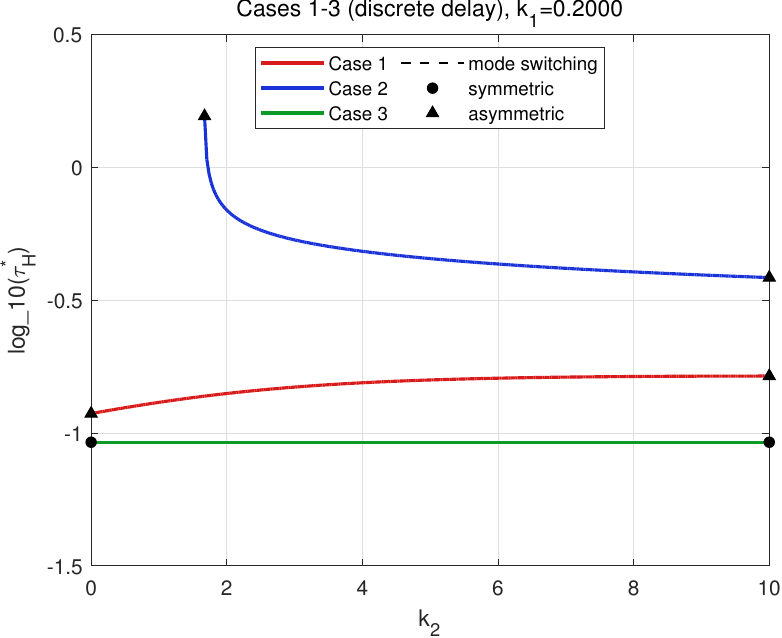}
			\caption{Cases 1--3 with $k_1=0.2$}	
			\label{Fig.tau_3}
		\end{subfigure}
		\begin{subfigure}[b]{0.4\textwidth}
			\centering
			\includegraphics[width=\linewidth]{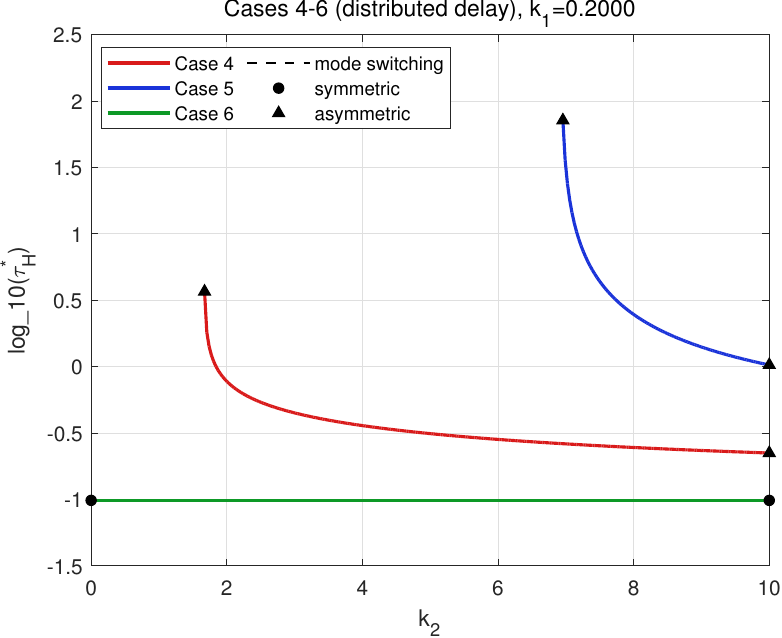}
			\caption{Cases 4--6 with $k_1=0.2$}	
			\label{Fig.tau_4}
		\end{subfigure}
		\caption{Graphs of $\tau_H^*$ in Cases 1--6. Vertical axis: $\log_{10}(\tau_H^*)$; Horizontal axis: $k_1$ ((a), (b)) or $k_2$ ((c), (d)). Red: Case 1 or 4; Blue: Case 2 or 5; Green: Case 3 or 6. In (a), the dashed curves indicate that the mode determining the first critical delay $\tau_H^*$ switches between the symmetric and asymmetric modes in Case 3. In (b), the Case 6 curve consists of two disconnected branches, separated by a parameter interval of delay-independent stability.  Moreover, in (a) (resp., (b)) the thick blue curve indicates an overlapping asymmetric critical-delay branch shared by Cases 2 and 3 (resp., 5 and 6).}
		\label{Fig.tau}
	\end{figure}
	
	The graphs of $\tau_H^*$ (more precisely, $\log_{10}(\tau_H^*)$) with respect to $k_1$ or $k_2$ are shown in Figure \ref{Fig.tau} through \eqref{eq.NS13} and \eqref{eq.NS46}, aiming to numerically analyze the effect of inter-layer coupling strengths on the critical value $\tau_H^*$. In Figures \ref{Fig.tau_1} and \ref{Fig.tau_2}, we set $k_2=2>k_2^*=1.6684$ and $k_1>\xi(k_2)=0.0106$ to explore the relationship between $\tau_H^*$ and $k_1$. From Figure \ref{Fig.SLEP_dia}, as $k_1$ increases, all six cases first pass through the Turing instability region and then enter the Hopf bifurcation region. As $k_1$ further increases, Case 5 enters the stability region, while Case 6 first enters the stability region and then re-enters the Hopf bifurcation region. In Figures \ref{Fig.tau_3} and \ref{Fig.tau_4}, we set $k_1=0.2$ and observe the relationship between $\tau_H^*$ and $k_2$. From Figure \ref{Fig.SLEP_dia}, since $k_1=0.2<\rho_0^*/2=0.3261$, all six cases eventually enter the Turing instability region as $k_2$ increases. Before entering this region, Cases 1, 3, and 6 lie in the Hopf bifurcation regions, while Cases 2, 4, and 5 pass from the stability regions into the Hopf bifurcation regions. According to the modes of primary Hopf bifurcation in different cases, we show them by the corresponding markers in Figure \ref{Fig.tau}.
	
	One can conclude the following facts from Figure \ref{Fig.tau}:
	\begin{enumerate}
		\item {\bf Pattern selection.}
		In Cases 1, 4, and 5, Hopf bifurcations can only occur through the asymmetric mode. Within the parameter ranges examined, when both the symmetric and asymmetric modes can induce Hopf bifurcations, the asymmetric mode dominates in Case 2, while the symmetric mode dominates in Cases 3 and 6. Mode switching may also occur in these three cases under certain parameter conditions. Taking Case 3 as an example, when $k_1\in(\mathcal{K}_1^{3,a}(k_2),\mathcal{K}_1^{3,s})$, the asymmetric mode undergoes a Hopf bifurcation while the symmetric mode does not, leading to the generation of anti-phase breathers (see Figures \ref{Fig.ab-j}--\ref{Fig.ab-l}). As $k_1$ increases, both modes can induce Hopf bifurcations, but the symmetric mode becomes dominant since it reaches its critical delay earlier. Thus, the primary Hopf bifurcation switches to the symmetric mode, leading to in-phase breathers bifurcating from synchronized layered steady state (see Figures \ref{Fig.sb-a}--\ref{Fig.sb-c}).
		
		\item {\bf Sensitivity to time delays.}
		A larger $\tau_H^*$ indicates that the system is more robust against delay-induced instability. In Figures \ref{Fig.tau_1} and \ref{Fig.tau_2}, we observe that the asymmetric critical delays of Cases 3 and 6 coincide with those of Cases 2 and 5, because they share the same asymmetric SLEP system respectively. In most of the parameter ranges examined in Figure \ref{Fig.tau}, the relative robustness can be summarized as
		\begin{equation*}
			\text{Discrete: Case 2 $>$ Case 1 $>$ Case 3;
				Distributed: Case 5 $>$ Case 4 $>$ Case 6.}
		\end{equation*}
		Namely, cross-delayed couplings (Cases 2 and 5) are relatively the most robust against time delays, followed by fully delayed couplings (Cases 1 and 4), while self-delayed couplings (Cases 3 and 6) are the most sensitive to time delays. In particular, an intuitive explanation for these systematic differences is discussed in Section 5.
		
		\item {\bf Dependence of the first critical delay on the inter-layer coupling strengths.}
		When $k_2=2$, from Figures \ref{Fig.tau_1} and \ref{Fig.tau_2}, $\tau_H^*$ is decreasing with respect to $k_1$ in Cases 1--4 and in the symmetric mode of Case 6. However, the asymmetric mode of Case 6 coincides with that of Case 5, where the corresponding critical delay is increasing with respect to $k_1$. In particular, Figure \ref{Fig.tau_2} shows that $\tau_H^*$ increases and eventually disappears as $k_1$ increases in Case 5, reflecting the stabilizing effect of $k_1$. When $k_1=0.2$, by Figure \ref{Fig.SLEP_dia}, the first loss of stability in Cases 3 and 6 is induced by the symmetric mode within the examined range of $k_2$. In Figures \ref{Fig.tau_3} and \ref{Fig.tau_4}, since the corresponding symmetric SLEP systems are independent of $k_2$, the first critical delays $\tau_H^*$ in Cases 3 and 6 remain unchanged as $k_2$ varies. In Cases 2 and 5, $\tau_H^*$ is decreasing with respect to $k_2$, while it is increasing in Case 1 and is decreasing in Case 4. 
	\end{enumerate}
	
	\subsection{Numerical simulations on \eqref{BVDP} in Cases 1--6 for small $\ep>0$}
	
	In this subsection, we shall show the simulations of \eqref{BVDP} in Cases 1--6 to support Theorem \ref{th.assembly} and some numerical observations from SLEP systems in Subsection \ref{sec.NS-SLEP}. 
	
	For numerical purposes, we note that system \eqref{BVDP} in Case 4 can be rewritten as
	\begin{equation*}
		\begin{cases}
			\begin{split}
				\alpha\dfrac{\partial u_1}{\partial t}&=\ep\Delta u_1+\dfrac{1}{\ep}(u_1(u_1+a) (1-u_1)-v_1)+ k_1 \left(w_2-w_1\right),\\
				\dfrac{\partial v_1}{\partial t}&=d\Delta v_1+u_1-\gamma v_1+\theta+ k_2(v_2-v_1),\\
				\alpha\dfrac{\partial u_2}{\partial t}&=\ep\Delta u_2+\dfrac{1}{\ep}(u_2(u_2+a) (1-u_2)-v_2)+ k_1 \left(w_1-w_2\right),\\
				\dfrac{\partial v_2}{\partial t}&=d\Delta v_2+u_2-\gamma v_2+\theta+ k_2(v_1-v_2),\\
				\frac{\pa w_1}{\pa t}&=\frac{1}{\tau}(u_1-w_1),\; \frac{\pa w_2}{\pa t}=\frac{1}{\tau}(u_2-w_2),
			\end{split}
		\end{cases}
	\end{equation*}
	where $w_i(x,t):=\int_{-\infty}^{t}\frac{1}{\tau}e^{-\frac{t-s}{\tau}}u_i(x,s)ds$, $i=1,2$. System \eqref{BVDP} in Cases 5--6 can be represented similarly.
	
	Then, setting $(k_1,k_2)=(0.6,2)$, we shall choose the representative values of $\tau>\tau_H^\ep:=\tau_H(\ep)$, where $\tau_H^\ep$ can be numerically derived by finite difference and Newton iteration with space step $\Delta x=2\times 10^{-3}$. Then the values and expected phenomena (from Figures \ref{Fig.SLEP_dia} and \ref{Fig.tau}) are summarized in Table \ref{tab.tau}. In particular, at $(k_1,k_2)=(0.6,2)$, synchronized layered steady state $(U(x,\ep),U(x,\ep))$ is stable in Case 5.
	
	\begin{table}[htbp]
		\centering
		\caption{Values of the Hopf bifurcation parameter of \eqref{BVDP} in Cases 1--6, where $(k_1,k_2,\ep)=(0.6,2,0.05)$. There is no Hopf critical delay in Case 5.}
		\label{tab.tau}
		\begin{tabular}{lcccc}
			\toprule
			\textbf{Case} & $\tau_H^*$ & $\tau_H^\ep$ & $\tau$ & Observed phenomenon\\
			\midrule
			\textbf{1 (fully discrete delay)} & 0.0624 & 0.0629 & 0.1 & anti-phase breathers  \\
			\textbf{2 (cross-discrete delay)} & 0.2782 & 0.2909 & 0.5 & anti-phase breathers  \\
			\textbf{3 (self-discrete delay)} & 0.0322 & 0.0319 & 0.04 & in-phase breathers \\
			\textbf{4 (fully distributed delay)} & 0.3151 & 0.3432 & 0.4 & anti-phase breathers  \\
			\textbf{5 (cross-distributed delay)} & -- & -- & 10 & synchronized layered steady state   \\
			\textbf{6 (self-distributed delay)} & 0.0316 & 0.0312 & 0.04 & in-phase breathers   \\
			\bottomrule
		\end{tabular}
	\end{table}
	
	\begin{figure}[htbp]	
		\centering
		\begin{subfigure}[b]{0.3\textwidth}
			\centering
			\includegraphics[width=\linewidth]{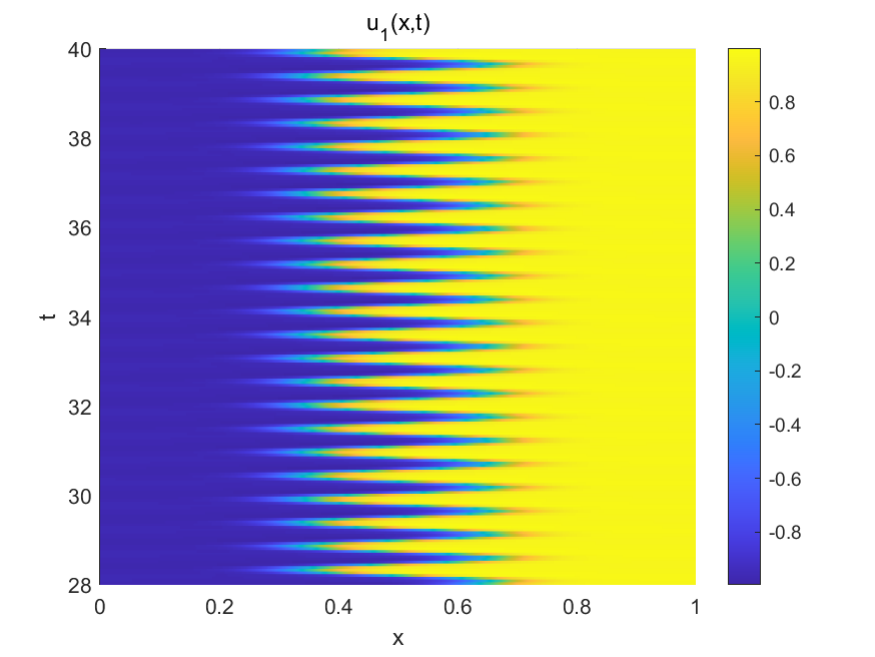}
			\caption{$u_1=u_{1,p}(x,t,\ep)$}	\label{Fig.ab-a}
		\end{subfigure}
		\begin{subfigure}[b]{0.3\textwidth}
			\centering
			\includegraphics[width=\linewidth]{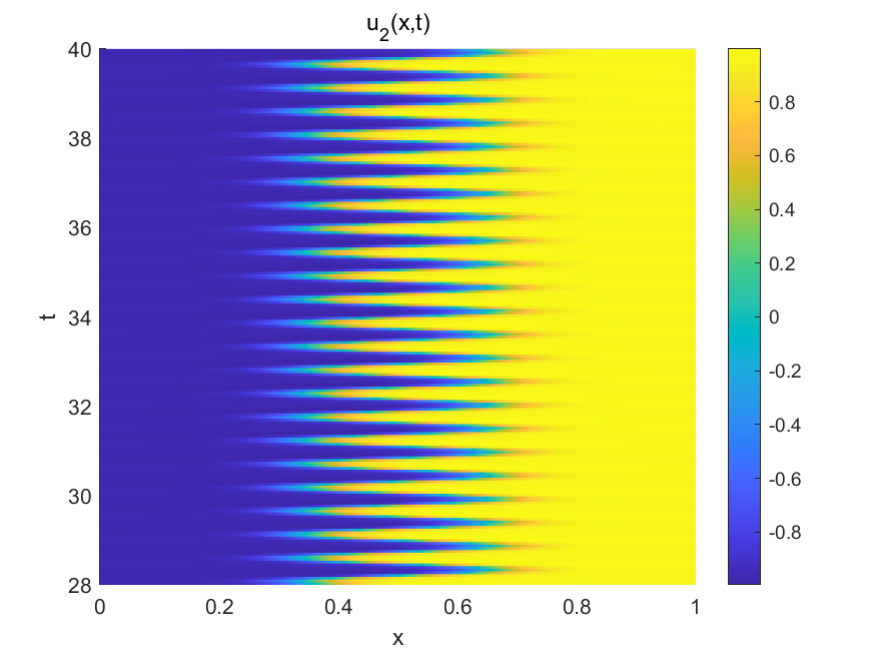}
			\caption{$u_2=u_{2,p}(x,t,\ep)$}
		\end{subfigure}
		\begin{subfigure}[b]{0.3\textwidth}
			\centering
			\includegraphics[width=\linewidth]{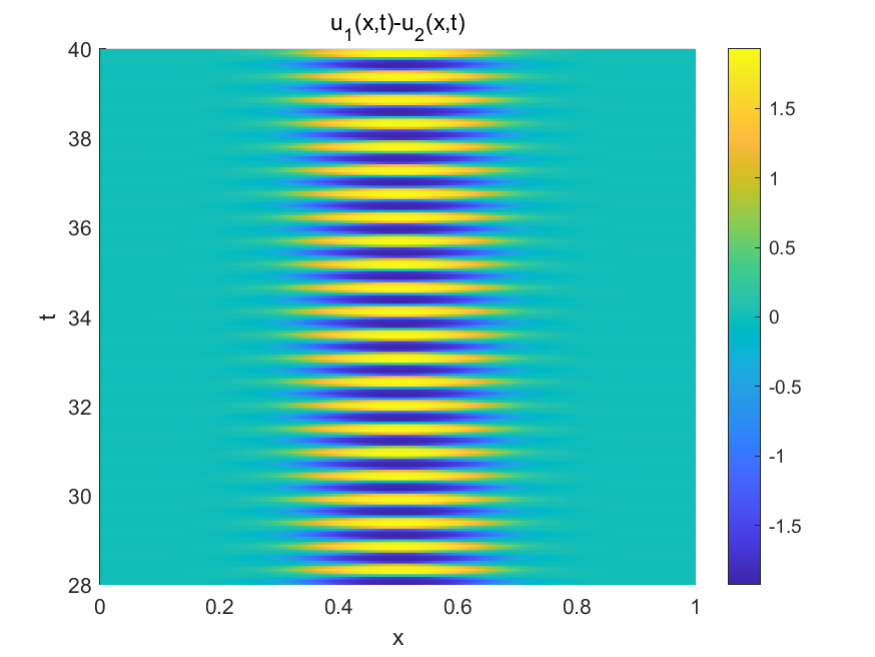}
			\caption{$u_{1,p}(x,t,\ep)-u_{2,p}(x,t,\ep)$}
			\label{Fig.ab-c}
		\end{subfigure}
		\begin{subfigure}[b]{0.3\textwidth}
			\centering
			\includegraphics[width=\linewidth]{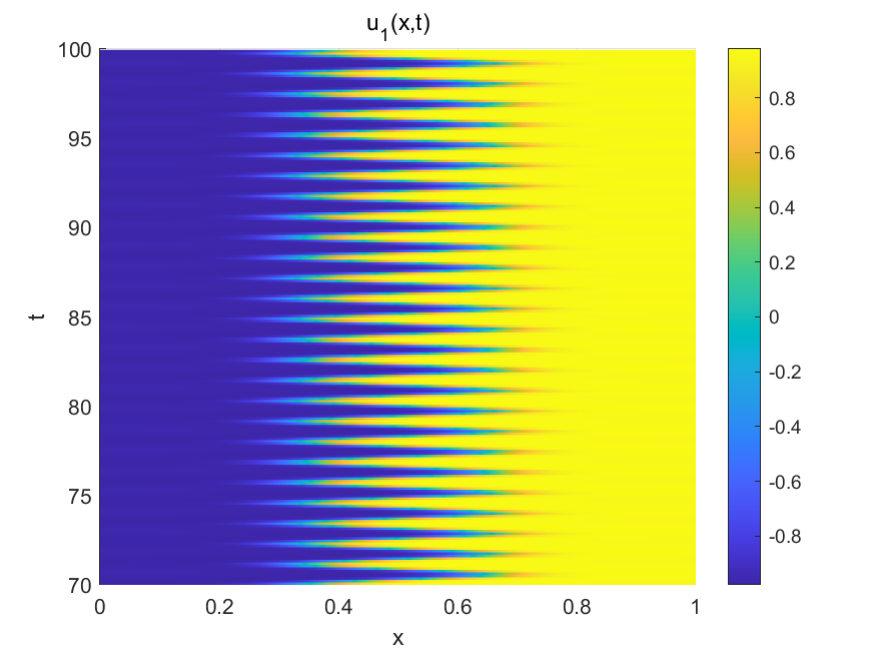}
			\caption{$u_1=u_{1,p}(x,t,\ep)$}
			\label{Fig.ab-d}
		\end{subfigure}
		\begin{subfigure}[b]{0.3\textwidth}
			\centering
			\includegraphics[width=\linewidth]{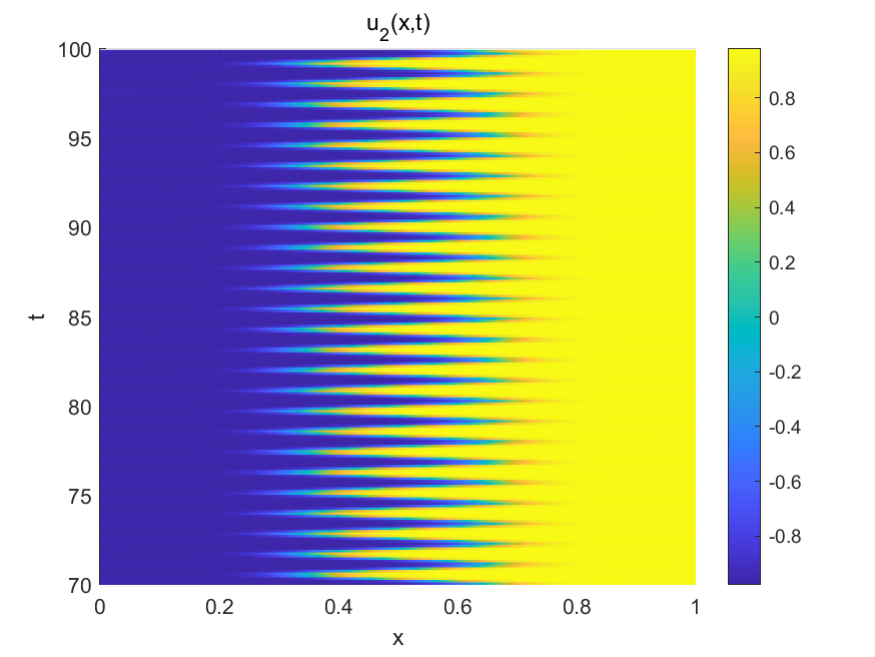}
			\caption{$u_2=u_{2,p}(x,t,\ep)$}
		\end{subfigure}
		\begin{subfigure}[b]{0.3\textwidth}
			\centering
			\includegraphics[width=\linewidth]{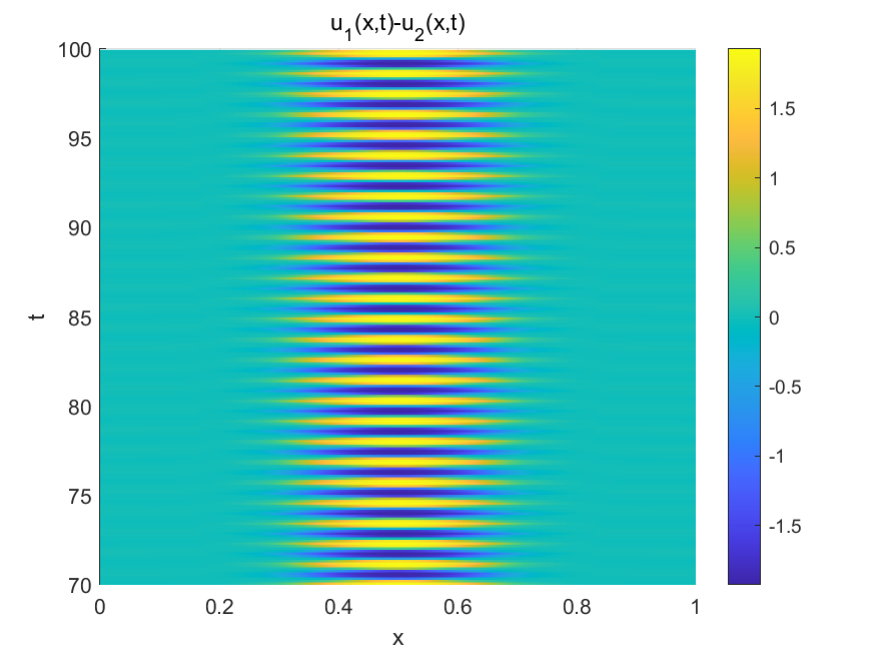}
			\caption{$u_{1,p}(x,t,\ep)-u_{2,p}(x,t,\ep)$}
			\label{Fig.ab-f}
		\end{subfigure}
		\begin{subfigure}[b]{0.3\textwidth}
			\centering
			\includegraphics[width=\linewidth]{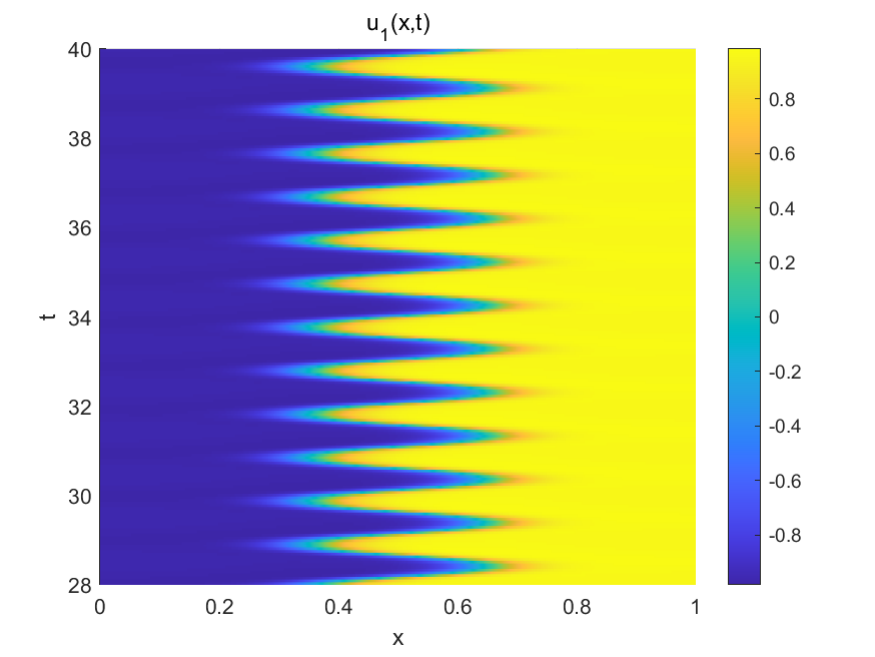}
			\caption{$u_1=u_{1,p}(x,t,\ep)$}	
			\label{Fig.ab-g}
		\end{subfigure}
		\begin{subfigure}[b]{0.3\textwidth}
			\centering
			\includegraphics[width=\linewidth]{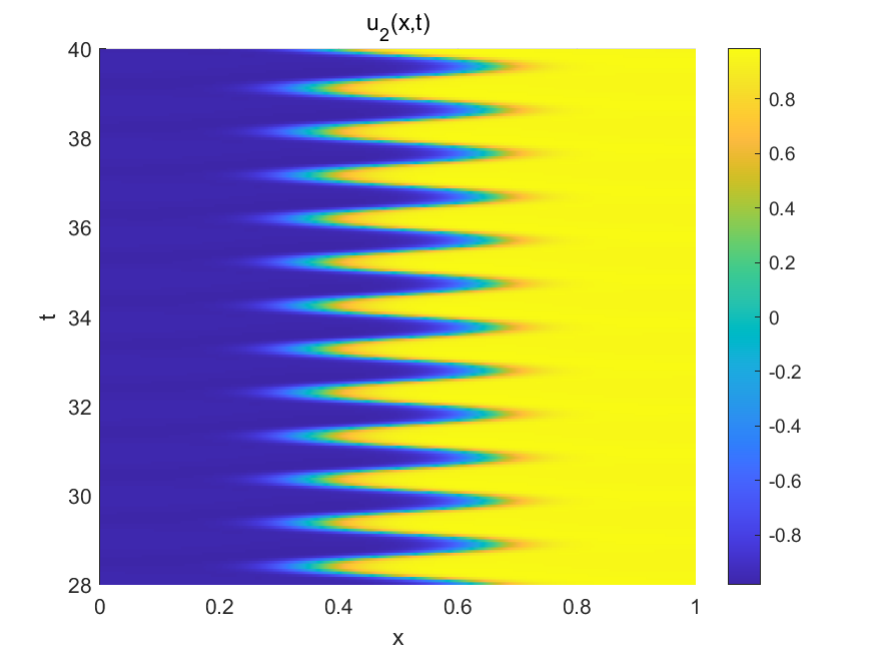}
			\caption{$u_2=u_{2,p}(x,t,\ep)$}
		\end{subfigure}
		\begin{subfigure}[b]{0.3\textwidth}
			\centering
			\includegraphics[width=\linewidth]{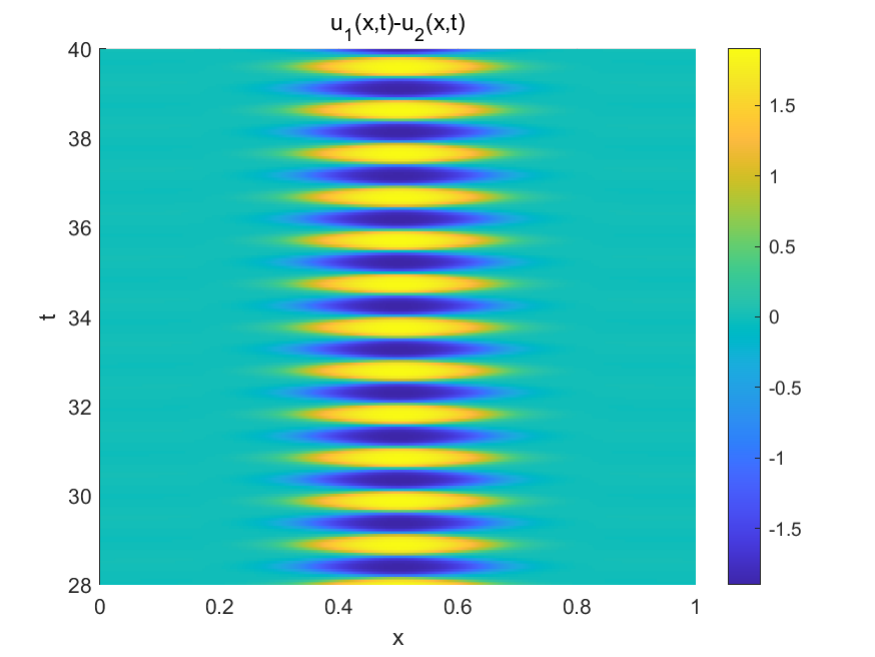}
			\caption{$u_{1,p}(x,t,\ep)-u_{2,p}(x,t,\ep)$}
			\label{Fig.ab-i}
		\end{subfigure}
		\begin{subfigure}[b]{0.3\textwidth}
			\centering
			\includegraphics[width=\linewidth]{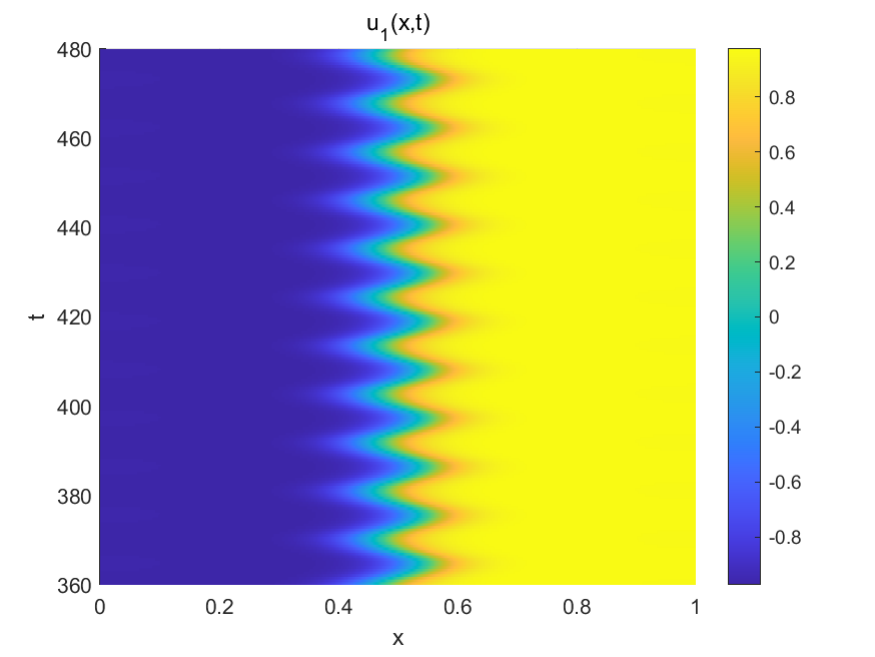}
			\caption{$u_1=u_{1,p}(x,t,\ep)$}
			\label{Fig.ab-j}
		\end{subfigure}
		\begin{subfigure}[b]{0.3\textwidth}
			\centering
			\includegraphics[width=\linewidth]{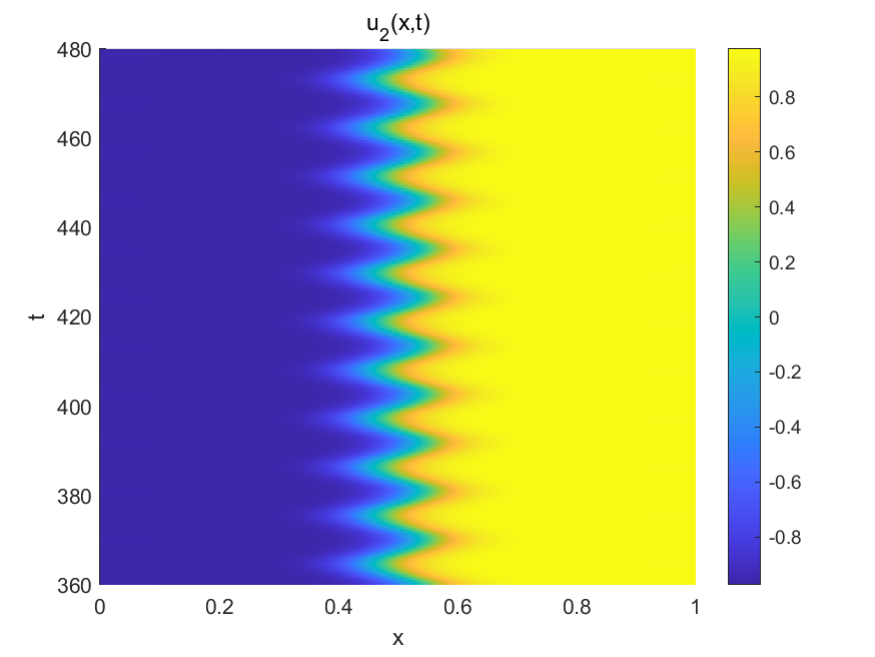}
			\caption{$u_2=u_{2,p}(x,t,\ep)$}
		\end{subfigure}
		\begin{subfigure}[b]{0.3\textwidth}
			\centering
			\includegraphics[width=\linewidth]{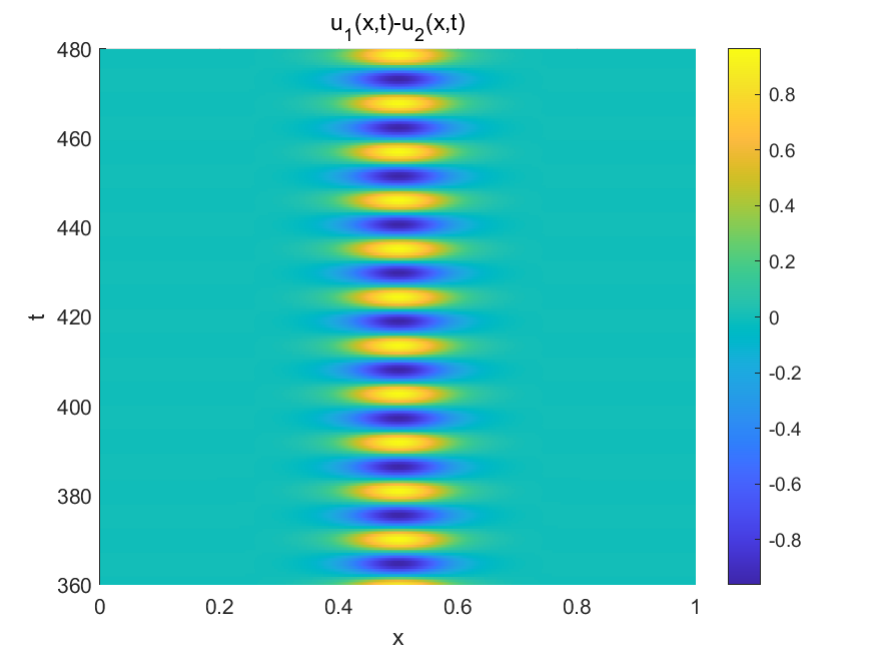}
			\caption{$u_{1,p}(x,t,\ep)-u_{2,p}(x,t,\ep)$}
			\label{Fig.ab-l}
		\end{subfigure}
		\caption{Anti-phase Hopf bifurcating breathers with respect to system \eqref{BVDP} in Cases 1--4.  (a)-(c): Case 1, $(k_1,k_2)=(0.6,2)$,  $\tau=0.1>\tau_H^\ep=0.0629$. (d)-(f): Case 2, $(k_1,k_2)=(0.6,2)$, $\tau=0.5>\tau_H^\ep=0.2909$. (g)-(i): Case 4,  $(k_1,k_2)=(0.6,2)$, $\tau=0.4>\tau_H^\ep=0.3432$. (j)-(l): Anti-phase Hopf-bifurcating breathers appear in Case 3 by mode switching, where $(k_1,k_2)=(0.03,2)$, $\tau=4>\tau_H^\ep=3.5569$, $\Delta t=1\times 10^{-4}$.}
		\label{Fig.ab}
	\end{figure}
	
	\begin{figure}[htbp]	
		\centering
		\begin{subfigure}[b]{0.3\textwidth}
			\centering
			\includegraphics[width=\linewidth]{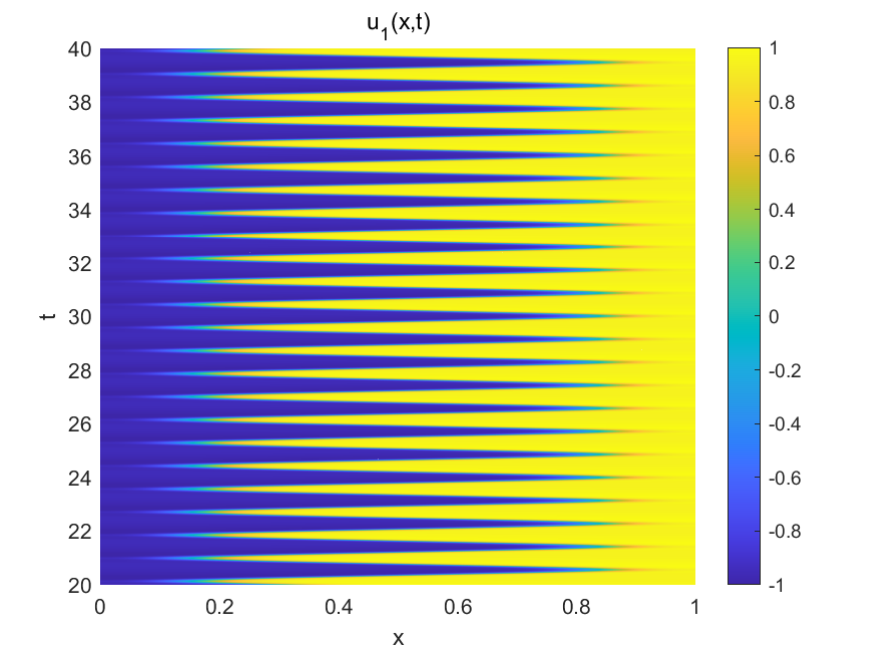}
			\caption{$u_1=u_{1,p}(x,t,\ep)$}
			\label{Fig.sb-a}
		\end{subfigure}
		\begin{subfigure}[b]{0.3\textwidth}
			\centering
			\includegraphics[width=\linewidth]{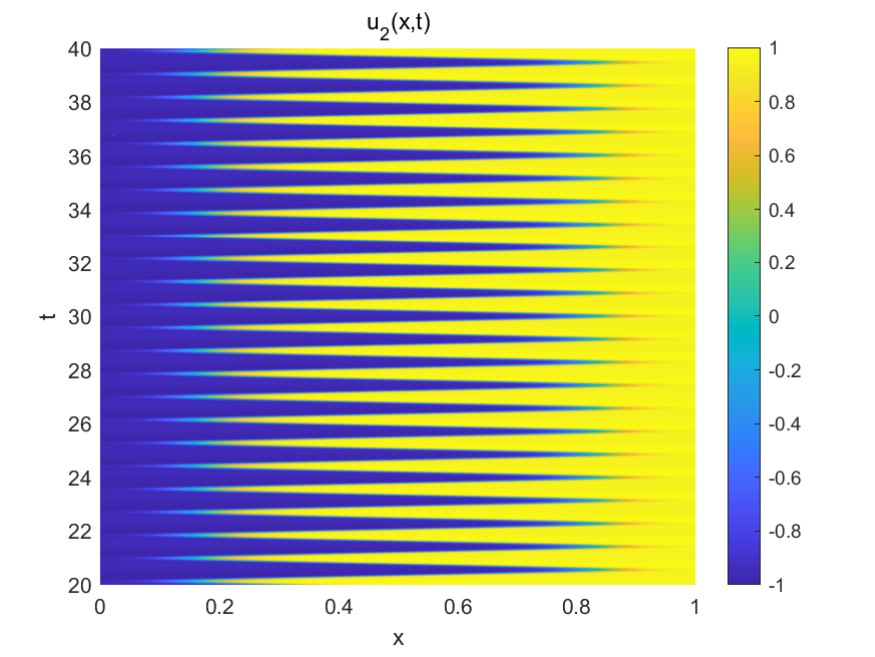}
			\caption{$u_2=u_{2,p}(x,t,\ep)$}
		\end{subfigure}
		\begin{subfigure}[b]{0.3\textwidth}
			\centering
			\includegraphics[width=\linewidth]{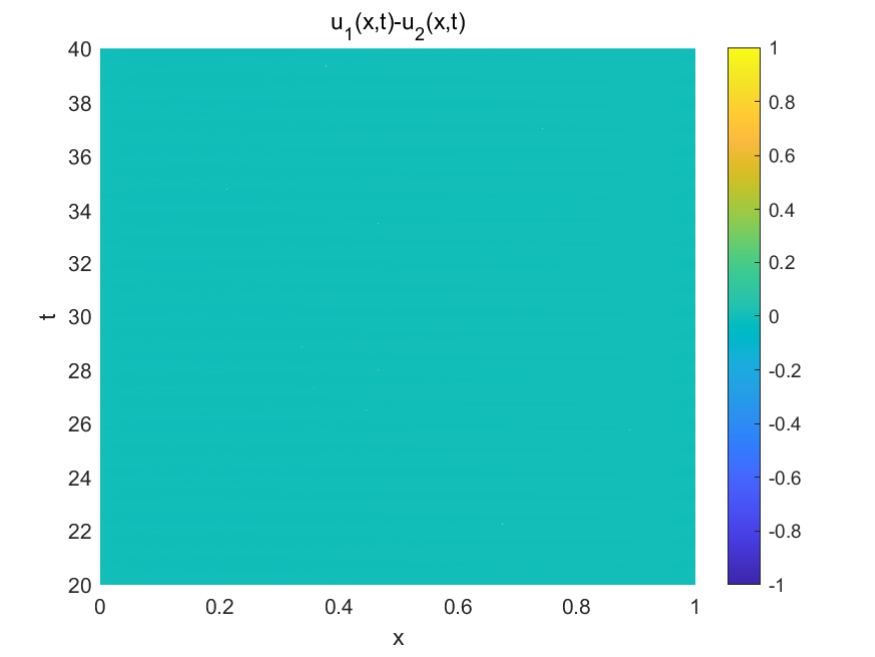}
			\caption{$u_{1,p}(x,t,\ep)-u_{2,p}(x,t,\ep)$}
			\label{Fig.sb-c}
		\end{subfigure}
		\begin{subfigure}[b]{0.3\textwidth}
			\centering
			\includegraphics[width=\linewidth]{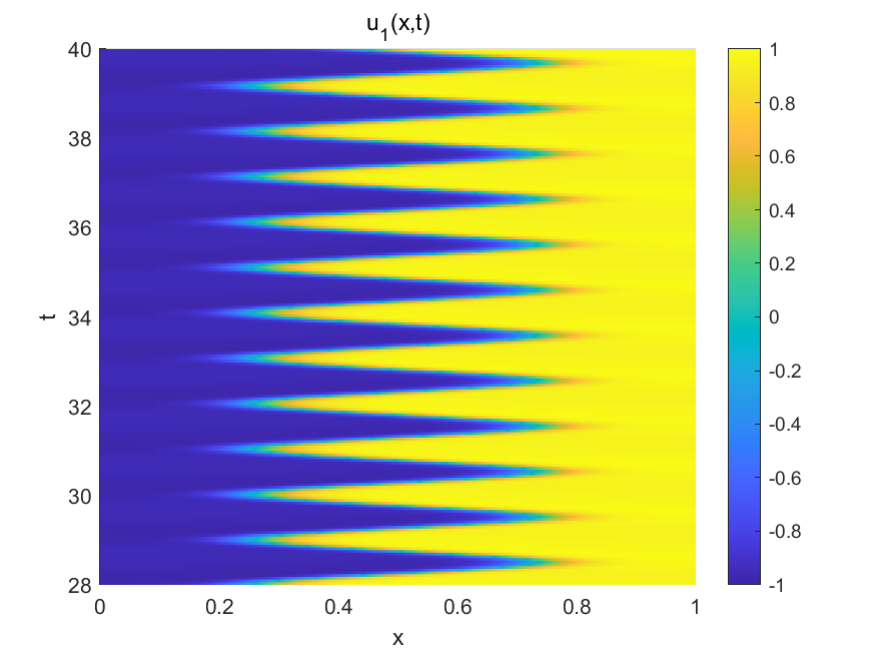}
			\caption{$u_1=u_{1,p}(x,t,\ep)$}	
			\label{Fig.sb-d}
		\end{subfigure}
		\begin{subfigure}[b]{0.3\textwidth}
			\centering
			\includegraphics[width=\linewidth]{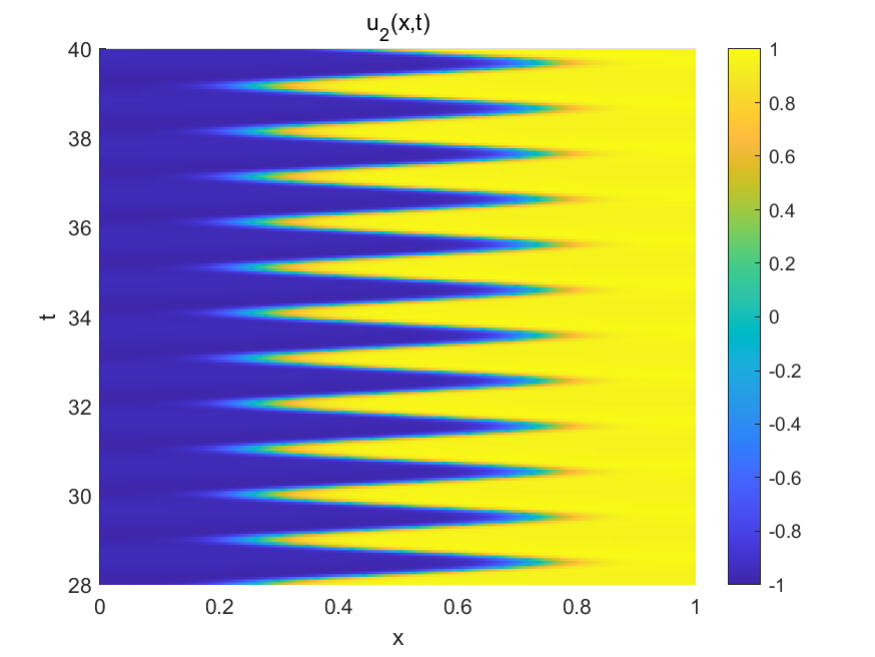}
			\caption{$u_2=u_{2,p}(x,t,\ep)$}
		\end{subfigure}
		\begin{subfigure}[b]{0.3\textwidth}
			\centering
			\includegraphics[width=\linewidth]{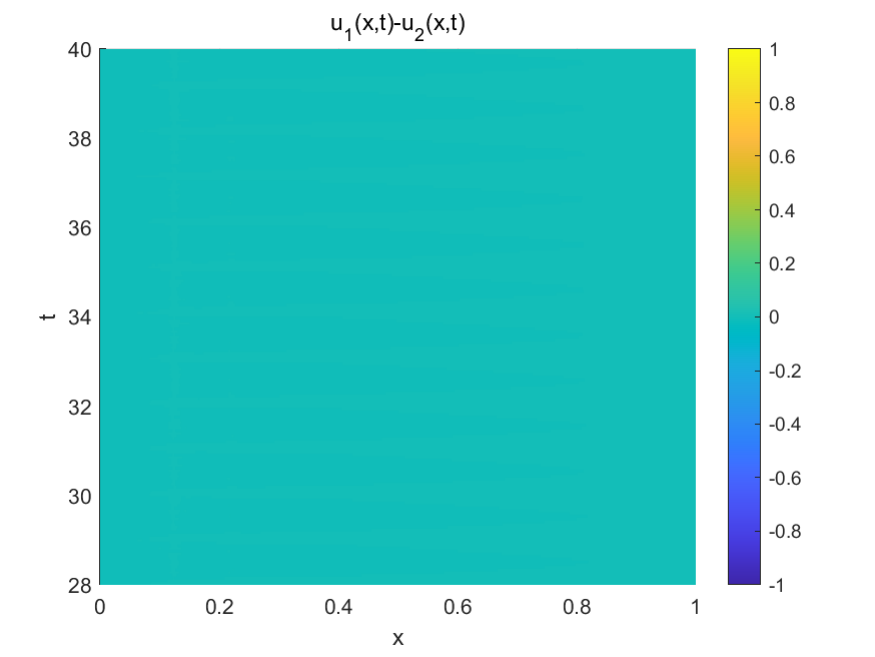}
			\caption{$u_{1,p}(x,t,\ep)-u_{2,p}(x,t,\ep)$}
			\label{Fig.sb-f}
		\end{subfigure}
		\caption{In-phase Hopf bifurcating breathers with respect to system \eqref{BVDP} in Cases 3 and 6, where $(k_1,k_2)=(0.6,2)$.  (a)-(c): Case 3, $\tau=0.04>\tau_H^\ep=0.0319$. (d)-(f): Case 6, $\tau=0.04>\tau_H^\ep=0.0312$.}
		\label{Fig.sb}
	\end{figure}
	
	\begin{figure}
		\centering
		\begin{subfigure}[b]{0.3\textwidth}
			\centering
			\includegraphics[width=\linewidth]{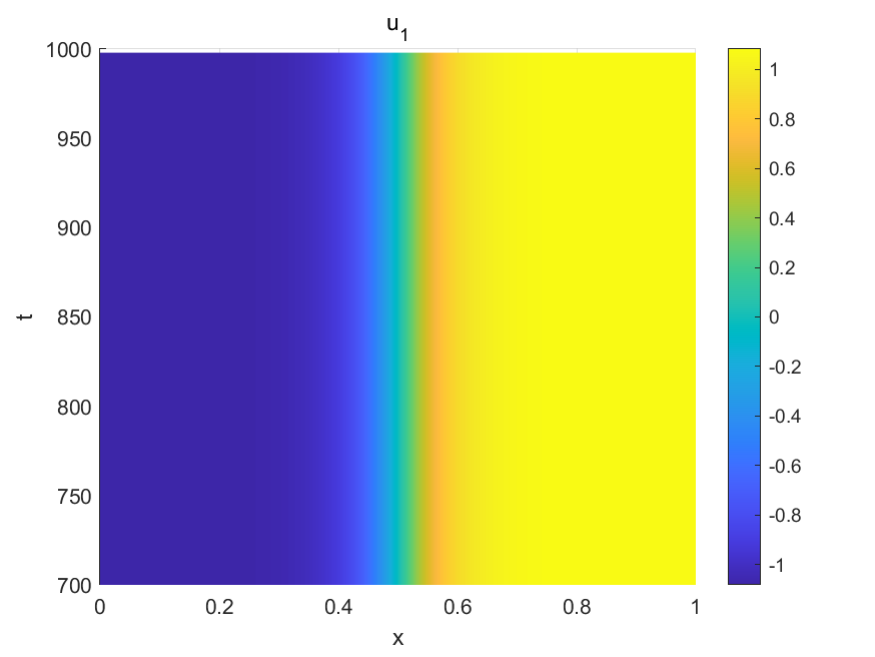}
			\caption{$u_1=u(x,\ep)$}
			\label{Fig.S-a}
		\end{subfigure}
		\begin{subfigure}[b]{0.3\textwidth}
			\centering
			\includegraphics[width=\linewidth]{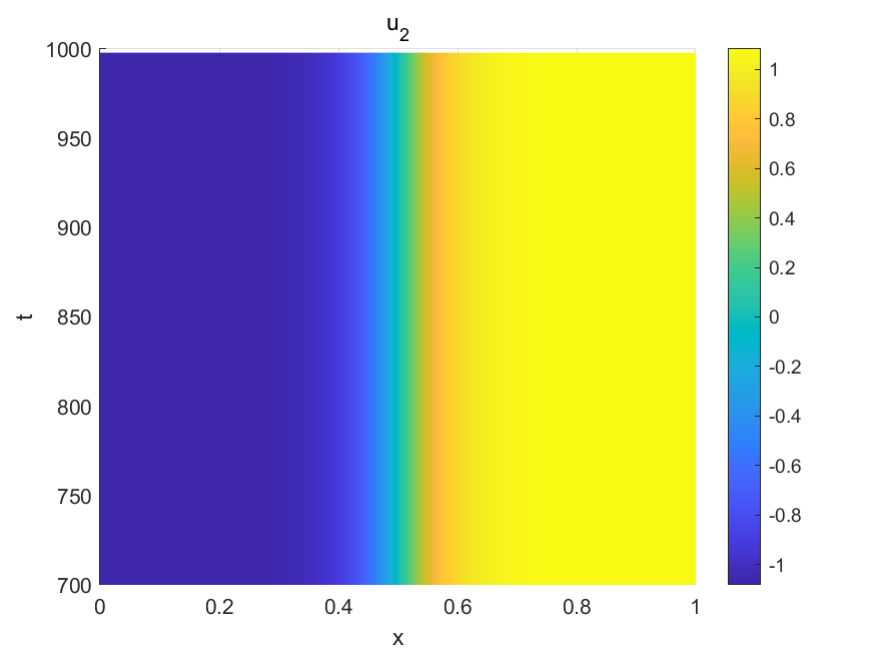}
			\caption{$u_2=u(x,\ep)$}
		\end{subfigure}
		\begin{subfigure}[b]{0.3\textwidth}
			\centering
			\includegraphics[width=\linewidth]{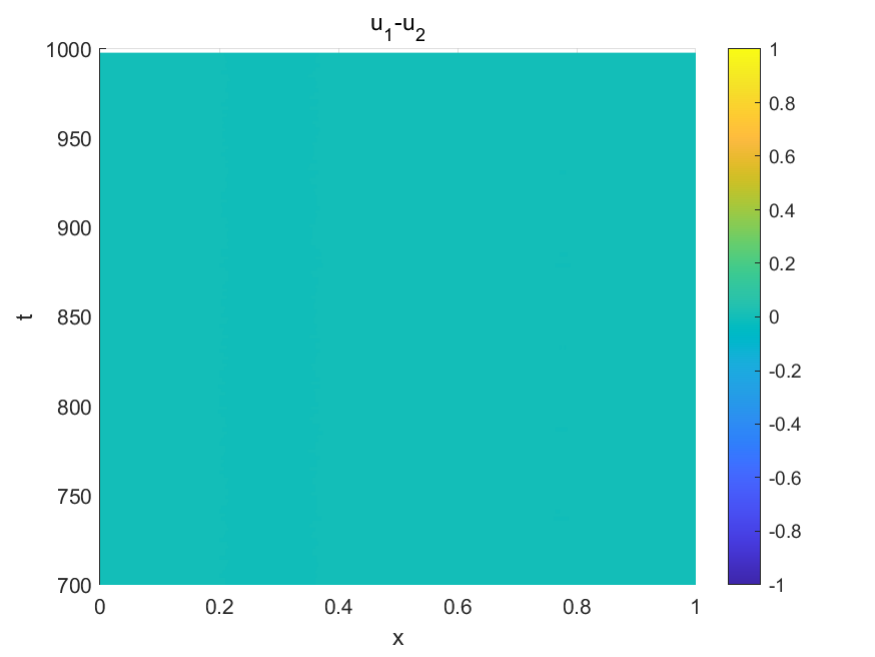}
			\caption{$u_1-u_2$}
			\label{Fig.S-c}
		\end{subfigure}
		\caption{Stable synchronized layered steady state of \eqref{BVDP} in Case 5, where $(k_1,k_2)=(0.6,2)$, $\tau=10$, $\Delta t=1\times 10^{-3}$.}
		\label{Fig.S}
	\end{figure}
	
	Using the parameter values in Table \ref{tab.tau}, we obtain the anti-phase breathers in Cases 1, 2, and 4 (see Figures \ref{Fig.ab-a}--\ref{Fig.ab-i}), the in-phase breathers in Cases 3 and 6 (see Figure \ref{Fig.sb}), and the stable synchronized layered steady state in Case 5 (see Figure \ref{Fig.S}), where the default time step is $\Delta t=2\times 10^{-4}$ unless otherwise stated. The numerical simulations suggest that the Hopf bifurcations shown in Figures \ref{Fig.ab} and \ref{Fig.sb} are supercritical.
	
	\begin{figure}[htbp]	
		\centering
		\begin{subfigure}[b]{0.3\textwidth}
			\centering
			\includegraphics[width=\linewidth]{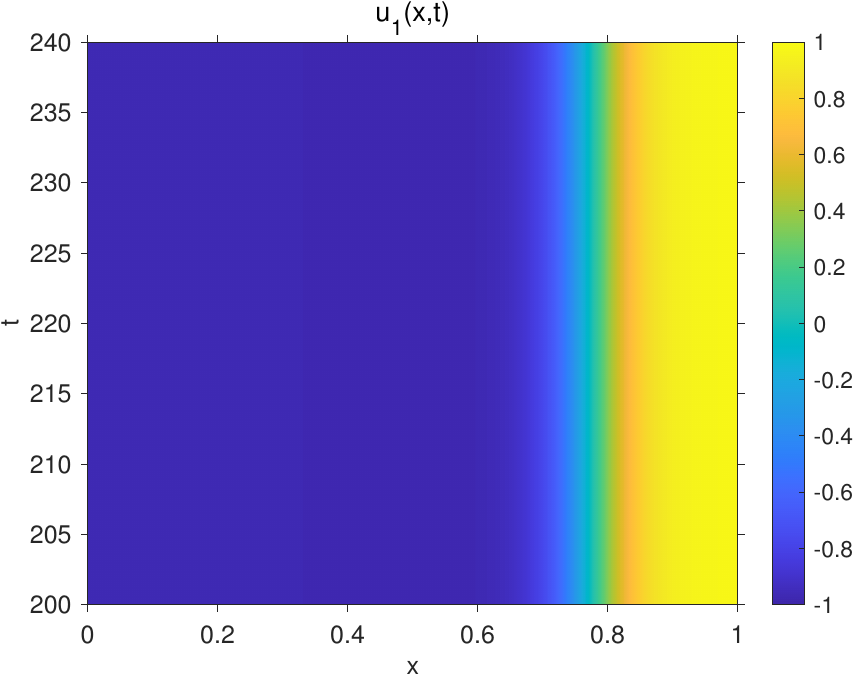}
			\caption{$u_1=\ol u(x)$}	
			\label{Fig.TH-a}
		\end{subfigure}
		\begin{subfigure}[b]{0.3\textwidth}
			\centering
			\includegraphics[width=\linewidth]{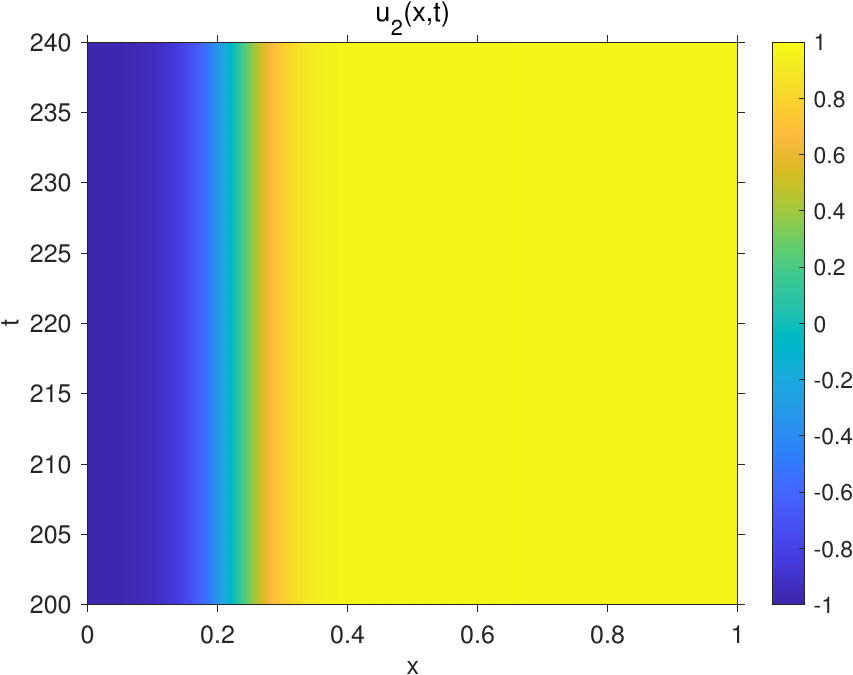}
			\caption{$u_2=\ul u(x)$}	
			\label{Fig.TH-b}
		\end{subfigure}
		\begin{subfigure}[b]{0.3\textwidth}
			\centering
			\includegraphics[width=\linewidth]{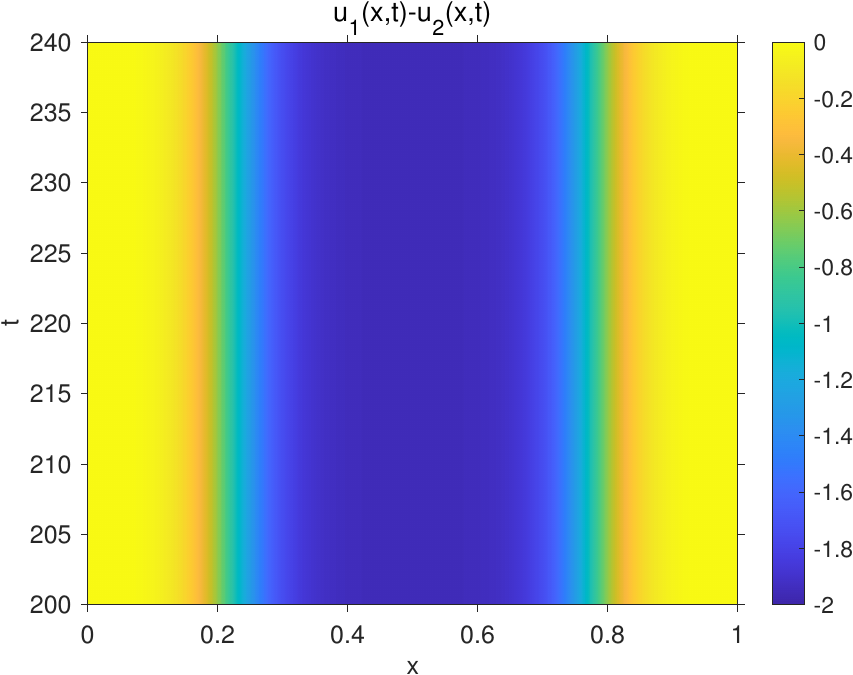}
			\caption{$u_1-u_2$}	
			\label{Fig.TH-c}
		\end{subfigure}
		\centering
		\begin{subfigure}[b]{0.3\textwidth}
			\centering
			\includegraphics[width=\linewidth]{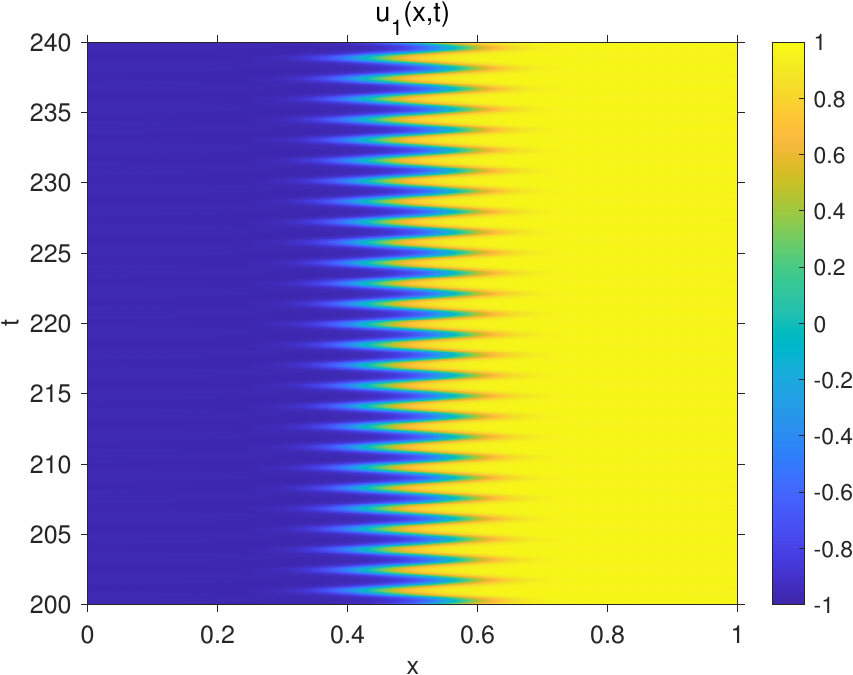}
			\caption{$u_1=u_p(x,t)$}	
			\label{Fig.TH-d}
		\end{subfigure}
		\begin{subfigure}[b]{0.3\textwidth}
			\centering
			\includegraphics[width=\linewidth]{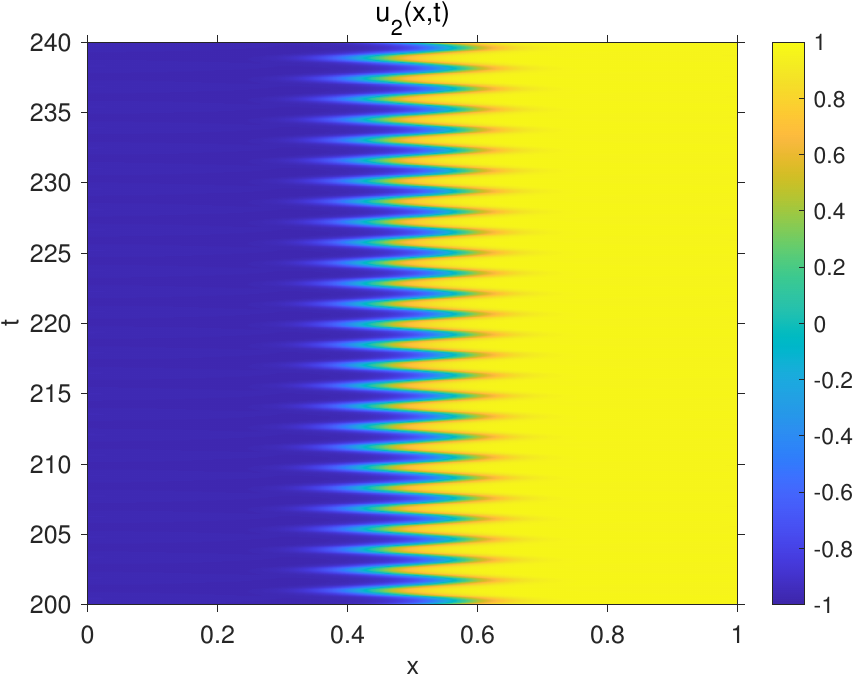}
			\caption{$u_2=u_p(x,t)$}	
			\label{Fig.TH-e}
		\end{subfigure}
		\centering
		\begin{subfigure}[b]{0.3\textwidth}
			\centering
			\includegraphics[width=\linewidth]{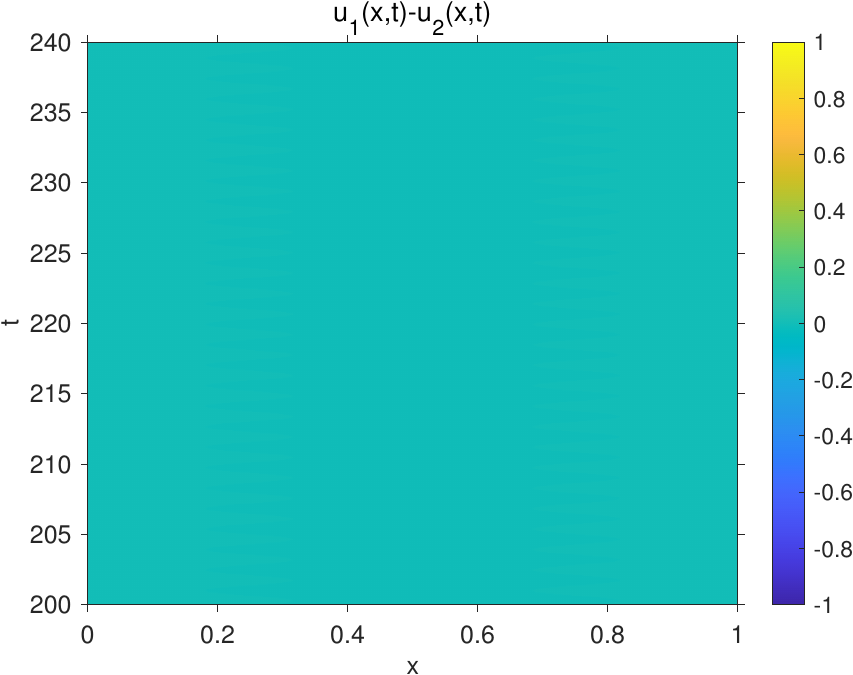}
			\caption{$u_1-u_2$}	
			\label{Fig.TH-f}
		\end{subfigure}
		\caption{Coexistence of stationary and oscillatory patterns of system \eqref{BVDP} in Case 3, where $k_2=5$, $k_1=0.068$, $\tau=0.4>\tau_H(\ep)=0.2843$. (a)--(c): a symmetry-breaking layered steady state arising from Turing instability. (d)--(f): an in-phase breather via Hopf bifurcation.}
		\label{Fig.TH_3}
	\end{figure} 
	
	In particular, we use Case 3 to illustrate the mode switching and interaction between Turing instability and delay-induced oscillations, which we refer to as Turing-Hopf interaction. (1) For small $k_1$, anti-phase breathers may arise through Hopf bifurcation. See Figures \ref{Fig.ab-j}-\ref{Fig.ab-l}, where we set $(k_1,k_2)=(0.03,2)$, $\tau=4>\tau_H^\ep=3.5569$. (2) The coexistence of Turing instability and symmetric Hopf crossings in Cases 3 and 6 (see Figure \ref{Fig.SLEP_dia}) motivates a numerical exploration of their possible interaction. We only explore the dynamics of Case 3 as a numerical example (see Figure \ref{Fig.TH_3}). Setting $k_2=5$ and $k_1=0.068\in(\mathcal{K}_1^{3,s},\xi(k_2))$ for $\mathcal{K}_1^{3,s}=0.0493$ and $\xi(k_2)=0.0751$, $\tau=0.4>\tau_H(\ep)=0.2843$, numerical simulations reveal the coexistence of two types of patterns: a symmetry-breaking layered steady state arising from Turing instability (Figures \ref{Fig.TH-a}--\ref{Fig.TH-c}) and an in-phase breather via  Hopf bifurcation in the symmetric mode (Figures \ref{Fig.TH-d}--\ref{Fig.TH-f}), which are obtained from initial histories with nearly opposite and nearly identical perturbations in the two layers around the synchronized layered steady state, respectively. These observations suggest a nonlinear interaction between stationary symmetry breaking and delay-induced oscillations, and demonstrate the joint effects of time delay and inter-layer coupling on pattern formation.
	
	\section{Concluding remarks}\label{S.CM}
	In this paper, we investigate the stability and bifurcation of the synchronized layered steady state $(U(x,\ep),U(x,\ep))$ in a class of two-layer membrane-coupled RD systems with delayed inter-layer coupling. We consider both discrete and distributed delays, with the delay acting on the cross-layer transmission, the local self-feedback, or both, leading to six representative coupling configurations. Our main objective is to understand how the temporal nature and placement of the delays, together with the inter-layer coupling strengths, affect the stability of the synchronized layered steady state, and in particular how increasing the delay can destabilize an otherwise stable synchronized layered steady state through Hopf bifurcation and generate in-phase or anti-phase breathers. Finally, the analytical results are illustrated and complemented by numerical simulations for a two-layer FitzHugh–Nagumo system.
	
	The main contributions of this work are as follows. (i) We extend the SLEP method to membrane-coupled RD systems with time delays, and develop a unified SLEP framework which incorporates the symmetric and asymmetric sub-problems arising from all six delayed coupling configurations. (ii) Within this framework, we establish sufficient conditions for delay-independent stability, persistence of Turing instability, and delay-induced Hopf bifurcation. In particular, for every fixed $k_2>k_2^*$, all six configurations admit a nonempty interval of $k_1$ immediately above the delay-free Turing threshold 	$\xi(k_2)$, where the synchronized layered steady state is stable in the absence of delay, but a simple conjugate pair of eigenvalues in the asymmetric mode crosses the imaginary axis transversely from left to right as the delay increases. Additional sufficient conditions are also derived to identify the first destabilizing mode and establish the occurrence of primary Hopf bifurcations. (iii) By comparing the six coupling configurations, we identify the common spectral structures and distinct stability and bifurcation mechanisms determined by the temporal nature and placement of delays. These include the generation of in-phase or anti-phase breathers through symmetric or asymmetric Hopf bifurcations, respectively, and stabilization by sufficiently strong activator coupling in the cross-distributed-delayed configuration.
	
	In particular, we would like to make the following concluding remarks.
	
	\begin{enumerate}
		\item One general conclusion emerging from the six coupling configurations is that the temporal nature and the placement of the delay play different but closely related roles in determining the instability mechanism. The symmetric–asymmetric decomposition provides a natural way to see this difference. For the fully delayed couplings (Cases 1 and 4), the symmetric mode is not affected by the delay, and the Hopf instability can only be generated by the asymmetric mode, leading to anti-phase breathers. For the self-delayed couplings (Cases 3 and 6), in contrast, the symmetric mode becomes  important and may dominate the primary Hopf bifurcation, especially for strong activator coupling, thereby tending to trigger in-phase breathers. The cross-delayed couplings reveal a clearer difference between discrete and distributed delays: although Cases 2 and 5 have the same delay placement, their stability properties can be significantly different. In this sense, the placement of the delay largely determines which mode plays the dominant role, while the temporal nature of the delay influences whether and when that mode destabilizes the synchronized layered steady state.
		\item The strong-coupling-induced stability in Case 5 suggests that, under certain coupling architectures, strengthening the activator coupling can suppress (rather than enhance) delay-induced oscillatory instability of the synchronized layered steady state. This may provide a useful design principle for coupled reactors, synthetic biological circuits or neural networks with cross-delayed distributed coupling. An intuitive explanation of this phenomenon can be obtained directly from the coupling structure. Recall that the activator coupling terms are
		\begin{equation*}
			T_1:=k_1\left( L_\tau[u_2]-u_1\right),\;\; 
			T_2:=k_1\left( L_\tau[u_1]-u_2\right),\;\;
			L_\tau[u]:=\int_{-\infty}^{t}\frac{1}{\tau}e^{-\frac{t-s}{\tau}}u(x,s)ds.
		\end{equation*}
		Suppose that a small perturbation $\eta_1$ is introduced into $u_1$ in the first layer. When $k_1$ is large, a strong restoring effect in $T_1$ is produced by the instantaneous term $-k_1\eta_1$, which acts immediately against this perturbation. Meanwhile, the perturbation is transmitted to $T_2$ through $k_1 L_{\tau}[\eta_1]$ but is dispersed by the averaging effect of $L_\tau$ over the whole past history, such that this transmitted perturbation is spread out in time before it acts on the second layer. Moreover, any perturbation $\eta_2$ in $u_2$ (caused by $u_1$) is again subject to the strong instantaneous restoring term $-k_1\eta_2$ in $T_2$. When the perturbation is transmitted back to the first layer, it is once more averaged by the distributed delay. Thus, for sufficiently strong activator coupling, the local restoring effects remain instantaneous and strong, whereas the cross-layer propagation of perturbations is repeatedly dispersed by the memory kernel. This makes it difficult for small perturbations to accumulate and may explain the delay-independent stability obtained in Case 5.  
		
		The situation is quite different in Case 2, because the cross-layer signal is transmitted through a discrete delay, which shifts a perturbation in time but does not average it over a range of past states. Hence, although the instantaneous terms $-k_1u_i$ still provide strong restoring effects, the perturbation can be preserved without being averaged over its past history, such that it may accumulate and eventually lead to  instability of $(U(x,\ep),U(x,\ep))$. This helps explain why strong activator coupling does not cause the same stability-locking effect in Case 2. Similar arguments help explain the differences between Cases 4--6. In Case 4, both the cross-layer transmission and the self-feedback are distributed in time, so the restoring response itself is also delayed and dispersed. In Case 6, the perturbation in cross-layer transmission is instantaneous, while the stabilizing self-feedback is dispersed. These differences may also help explain the different robustness against delay-induced destabilization observed among the fully, cross-, and self-delayed coupling structures.
		\item Higher-codimensional bifurcations. The present analysis is mainly restricted to primary Hopf bifurcations, while higher-codimensional bifurcations are left open. However, the numerical results in Case 3 reveal the coexistence of symmetry-breaking layered steady states and in-phase breathers associated with Turing and Hopf instabilities, respectively, and suggest the initial-history-dependent selection between them. In the singular limit $\ep\to0^+$, we have shown in Figure \ref{Fig.SLEP_dia} that, for some $(k_1,k_2)\in\Gamma$ and at critical delay $\tau=\tau_H^*$, the asymmetric sub-problem has a zero eigenvalue, while the symmetric sub-problem has a pair of purely imaginary eigenvalues. Thus, it is natural to further investigate whether a Turing-Hopf point persists for system \eqref{eq-o} when $\ep>0$ is sufficiently small. On the other hand, the present theory does not rule out the possibility of a Hopf–Hopf bifurcation caused by the simultaneous criticality of the symmetric and asymmetric sub-problems. These analyses would be important for understanding pattern selection and mode competition in the multi-layer systems with delayed couplings.
		\item Effects of the Gamma shape parameter. The present paper considers the exponential weak kernel $g_1(t)$ in \eqref{gn}. A natural and challenging extension is to investigate higher-order Gamma kernels $g_n(t)$ ($n=2,3,\ldots$). Recently, Shi's survey \cite{S} provides a systematic framework for reaction--diffusion models with Gamma-type distributed memory, including the comparisons between weak and strong kernels. Thus, we may determine how $n$ changes the size and geometry of the Hopf bifurcation regions and the delay-independent stability regions in the $(k_1,k_2)$ plane, and how it affects the primary destabilizing mode. Finally, the effects of highly dispersed Gamma delays studied by Omi and Shinomoto \cite{OS} also motivate further analysis of the regime $0<n<1$.
		\item Delayed reaction kinetics and further extensions of the SLEP method. A related issue was studied by Fadai, Ward and Wei \cite{FWW} for a singularly perturbed Gierer--Meinhardt system with a discrete delay entering the nonlinear reaction kinetics. Using the nonlocal eigenvalue problem (NLEP) method, they investigated the stability and delay-induced Hopf bifurcation of spike solutions. Their work shows that singular spectral reduction can still be effectively used when the delay enters directly into the reaction terms. However, the present SLEP framework does not directly cover the corresponding stability and bifurcation problems for layered steady states when delays are introduced into the nonlinear reaction terms. Extending the SLEP method to such problems would require further development of the singular-limit spectral theory. This may provide an interesting direction for future study.
	\end{enumerate} 
	
	\section{Appendix}
	\begin{lemma}\label{le.AA1} 
		Suppose that $0<\ep<\ep_0$, and that $(u(x,\ep),v(x,\ep))$ and $(u^\star(x),v^\star(x))$ are the steady state solutions of \eqref{eq-s} and \eqref{eq.reduced}, respectively. Define
		\begin{equation*}
			\begin{split}
				f_u^\ep:=f_u(u(x,\ep),v(x,\ep)),\;f_v^\ep:=f_v(u(x,\ep),v(x,\ep)),\\
				g_u^\ep:=g_u(u(x,\ep),v(x,\ep)),\;g_v^\ep:=g_v(u(x,\ep),v(x,\ep)).
			\end{split}
		\end{equation*}
		Then, the eigenvalue problem
		\begin{equation*}
			\begin{cases}
				\begin{split}
					\ep^2\dfrac{d^2\phi}{dx^2}+f_u^\ep\phi&=\mu\phi,\;x\in(0,1),\\
					\phi'(0)=\phi'(1)&=0,
				\end{split}
			\end{cases}
		\end{equation*}
		admits a complete orthonormal set of eigenpairs $\{(\mu_n^{\ep}, \varphi_n^{\ep}(x))\}$ in $L^2(0,1)$ sense with $\Vert\varphi_n^{\ep}\Vert_{L^2}=1$, where $n\in\mathbb{N}_0:=\{0,1,\cdots\}$. Moreover,
		\begin{enumerate}[label=$\arabic*$.]
			\item There exists a $\mu_*<0$ independent of $\ep$, such that $\mu_0^{\ep}>0>\mu_*>\mu_1^{\ep}>\cdots$.
			In particular, $\mu_0^{\ep}=\rho(\ep)\ep$, where $\rho(\ep)$ is a positive continuous function up to $\ep\to 0^+$ and
			\begin{equation*}
				\rho_0^{*}=\lim_{\ep\to 0^+}\rho(\ep)>0.
			\end{equation*}
			\item Let $F(u, v)$ be a smooth function of $u$ and $v$. For any $ h\in L^2(0,1) \cap L^{\infty}(0,1)$ and $p(\la)\in\C$, define
			\begin{equation}\label{eq.dag}
				(\ep^2\dfrac{d^2}{dx^2}+f_u^\ep+\ep p(\la) -\ep\alpha\la)^{\dag} h:=\sum_{n\geq1}\dfrac{\langle  h, \varphi_n^\ep\rangle}{\mu_n^\ep+\ep p(\la)-\ep\alpha\la}\varphi_n^\ep.
			\end{equation}  
			Then, for any $\alpha\in \R$ and $\la\in\C$, as $\ep\to 0^+$, 
			\begin{equation*}
				(\ep^2\dfrac{d^2}{dx^2}+f_u^\ep+\ep p(\la)- \ep\alpha\la)^{\dag}(F(u(x,\ep), v(x,\ep)) h) \to \ds\frac{F(u^\star(x), v^\star(x))h}{f_u(u^\star(x), v^\star(x))}, 
			\end{equation*}
			in strongly $L^2$-sense.
			\item Define $\wt\varphi_0^\ep(y):=\varphi_0^\ep(x^*+\ep y)$, then as $\ep\to0^+$, it holds that
			\begin{equation*}\label{tends1}
				\sqrt{\ep}\wt\varphi_0^\ep(y)\to \kappa^* \frac{d}{dy}\wt u^*(y)
			\end{equation*}
			uniformly for any compact subset of $\R$, where $\kappa^*=\Vert \frac{d}{dy}\wt u^*(y)\Vert_{L^2(\R)}^{-1}$ and $\wt u^*(y)$ is the solution of
			\begin{equation*}
				\begin{cases}
					\begin{split}
						&\frac{d^2 u}{dy^2}+f(u,v^*)=0, && y\in(-\infty,\infty),\\
						&u(\pm\infty)=h_\pm(v^*),\\
						&u(0)=h_0(v^*).
					\end{split}
				\end{cases}
			\end{equation*}
			\item Let $\delta(x-x^*)$ be the Dirac's delta function at $x = x^*$, where $x^*\in(0,1)$ is the layer position of $(u^\star(x), v^\star(x))$ stated in Theorem \ref{th.1}. Then, in the $H^{-1}$ sense, it holds true that
			\begin{equation}\label{fvgu}
				\begin{split}
					\lim\limits_{\ep\to 0^+}\ds\frac{-f_v^{\ep}\varphi_0^\ep}{\sqrt{\ep}}=c_1^*\delta(x - x^*),\;
					\lim\limits_{\ep\to 0^+}\ds\frac{g_u^{\ep}\varphi_0^\ep}{\sqrt{\ep}}=c_2^*\delta(x - x^*),
				\end{split}
			\end{equation}
			where 
			\begin{equation}\label{c*}
				\begin{split}
					c_1^*=-\kappa^*\frac{dJ(v^*)}{dv}>0,\; c_2^*=\kappa^*\left\lbrace g(h_+(v^*),v^*)-g(h_-(v^*),v^*)\right\rbrace >0,
				\end{split}
			\end{equation}	
		\end{enumerate}
		
		Finally, there exists a complete orthonormal set of eigenpairs $\{(\gamma_n, \psi_n)\}_{n=0}^{\infty}$ for the eigenvalue problem
		\begin{equation}\label{gamman}
			\begin{cases}
				\begin{split}
					&-d\frac{d^2}{dx^2}\psi_n+\frac{ f_u^\star g_v^\star-f_v^\star g_u^\star}{- f_u^\star}\psi_n=\gamma_n\psi_n,\;x\in(0,1),\\
					&\psi_n'(0)=\psi_n'(1)=0.
				\end{split}
			\end{cases}		
		\end{equation}
		in the $L^2(0,1)$ sense, satisfying $\gamma_n>0$ for all $n\ge0$, where the partial derivatives are evaluated at $(u^\star(x),v^\star(x))$.
	\end{lemma}
	\begin{proof}
		For detailed proofs of Lemma \ref{le.AA1}, we refer to Lemma 2.1, Lemma 2.2, Lemma 2.3 of \cite{NM}, or Lemma 2.2 and Lemma 2.3 of \cite{NF}.
	\end{proof}
	
	\vskip1cm
	
	{\small
		\bibliographystyle{plain}

\begin{thebibliography}{MM}
			
			\bibitem{Atay}\href{https://doi.org/10.1103/PhysRevLett.91.094101}{F. Atay, Distributed delays facilitate amplitude death of coupled oscillators, Phys. Rev. Lett. 91 (9) (2003) 094101.}
			
			\bibitem{AEKV}\href{https://doi.org/10.1007/s00332-024-10033-7}{N. Awal, I. Epstein, T. Kaper, T. Vo, Strong symmetry breaking in coupled, identical Lengyel--Epstein oscillators via folded singularities, J. Nonlinear Sci. 34  (2024) 53.}
			
			\bibitem{Bar}
			\href{https://doi.org/10.1021/j150660a048}
			{K. Bar-Eli, Coupling of chemical oscillators, J. Phys. Chem. 88 (16) (1984) 3616--3622.}
			
			\bibitem{CDBD}\href{https://doi.org/10.1103/PhysRevLett.64.2953}{V. Castets, E. Dulos, J. Boissonade, P. De Kepper, Experimental evidence of a sustained standing Turing-type nonequilibrium chemical pattern, Phys. Rev. Lett. 64 (24) (1990) 2953--2956.}
			
			\bibitem{CZL}
			\href{https://doi.org/10.1016/j.jfranklin.2015.02.022}
			{S. Cai, P. Zhou, Z. Liu, Synchronization analysis of hybrid-coupled delayed dynamical networks with impulsive effects: A unified synchronization criterion, J. Franklin Inst. 352 (5) (2015) 2065--2089.}
			
			\bibitem{FWW}
			\href{https://doi.org/10.1137/16M1063460}
			{N. Fadai, M. Ward, J. Wei, Delayed reaction kinetics and the stability of spikes in the Gierer--Meinhardt model, SIAM J. Appl. Math. 77 (2) (2017) 664--696.}
			
			\bibitem{JL}
			\href{https://doi.org/10.1063/1.2041427}
			{L. Ji, Q. Li, Turing pattern formation in coupled reaction-diffusion system with distributed delays, J. Chem. Phys. 123 (9) (2005) 094509.}
			
			\bibitem{KS}
			\href{https://doi.org/10.1016/0022-5193(80)90352-5}
			{M. Kawato, R. Suzuki, Two coupled neural oscillators as a model of the circadian pacemaker, J. Theor. Biol. 86 (3) (1980) 547--575.}
			
			\bibitem{LE1}
			\href{https://doi.org/10.1126/science.251.4994.650}
			{I. Lengyel, I. Epstein, Modeling of Turing structures in the chlorite--iodide--malonic acid--starch reaction system, Science 251 (4994) (1991) 650--652.}
			
			\bibitem{LE2}
			\href{https://doi.org/10.1063/1.165819}
			{I. Lengyel, I. Epstein, Diffusion-induced instability in chemically reacting systems: Steady-state multiplicity, oscillation, and chaos, Chaos 1 (1) (1991) 69--76.}
			
			\bibitem{LWO}
			\href{https://doi.org/10.1063/1.4875262}
			{J. Li, H. Wang, Q. Ouyang, Square Turing patterns in reaction-diffusion systems with coupled layers, Chaos 24 (2) (2014) 023115.}
			
			\bibitem{LYGYW}
			\href{https://doi.org/10.1016/j.physd.2007.09.014}
			{P. Liu, F. Yi, Q. Guo, J. Yang, W. Wu, Analysis on global exponential robust stability of reaction-diffusion neural networks with S-type distributed delays, Physica D 237 (4) (2008) 475--485.}
			
			\bibitem{Ma}
			\href{https://doi.org/10.1016/0022-0396(79)90156-6}
			{K. Maginu, Stability of spatially homogeneous periodic solutions of reaction-diffusion equations, J. Differ. Equ. 31 (1) (1979) 130--138.}
			
			\bibitem{MTH}
			\href{https://doi.org/10.1137/0511057}
			{M. Mimura, M. Tabata, Y. Hosono, Multiple solutions of two-point boundary value problems of Neumann type with a small parameter, SIAM J. Math. Anal. 11 (4) (1980) 613--631.}
			
			\bibitem{NF}
			\href{https://doi.org/10.1137/0518124}
			{Y. Nishiura, H. Fujii, Stability of singularly perturbed solutions to systems of reaction-diffusion equations, SIAM J. Math. Anal. 18 (6) (1987) 1726--1770.}
			
			\bibitem{NM}
			\href{https://doi.org/10.1137/0149029}
			{Y. Nishiura, M. Mimura, Layer oscillations in reaction-diffusion systems, SIAM J. Appl. Math. 49 (2) (1989) 481--514.}
			
			\bibitem{NT}
			\href{https://doi.org/10.1090/S0002-9947-05-04010-9}
			{W. Ni, M. Tang, Turing patterns in the Lengyel--Epstein system for the CIMA reaction, Trans. Amer. Math. Soc. 357 (10) (2005) 3953--3969.}
			
			\bibitem{OS}
			\href{https://doi.org/10.1103/PhysRevE.77.046214}
			{T. Omi, S. Shinomoto, Can distributed delays perfectly stabilize dynamical networks?, Phys. Rev. E 77 (4) (2008) 046214.}
			
			\bibitem{PPR}
			\href{https://doi.org/10.1103/PhysRevE.101.052203}
			{K. Pal, S. Paul, D. Ray, Spatiotemporal antiresonance in coupled reaction-diffusion systems, Phys. Rev. E 101 (5) (2020) 052203.}
			
			\bibitem{RF}
			\href{https://doi.org/10.1016/j.physd.2003.12.004}
			{S. Ruan, R. Filfil, Dynamics of a two-neuron system with discrete and distributed delays, Physica D 191 (3--4) (2004) 323--342.}
			
			\bibitem{RSJ}\href{https://doi.org/10.1103/PhysRevLett.80.5109}{D. Ramana Reddy, A. Sen, G. Johnston, Time delay induced death in coupled limit cycle oscillators, Phys. Rev. Lett. 80 (23) (1998) 5109--5112.}
			
			
			\bibitem{S}
			\href{https://doi.org/10.3934/dcds.2026134}
			{J. Shi, Bifurcation and pattern formation in reaction-diffusion models with nonlocal advection and time delays, Discrete Contin. Dyn. Syst. 57 (2026) 496--533.}
			
			\bibitem{SSS}
			\href{https://doi.org/10.1016/j.jde.2017.07.024}
			{Q. Shi, J. Shi, Y. Song, Hopf bifurcation in a reaction-diffusion equation with distributed delay and Dirichlet boundary condition, J. Differ. Equ. 263 (10) (2017) 6537--6575.}
			
			\bibitem{ST}
			\href{https://doi.org/10.1137/10080885X}
			{C.-W. Shih, J.-P. Tseng, Global synchronization and asymptotic phases for a ring of identical cells with delayed coupling, SIAM J. Math. Anal. 43 (4) (2011) 1667--1697.}
			
			\bibitem{SWS}
			\href{https://doi.org/10.1007/s10884-012-9268-z}
			{Y. Su, J. Wei, J. Shi, Hopf bifurcation in a diffusive logistic equation with mixed delayed and instantaneous density dependence, J. Dynam. Differential Equations 24 (4) (2012) 897--925.}
			
			\bibitem{SY0}
			\href{https://doi.org/10.1137/23M159055X}
			{G. She, F. Yi, Dynamics and bifurcations in a nondegenerate homogeneous diffusive SIR rabies model, SIAM J. Appl. Math. 84 (2) (2024) 632--660.}
			
			\bibitem{SY1}
			\href{https://doi.org/10.1007/s00332-024-10081-z}
			{G. She, F. Yi, Stability and bifurcation analysis of a reaction-diffusion SIRS epidemic model with the general saturated incidence rate, J. Nonlinear Sci. 34 (2024) 101.}
			
			\bibitem{SY2}
			\href{https://doi.org/10.1016/j.jde.2024.01.003}
			{Q. Song, F. Yi, Spatiotemporal patterns and bifurcations of a delayed diffusive predator-prey system with fear effects, J. Differ. Equ. 388 (2024) 151--187.}
			
			\bibitem{SY3}
			\href{https://doi.org/10.1137/24M1687893}
			{Q. Song, F. Yi, Turing instability of the periodic solutions for the reaction-diffusion system involving nonlocal delays, SIAM J. Appl. Dyn. Syst. 24 (4) (2025) 2767--2803.}
			
			\bibitem{Tu}
			\href{https://doi.org/10.1098/rstb.1952.0012}
			{A. Turing, The chemical basis of morphogenesis, Philos. Trans. R. Soc. Lond. B 237 (1952) 37--72.}
			
			\bibitem{TMN}
			\href{https://doi.org/10.1007/BF03167236}
			{T. Takaishi, M. Mimura, Y. Nishiura, Pattern formation in coupled reaction-diffusion systems, Japan J. Indust. Appl. Math. 12 (3) (1995) 385--424.}
			
			\bibitem{WY1}
			\href{https://doi.org/10.1016/j.jde.2022.02.031}
			{M. Wang, F. Yi, On the dynamics of the diffusive Field--Noyes model for the Belousov--Zhabotinskii reaction, J. Differ. Equ. 318 (2022) 443--479.}
			
			\bibitem{WY2}
			\href{https://doi.org/10.1007/s00208-026-03483-9}
			{Q. Wu, F. Yi, Turing instability of synchronized layered steady states and ``breathing'' periodic solutions in membrane-coupled reaction-diffusion systems, Math. Ann. 395 (3) (2026) 58.}
			
			\bibitem{Yi}
			\href{https://doi.org/10.1016/j.jde.2021.02.006}
			{F. Yi, Turing instability of the periodic solutions for reaction-diffusion systems with cross-diffusion and the patch model with cross-diffusion-like coupling, J. Differ. Equ. 281 (2021) 379--410.}
			
			\bibitem{YDZE}
			\href{https://doi.org/10.1103/PhysRevLett.88.208303}
			{L. Yang, M. Dolnik, A. Zhabotinsky, I. Epstein, Spatial resonances and superposition patterns in a reaction-diffusion model with interacting Turing modes, Phys. Rev. Lett. 88 (20) (2002) 208303.}
			
			\bibitem{YE}
			\href{https://doi.org/10.1103/PhysRevLett.90.178303}
			{L. Yang, I. Epstein, Oscillatory Turing patterns in reaction-diffusion systems with two coupled layers, Phys. Rev. Lett. 90 (17) (2003) 178303.}
			
			\bibitem{ZLLX}
			\href{https://doi.org/10.1109/TNNLS.2022.3224393}
			{X. Zhang, C. Li, H. Li, J. Xu, Synchronization of neural networks involving distributed-delay coupling: A distributed-delay differential inequalities approach, IEEE Trans. Neural Netw. Learn. Syst. 35 (6) (2024) 8086--8096.}
			
			\bibitem{ZSKK}
			\href{https://doi.org/10.1063/1.4984927}
			{W. Zou, M. \v{S}ebek, I. Z. Kiss, J. Kurths, Revival of oscillations from deaths in diffusively coupled nonlinear systems: Theory and experiment, Chaos 27 (6) (2017) 061101.}
			
	\end{thebibliography}
		}
	
\end{document}